\documentclass{QT}

\usepackage{amscd}
\usepackage{verbatim}
\usepackage{mathrsfs}
\usepackage{lscape}
\usepackage{stmaryrd}
\usepackage[all,knot]{xy}
\xyoption{arc}

\address[davidero@usc.edu]{David E. V. Rose, Department of Mathematics, University of Southern California, 
Los Angeles, CA 90089, USA}

\numberwithin{equation}{section}
\newtheorem{thm}{Theorem}[section]
\newtheorem{cor}[thm]{Corollary}
\newtheorem{lem}[thm]{Lemma}
\newtheorem{prop}[thm]{Proposition}

\theoremstyle{definition}
\newtheorem{ex}[thm]{Example}
\newtheorem{defn}[thm]{Definition}

\newtheorem*{nnthm}{Theorem}

\newcommand{\abs}[1]{\left\vert#1\right\vert}

\newcommand{\R}{\mathbb R}
\newcommand{\C}{\mathbb C}

\newcommand{\Z}{\mathbb Z}

\newcommand{\Hom}{\operatorname{Hom}}
\newcommand{\cone}{\operatorname{cone}}
\newcommand{\Oh}{\operatorname{O}}
\newcommand{\Rep}{\operatorname{Rep}}
\newcommand{\wt}{\operatorname{wt}}
\newcommand{\Ch}{\operatorname{Kom}}
\newcommand{\K}{\operatorname{K}}
\newcommand{\Tot}{\operatorname{Tot}}
\newcommand{\supp}{\operatorname{supp}}
\newcommand{\Id}{\operatorname{id}}

\newcommand{\diskv}{
\xy 0;/r.1pc/:
(0,10)*\xycircle(10,3){-};
(10,10); (-10,10) **\crv{(10,-15)&(-10,-15)};
\endxy
}

\newcommand{\ctv}{
\xy 0;/r.1pc/:
(3,10)*{
\xy
(-3,-10)*\xycircle(10,3){-};
(7,-10); (-13,-10) **\crv{(0,-20)&(15,-40)&(0,-70)&(-20,-40)&(-5,-20)};
(-1,-38); (-2,-50) **\crv{(-4,-43)}\POS?(.2)="a" \POS?(.85)="b";
"a"; "b" **\crv{(1,-45)};
(-2.5,-23.5)*\xycircle(7.6,-3){-};
(0,-23.7); (4,-28) **\dir{-};
(-4.2,-23.6); (.85,-29.1) **\dir{-};
(-8,-24.4); (-3.15,-29.5) **\dir{-};
\endxy};
\endxy
}

\newcommand{\tv}{
\xy 0;/r.1pc/:
(3,10)*{
\xy
(-3,-10)*\xycircle(10,3){-};
(7,-10); (-13,-10) **\crv{(0,-20)&(15,-40)&(0,-70)&(-20,-40)&(-5,-20)};
(-1,-38); (-2,-50) **\crv{(-4,-43)}\POS?(.2)="a" \POS?(.85)="b";
"a"; "b" **\crv{(1,-45)};
\endxy};
\endxy
}

\newcommand{\smallproj}{
\xy 0;/r.15pc/:
(-6,0); (-1,0) **\dir{-};
(-1,-4)*{}; (-1,4)*{} **\dir{-};
(-1,-4)*{}; (1,-4)*{} **\dir{-};
(-1,4)*{}; (1,4)*{} **\dir{-};
(1,-4)*{}; (1,4)*{} **\dir{-};
(1,0); (6,0) **\dir{-};
\endxy
}

\newcommand{\smooth}{
\xy
(-2,2)*{}="a";
(2,2)*{}="b";
(-2,-2)*{}="c";
(2,-2)*{}="d";
"a";"c" **\crv{-};
"b";"d" **\crv{-};
\endxy
}

\newcommand{\singular}{
\xy
(-2,2)*{}="a";
(0,1)*{}="b";
(2,2)*{}="c";
(0,-1)*{}="d";
(-2,-2)*{}="e";
(2,-2)*{}="f";
"a";"b" **\dir{-};
"b";"c" **\dir{-};
"b";"d" **\dir{-};
"e";"d" **\dir{-};
"d";"f" **\dir{-};
\endxy
}

\newcommand{\id}[1]{\ \xy
(3,0)*{_{#1}};
{\ar (-3,0)*{};(1,0)*{}};
\endxy}

\newcommand{\idmn}[2]{\ \xy
{\ar (-3,2)*{};(1,2)*{}};
{\ar (1,-2)*{};(-3,-2)*{}};
(3,2)*{_{#1}};
(3,-2)*{_{#2}};
\endxy}

\newcommand{\idtwo}{
\xy
{\ar@/_.2pc/(-5,3)*{};(5,3)*{}};
{\ar@/^.2pc/(-5,-3)*{};(5,-3)*{}}; 
\endxy
}

\newcommand{\twistw}[1]{\xy
\htwist~{(-10,3)}{(0,3)}{(-10,-3)}{(0,-3)};
\htwist~{(0,3)}{(10,3)}{(0,-3)}{(10,-3)};
(0,1)*{\vdots};
(13,0)*{_{w}};
(0,5)*{_{#1}};
\endxy}

\newcommand{\smalltwohalftwist}{\xy 0;/r.1pc/:
\htwist~{(0,5)}{(10,5)}{(0,-5)}{(10,-5)};
{\ar (10,5)*{};(13,5)*{}};
{\ar (10,-5)*{};(13,-5)*{}};
\endxy}

\newcommand{\smalltwistm}[2]{\xy 0;/r.1pc/:
\htwist~{(-10,5)}{(0,5)}{(-10,-5)}{(0,-5)};
\htwist~{(0,5)}{(10,5)}{(0,-5)}{(10,-5)};
(0,2.5)*{\vdots};
{\ar (10,5)*{};(14,5)*{}};
{\ar (10,-5)*{};(14,-5)*{}};
(15,0)*{_{#1}};
(0,7)*{^{#2}};
\endxy}

\newcommand{\twistm}[2]{\xy
\htwist~{(-10,3)}{(0,3)}{(-10,-3)}{(0,-3)};
\htwist~{(0,3)}{(10,3)}{(0,-3)}{(10,-3)};
(0,1)*{\vdots};
{\ar (10,3)*{};(11,3)*{}};
{\ar (10,-3)*{};(11,-3)*{}};
(13,0)*{_{#1}};
(0,7)*{_{#2}};
\endxy}

\newcommand{\twisttwo}[1]{
\xy
{\ar (-5,5)*{}; (6,5)*{}};
{\ar (-5,-5)*{}; (6,-5)*{}};
\htwist~{(-5,2)}{(0,2)}{(-5,-2)}{(0,-2)};
\htwist~{(0,2)}{(5,2)}{(0,-2)}{(5,-2)};
{\ar (5,2)*{};(6,2)*{}};
{\ar (5,-2)*{};(6,-2)*{}};
(0,3)*{_{#1}};
\endxy
}

\newcommand{\twistds}[2]{
\xy
\htwist~{(-10,5)}{(0,5)}{(-10,-5)}{(0,-5)};
\htwist~{(0,5)}{(10,5)}{(0,-5)}{(10,-5)};
(0,1)*{\vdots};
(-5,0)="h1";
(5,0)="h2";
(-10,1)*{};"h1"+(-1,1) **\dir{-};
(-10,-1)*{};"h1"+(-1,-1) **\dir{-};
{\ar "h2"+(1,1);(11,1)};
{\ar "h2"+(1,-1);(11,-1)};
{\ar (10,5)*{};(11,5)*{}};
{\ar (10,-5)*{};(11,-5)*{}};
(13,0)*{_{#1}};
(0,7)*{_{#2}};
\endxy
}

\newcommand{\leftmedtwistds}[2]{
\xy 0;/r.15pc/:
\htwist~{(-10,5)}{(0,5)}{(-10,-5)}{(0,-5)};
\htwist~{(0,5)}{(10,5)}{(0,-5)}{(10,-5)};
(0,1)*{\vdots};
(-6,0)="h1";
(6,0)="h2";
(-10,1)*{};"h1"+(-1,1) **\dir{-};
(-10,-1)*{};"h1"+(-1,-1) **\dir{-};
(10,1)*{};"h2"+(0,1) **\dir{-};
(10,-1)*{};"h2"+(0,-1) **\dir{-};
{\ar "h1"+(0,1);(-11,1)};
{\ar "h1"+(0,-1);(-11,-1)};
{\ar (-10,5)*{};(-11,5)*{}};
{\ar (-10,-5)*{};(-11,-5)*{}};
(15,0)*{_{#1}};
(0,7)*{_{#2}};
\endxy
}

\newcommand{\smalltwistmn}[3]{\xy 0;/r.1pc/:
\htwist~{(-10,5)}{(0,5)}{(-10,-5)}{(0,-5)};
\htwist~{(0,5)}{(10,5)}{(0,-5)}{(10,-5)};
(0,2.5)*{\vdots};
{\ar (10,5)*{};(11,5)*{}};
{\ar (-10,-5)*{};(-11,-5)*{}};
(15,5)*{_{#1}};
(15,-5)*{_{#2}};
(0,9)*{_{#3}};
\endxy}

\newcommand{\smalltwist}[1]{\xy 0;/r.1pc/:
\htwist~{(-10,5)}{(0,5)}{(-10,-5)}{(0,-5)};
\htwist~{(0,5)}{(10,5)}{(0,-5)}{(10,-5)};
(0,2.5)*{\vdots};
(0,8)*{_{#1}};
\endxy}

\newcommand{\medtwistn}[2]{\xy 0;/r.15pc/:
\htwist~{(-10,5)}{(0,5)}{(-10,-5)}{(0,-5)};
\htwist~{(0,5)}{(10,5)}{(0,-5)}{(10,-5)};
(0,2.5)*{\vdots};
{\ar (-10,5)*{};(-12,5)*{}};
{\ar (-10,-5)*{};(-12,-5)*{}};
(13,0)*{_{#1}};
(0,7)*{^{#2}};
\endxy}

\newcommand{\smalltwotwistoneone}{\xy 0;/r.1pc/:
\htwist~{(-10,5)}{(0,5)}{(-10,-5)}{(0,-5)};
\htwist~{(0,5)}{(10,5)}{(0,-5)}{(10,-5)};
{\ar (10,5)*{};(12,5)*{}};
{\ar (-10,-5)*{};(-12,-5)*{}};
\endxy}

\newcommand{\medtwotwistoneone}{\xy 0;/r.15pc/:
\htwist~{(-10,5)}{(0,5)}{(-10,-5)}{(0,-5)};
\htwist~{(0,5)}{(10,5)}{(0,-5)}{(10,-5)};
{\ar (10,5)*{};(12,5)*{}};
{\ar (-10,-5)*{};(-12,-5)*{}};
\endxy}

\newcommand{\twistmn}[3]{\xy
\htwist~{(-10,3)}{(0,3)}{(-10,-3)}{(0,-3)};
\htwist~{(0,3)}{(10,3)}{(0,-3)}{(10,-3)};
(0,1)*{\vdots};
{\ar (10,3)*{};(11,3)*{}};
{\ar (-10,-3)*{};(-11,-3)*{}};
(15,3)*{_{#1}};
(15,-3)*{_{#2}};
(0,5)*{_{#3}};
\endxy}

\newcommand{\onewrapmn}[2]{
\xy
(-10,5)*{}; (0,-5)*{} **\crv{(-6.5,5)&(-3.5,-5)} \POS?(.4)*{\hole}="h1"
\POS?(.7)*{\hole}="h2";
(-10,0)*{}; "h1" **\crv{(-8,0)};
(-10,-5)*{}; "h2" **\crv{(-6,-5)} ?(0)*\dir{<};
"h1"; (10,0)*{} **\crv{(2,10)&(3,0)}\POS?(.7)*{\hole}="h4";
"h2"; (10,-5)*{} **\crv{(0,3)&(3,-4)} \POS?(.6)*{\hole}="h3" ;
(0,-5)*{}; "h3" **\crv{(2,-5)};
"h3"; "h4" **\crv{(3,-2)};
"h4"; (10,5)*{} **\crv{(7,5)};
(9.9, 4.985); (10,5)**\dir{}?(1)*\dir{>};
(9.9,.01); (10,0)**\dir{}?(1)*\dir{>};
(14,0)*{^{_{#1}}};
(12,-5)*{^{_{#2}}};
\endxy
}

\newcommand{\medonewrapmn}[2]{
\xy 0;/r.15pc/:
(-10,5)*{}; (0,-5)*{} **\crv{(-6.5,5)&(-3.5,-5)} \POS?(.4)*{\hole}="h1"
\POS?(.7)*{\hole}="h2";
(-10,0)*{}; "h1" **\crv{(-8,0)};
(-10,-5)*{}; "h2" **\crv{(-6,-5)} ;
"h1"; (10,0)*{} **\crv{(2,10)&(3,0)}\POS?(.7)*{\hole}="h4";
"h2"; (10,-5)*{} **\crv{(0,3)&(3,-4)} \POS?(.6)*{\hole}="h3" ;
(0,-5)*{}; "h3" **\crv{(2,-5)};
"h3"; "h4" **\crv{(3,-2)};
"h4"; (10,5)*{} **\crv{(7,5)};
(9.9, 4.985); (10,5)**\dir{}?(1)*\dir{>};
(9.9,.01); (10,0)**\dir{}?(1)*\dir{>};
(-9.9, -4.985); (-10,-5)**\dir{}?(1)*\dir{>}; 
(16,0)*{^{_{#1}}};
(14,-5)*{^{_{#2}}};
\endxy
}

\newcommand{\smallhalftwistmn}[2]{\xy 0;/r.1pc/:
\htwist~{(0,5)}{(10,5)}{(0,-5)}{(10,-5)};
{\ar (10,5)*{};(12,5)*{}};
{\ar (10,-5)*{};(12,-5)*{}};
(20,5)*{^{_{#1}}};
(20,-5)*{^{_{#2}}};
\endxy}

\newcommand{\crossdown}{\xy 0;/r.1pc/:
\htwist~{(0,5)}{(10,5)}{(0,-5)}{(10,-5)};
{\ar (0,-5)*{};(-2,-5)*{}};
{\ar (10,-5)*{};(12,-5)*{}};
(10,5); (12,5)**\dir{-};
(0,5); (-2,5)**\dir{-};
\endxy}

\newcommand{\crossup}{\xy 0;/r.1pc/:
\htwist~{(0,5)}{(10,5)}{(0,-5)}{(10,-5)};
{\ar (10,5)*{};(12,5)*{}};
{\ar (0,5)*{};(-2,5)*{}};
(10,-5); (12,-5)**\dir{-};
(0,-5); (-2,-5)**\dir{-};
\endxy}

\newcommand{\crossleft}{\xy 0;/r.1pc/:
(0,5)*{}; (-10,-5)*{} **\crv{(-3.5,5)&(-6.5,-5)} \POS?(.5)*{\hole}="h2";
(0,-5)*{}; "h2" **\crv{(-2.5,-5)};
"h2"; (-10,5) **\crv{(-7.5,5)};
{\ar (-10,5)*{};(-12,5)*{}};
{\ar (-10,-5)*{};(-12,-5)*{}};
\endxy}

\newcommand{\smallYro}{
\xy
(-3,0)*{}="a";
(.25,0)*{}="b";
(3,2.5)*{}="c";
(3,-2.5)*{}="d";
{\ar "b"; "a" };
{\ar "b"; "c" };
{\ar "b"; "d" };
\endxy
}

\newcommand{\smallYri}{
\xy
(-3,0)*{}="a";
(.25,0)*{}="b";
(3,2.5)*{}="c";
(3,-2.5)*{}="d";
"a";"b" **\dir{-}?(.6)*\dir{>};
"c";"b" **\dir{-}?(.6)*\dir{>};
"d";"b" **\dir{-}?(.6)*\dir{>};
\endxy
}

\newcommand{\smallYli}{
\xy
(3,0)*{}="a";
(-.25,0)*{}="b";
(-3,2.5)*{}="c";
(-3,-2.5)*{}="d";
"a";"b" **\dir{-}?(.6)*\dir{>};
"c";"b" **\dir{-}?(.6)*\dir{>};
"d";"b" **\dir{-}?(.6)*\dir{>};
\endxy
}

\newcommand{\smallYlo}{
\xy
(3,0)*{}="a";
(-.25,0)*{}="b";
(-3,2.5)*{}="c";
(-3,-2.5)*{}="d";
{\ar "b"; "a" };
{\ar "b"; "c" };
{\ar "b"; "d" };\endxy
}

\newcommand{\turnback}{\xy 
(-20,17.5)*{} = "L1";
(-20,12.5)*{} = "L2";
(-20,7.5)*{} = "L3";
(-20,2.5)*{} = "L4";
(-20,-2.5)*{} = "L5";
(-20,-7.5)*{} = "L6";
(-20,-12.5)*{} = "L7";
(-20,-17.5)*{} = "L8";
(20,17.5)*{} = "R1";
(20,12.5)*{} = "R2";
(20,7.5)*{} = "R3";
(20,2.5)*{} = "R4";
(20,-2.5)*{} = "R5";
(20,-7.5)*{} = "R6";
(20,-12.5)*{} = "R7";
(20,-17.5)*{} = "R8";
(-15,-12.5)*{\hole} = "h1";
(-15,-7.5)*{\hole} = "h2";
(-15,-2.5)*{\hole} = "h3";
(-10,-7.5)*{\hole} = "h4";
(-10,-2.5)*{\hole} = "h5";
(-11,12.5)*{\hole} = "h6";
(-3.5,-2.5)*{\hole} = "h7";
"L2"; "h6" **\dir{-};
"L4"; "L8" **\crv{(-13,4)&(-13,-19)}?(1)*\dir{>};
"L3"; "h1" **\crv{(-7,5)&(-7,-15)};
"L1"; "h4" **\crv{(0,10)&(0,-10)};
"L7"; "h1" **\dir{-}?(0)*\dir{<};
"L6"; "h2" **\dir{-}?(0)*\dir{<};
"L5"; "h3" **\dir{-};
"h2"; "h4" **\dir{-};
"h3"; "h5" **\dir{-};
"h5"; "h7" **\dir{-};
"h6"; "R2" **\dir{-};
"h7"; "R5" **\dir{-};
(15,-12.5)*{\hole} = "H1";
(15,-7.5)*{\hole} = "H2";
(15,-2.5)*{\hole} = "H3";
(10,-7.5)*{\hole} = "H4";
(10,-2.5)*{\hole} = "H5";
(10,12.5)*{\hole} = "H6";
(3.5,-2.5)*{\hole} = "H7";
"H4"; "R7" **\crv{(10,-13.5)};
"H7"; "R6" **\crv{(5,-10)};
"H5"; "R3" **\crv{(10,5)};
"H1"; "R8" **\crv{(15,-17)};
"H3"; "R4" **\crv{(15,2.5)};
"H7"; "H6" **\crv{(3,10)};
"H6"; "R1" **\crv{(12,15)};
"H4"; "H5" **\dir{-};
"H1"; "H2" **\dir{-};
"H2"; "H3" **\dir{-};
\endxy}

\newcommand{\threetwist}{\xy 
(-10,17.5)*{} = "L1";
(-10,12.5)*{} = "L2";
(-10,7.5)*{} = "L3";
(-10,2.5)*{} = "L4";
(-10,-2.5)*{} = "L5";
(-10,-7.5)*{} = "L6";
(-10,-12.5)*{} = "L7";
(-10,-17.5)*{} = "L8";
(10,17.5)*{} = "R1";
(10,12.5)*{} = "R2";
(10,7.5)*{} = "R3";
(10,2.5)*{} = "R4";
(10,-2.5)*{} = "R5";
(10,-7.5)*{} = "R6";
(10,-12.5)*{} = "R7";
(10,-17.5)*{} = "R8";
"L1"; "R1" **\dir{-}?(1)*\dir{>};
"L2"; "R2" **\dir{-}?(1)*\dir{>};
"L3"; "R3" **\dir{-}?(1)*\dir{>};
"L4"; "R4" **\dir{-}?(1)*\dir{>};
"L5"; "R5" **\dir{-}?(1)*\dir{>};
(-6,-10)*{\hole}="h1";
(-4.5,-14)*{\hole}="h2";
(-1,-9.5)*{\hole}="h3";
(1,-14.5)*{\hole}="h4";
(4.5,-10)*{\hole}="h5";
(6,-15)*{\hole}="h6";
"L6"; "h4" **\crv{(-2,-8)&(-6,-24)};
"h1"; "h6" **\crv{(2,-1)&(-1,-25)};
"h3"; "R8" **\crv{(7,-2)&(3,-19)};
"L7"; "h1" **\crv{(-8,-12.5)};
"L8"; "h2" **\crv{(-6,-17.5)};
"h2"; "h3" **\crv{(-2.5,-11)};
"h4"; "h5" **\crv{(2.5,-12)};
"h5"; "R6" **\crv{(7,-7.5)};
"h6"; "R7" **\crv{(8,-12.5)};
\endxy}

\title[Categorification of Quantum $\mathfrak{sl}_3$ Projectors]{A Categorification of Quantum $\mathfrak{sl}_3$ Projectors and the 
$\mathfrak{sl}_3$ Reshetikhin-Turaev Invariant of Tangles}

\author[D. E. V. Rose]{David E. V. Rose}

\begin{document}

\maketitle

\begin{abstract}
We construct a categorification of the quantum $\mathfrak{sl}_3$ projectors, the $\mathfrak{sl}_3$ analog of 
the Jones-Wenzl projectors, 
as the stable limit of the complexes assigned to 
$k$-twist torus braids (as $k \to \infty$) in a suitably shifted version of
Morrison and Nieh's geometric formulation of $\mathfrak{sl}_3$ link homology \cite{MorrisonNieh}. 
We use these projectors to give a categorification of the 
$\mathfrak{sl}_3$ Reshetikhin-Turaev invariant of framed tangles.
\end{abstract}



\section{Introduction}
In \cite{BarNatan2}, Bar-Natan introduced a geometric formulation of Khovanov's $\mathfrak{sl}_2$ homology 
theory for tangles. In this
framework, the invariant of a tangle takes values in a category whose objects are complexes composed of crossingless 
tangles and cobordisms. This construction gives a categorification of the Temperley-Lieb algebra, meaning that the 
Temperley-Lieb algebra can be recovered from this category by taking the Grothendieck group. 

The Jones-Wenzl projectors are special idempotent elements of the Temperley-Lieb algebra relevant to quantum topology, 
where they are used to give combinatorial constructions of the $\mathfrak{sl}_2$ Reshetikhin-Turaev invariants of 
framed links (i.e. the colored Jones polynomial) and the $\mathfrak{sl}_2$ Witten-Reshetikhin-Turaev invariants of 
3-manifolds. A natural question to ask is whether there exist objects in Bar-Natan's category corresponding to these 
projectors. 

This question has been answered in the affirmative by various authors using differing constructions. In  
\cite{Rozansky}, Rozansky has constructed complexes in Bar-Natan's category satisfying categorified versions of the 
defining relations of the Jones-Wenzl projectors and mapping to them 
via the canonical map to the Grothendieck group. These complexes 
are presented as the stable limit of the complexes assigned to $k$-twist torus braids as $k\to \infty$, i.e. as the 
complexes associated to `infinite twists.' Cooper and Krushkal \cite{CooperKrushkal} used a categorified version of the 
Frenkel-Khovanov recursion relation for the Jones-Wenzl projectors \cite{FrenkelKhovanov} to give an alternative 
construction of these categorified projectors. Finally, Frenkel, Stroppel, and Sussan \cite{FrenkelStroppelSussan} 
have used category $\mathcal{O}$ methods to construct categorified Jones-Wenzl projectors. It can be shown that the 
constructions of Rozansky and Cooper-Krushkal agree and that these constructions are related to that of 
Frenkel-Stroppel-Sussan \cite{StroppelSussan} (see also \cite{MazorchukStroppel} for a connection between the category 
$\mathcal{O}$ formulation of Khovanov homology and the webs and foams formulation used in this paper).

In this paper, we extend the results of Rozansky and Cooper-Krushkal to the $\mathfrak{sl}_3$ case.
Playing the role of the Temperley-Lieb algebra is Kuperberg's $\mathfrak{sl}_3$ spider \cite{Kuperberg}, a 
combinatorial construction which describes a full subcategory of the category $\Rep \mathcal{U}_q(\mathfrak{sl}_3)$ of 
(type I) finite dimensional representations of the quantum group 
$\mathcal{U}_q(\mathfrak{sl}_3)$ at generic $q$. The quantum $\mathfrak{sl}_3$ knot invariant has a simple description 
in this 
context which has been categorified in \cite{Khovanov3}. Mackaay and Vaz gave a geometric reformulation of this knot 
homology in \cite{MackaayVaz} akin to Bar-Natan's construction which was extended by Morrison and Nieh in
\cite{MorrisonNieh} to give a categorification of the $\mathfrak{sl}_3$ spider.

The analog of the Jones-Wenzl projectors are the $\mathfrak{sl}_3$ projectors, also called `internal clasps' or 
`magic elements', which correspond 
to projection onto highest weight irreducible summands in $\Rep \mathcal{U}_q(\mathfrak{sl}_3)$. 
These projectors have been studied in 
\cite{Kim} and have been used to construct quantum invariants of $3$-manifolds \cite{OhtsukiYamada}. The main result of 
this paper is the following: 

\begin{nnthm} The complex assigned to a $k$-twist torus braid (suitably shifted) 
in Morrison and Nieh's categorification of the 
$\mathfrak{sl}_3$ spider stabilizes as $k\to \infty$ and the stable limit gives a categorified 
$\mathfrak{sl}_3$ projector.
\end{nnthm}

\noindent We shall give a more precise statement of this result in Theorem \ref{1} 
and the remarks following that theorem.
We use these categorified projectors to construct an invariant of framed tangles which categorifies the 
$\mathfrak{sl}_3$ Reshetikhin-Turaev invariant of framed tangles. 
See Definition \ref{invt} and Theorem \ref{3} for the 
construction of this invariant and details of its properties.

The paper is organized as follows. 
In Section \ref{background} we provide the relevant background information on the $\mathfrak{sl}_3$ spider 
and projectors. We also review the $\mathfrak{sl}_3$ tangle invariant,
$\mathfrak{sl}_3$ homology, and categorification. This section also contains a summary of the major 
results of the paper.

We present the relevant homological algebra in Section \ref{homoalg}. The results in Subsection \ref{homoalg1} 
are standard and are for the most part presented without proof. Subsections \ref{homoalg2} and \ref{homoalg3} 
contain results concerning the calculus of chain complexes; these results are due to Rozansky and were 
originally presented in \cite{Rozansky}. We give a slightly different treatment which is adapted for our 
purposes.

Section \ref{sectioncatproj} contains the bulk of the content of the paper. In this section we construct 
the categorified projectors and show that they satisfy categorified versions of the properties which define 
the projectors in the $\mathfrak{sl}_3$ spider. We also show that the categorified projectors 
decategorify to give the $\mathfrak{sl}_3$ projectors.

We define the tangle invariant in Section \ref{sectioninvt} and compute some examples.

\noindent \textbf{Acknowledgments:} I would like to thank Scott Morrison and Lev Rozansky
for useful discussions; their work on $\mathfrak{sl}_3$ knot homology and 
categorified Jones-Wenzl projectors (respectively)
served as motivation for this work. I am also grateful to Dick Hain and Ezra Miller for a useful 
tea-time discussion concerning some homological algebra and to Sabin Cautis for 
pointing out a typo in (an earlier version of) Example \ref{ex++}.
Finally, I would like to thank my advisor Lenny Ng for his continued guidance. 
The author was partially supported by NSF grant DMS-0846346 during the completion of this work.

\section{Background and summary of results}\label{background}

\subsection{The $\mathfrak{sl}_3$ spider}
The $\mathfrak{sl}_3$ spider $\mathcal{S}$, introduced in \cite{Kuperberg}, can be interpreted as a pivotal category 
giving a diagrammatic description of the subcategory of $\Rep \mathcal{U}_q(\mathfrak{sl}_3)$ generated (as a pivotal 
category) by 
the standard representation (see \cite{MorrisonThesis} for details of this interpretation, \cite{BarrettWestbury} 
for background on pivotal categories, and \cite{Selinger} for the graphical description of such categories). 
The objects in this category are words in the symbols 
$+$ and $-$, tensor product is given by concatenation of words, and the dual of a word is obtained 
by reversing the word and 
switching all signs. For example \[(+-+)\otimes(--+) = (+-+--+)\] and \[(+-+--+)^* = (-++-+-).\] 
Morphisms in $\mathcal{S}$ are given by $\C(q)$-linear combinations of webs - oriented, trivalent planar graphs
whose edges are all directed into or out from a trivalent vertex - with appropriate boundary, modulo local relations. 
We shall refer to the parameter $q$ as the quantum degree. 
Edges should be directed into $+$ and out from $-$ in the codomain and vice-versa in the domain; the local relations 
are given as follows:

\begin{align}
\label{spidereqn1}
\xy 0;/r.1pc/:
(-10,0)*{}; (10,0)*{} **\crv{(0,8)}?(.45)*\dir{<};
(-10,0)*{}; (10,0)*{} **\crv{(0,-8)}?(.45)*\dir{<};
(-20,0)*{}; (-10,0)*{}; **\dir{-}?(.3)*\dir{<};
(10,0)*{}; (20,0)*{}; **\dir{-}?(.3)*\dir{<};
\endxy
&= [2] \hspace{1mm}
\xy
(-5,0)*{}; (5,0)*{} **\dir{-}?(.55)*\dir{>};
\endxy \\
\label{spidereqn2}
\xy 0;/r.15pc/:
(-5,5)*{}="a";
(5,5)*{}="b";
(-5,-5)*{}="c";
(5,-5)*{}="d";
(-10,10)*{}="e";
(10,10)*{}="f";
(-10,-10)*{}="g";
(10,-10)*{}="h";
"a";"b" **\dir{-}?(.6)*\dir{>};
"d";"b" **\dir{-}?(.6)*\dir{>};
"a";"c" **\dir{-}?(.6)*\dir{>};
"d";"c" **\dir{-}?(.6)*\dir{>};
"a";"e" **\dir{-}?(.6)*\dir{>};
"f";"b" **\dir{-}?(.6)*\dir{>};
"g";"c" **\dir{-}?(.6)*\dir{>};
"d";"h" **\dir{-}?(.6)*\dir{>};
\endxy
&=
\xy 0;/r.13pc/:
(-10,10)*{}="e";
(10,10)*{}="f";
(-10,-10)*{}="g";
(10,-10)*{}="h";
"f";"e" **\crv{-}?(.55)*\dir{>};
"g";"h" **\crv{-}?(.55)*\dir{>};
\endxy +
\xy 0;/r.13pc/:
(-10,10)*{}="e";
(10,10)*{}="f";
(-10,-10)*{}="g";
(10,-10)*{}="h";
"f";"h" **\crv{-}?(.55)*\dir{>};
"g";"e" **\crv{-}?(.55)*\dir{>};
\endxy\\
\label{spidereqn3}
\xy 0;/r.15pc/:
(0,0)*\xycircle(5,5){-};
{\ar@{->}(0,5)*{};(1,5)*{}}; 
\endxy 
&= [3]
\end{align}
where $ [n] = \frac{q^n-q^{-n}}{q-q^{-1}}$. 
A web with no digon, square, or circular faces is called non-elliptic. Using 
the above relations, any web can be expressed as a $\Z[q^{-1},q]$-linear sum of non-elliptic webs; 
in fact, for a fixed boundary, the non-elliptic webs give a basis for the corresponding $\Hom$-set \cite{Kuperberg}. An Euler 
characteristic argument shows that there cannot exist a non-empty, closed, non-elliptic web.

Tensor product is given by placing the webs next to each other (vertically) 
in a disjoint manner and composition (denoted by $\bullet$)
is given by gluing together the boundaries of the webs. For example,
\[
\xy
(0,0)*{};(5,0)*{} **\dir{-}?(.7)*\dir{>};
(0,0)*{};(-2.5,4.33)*{} **\dir{-}?(.7)*\dir{>};
(0,0)*{};(-2.5,-4.33)*{} **\dir{-}?(.7)*\dir{>};
\endxy 
\otimes
\xy
(-5,0)*{}; (5,0)*{} **\dir{-}?(.55)*\dir{>};
\endxy \ = \ 
\xy
(0,1)*{};(5,1)*{} **\dir{-}?(.7)*\dir{>};
(0,1)*{};(-2.5,5.33)*{} **\dir{-}?(.7)*\dir{>};
(0,1)*{};(-2.5,-3.33)*{} **\dir{-}?(.7)*\dir{>};
(-2.5,-5)*{}; (5,-5)*{} **\dir{-}?(.55)*\dir{>};
\endxy 
\]

and 
\[
\xy
(0,0)*{};(5,0)*{} **\dir{-}?(.7)*\dir{>};
(0,0)*{};(-2.5,4.33)*{} **\dir{-}?(.7)*\dir{>};
(0,0)*{};(-2.5,-4.33)*{} **\dir{-}?(.7)*\dir{>};
\endxy \bullet
\xy
(0,0)*{};(-5,0)*{} **\dir{-}?(.3)*\dir{<};
(0,0)*{};(2.5,4.33)*{} **\dir{-}?(.3)*\dir{<};
(0,0)*{};(2.5,-4.33)*{} **\dir{-}?(.3)*\dir{<};
\endxy \ = \ 
\xy
(-3,0)*{};(-5,-5)*{} **\dir{-}?(.6)*\dir{>};
(-3,0)*{};(-5,5)*{} **\dir{-}?(.6)*\dir{>};
(3,0)*{};(-3,0)*{} **\dir{-}?(.4)*\dir{<};
(5,-5)*{};(3,0)*{} **\dir{-}?(.6)*\dir{>};
(5,5)*{};(3,0)*{} **\dir{-}?(.6)*\dir{>};
\endxy \ \ .
\]
Note that tangle composition is denoted in the non-traditional but diagrammatically more pleasing order; 
when considering morphisms in other categories we will use the traditional order for composition.
The dual of a web is obtained by rotating the web $180^{\circ}$ and the pairing is given by the webs
\[
\xy
(-5,5)*{}; (-5,-5)*{} **\crv{(5,5)&(5,-5)}?(1)*\dir{>};
\endxy
\text{    and    } \ \
\xy
(-5,5)*{}; (-5,-5)*{} **\crv{(5,5)&(5,-5)}?(0)*\dir{<};
\endxy.
\]
We will call these webs and their duals $U$-webs.

The correspondence between $\mathcal{S}$ and the subcategory of $\Rep \mathcal{U}_q(\mathfrak{sl}_3)$ generated by the 
standard
representation is given as follows. The symbols $+$ and $-$ correspond to the standard representation $V$ of 
$\mathcal{U}_q(\mathfrak{sl}_3)$ and its dual $V^*$. The morphisms
\[
\xy
(0,0)*{};(5,0)*{} **\dir{-}?(.7)*\dir{>};
(0,0)*{};(-2.5,4.33)*{} **\dir{-}?(.7)*\dir{>};
(0,0)*{};(-2.5,-4.33)*{} **\dir{-}?(.7)*\dir{>};
\endxy
\text{\ \ \  and \ \ \ }
\xy
(0,0)*{};(5,0)*{} **\dir{-}?(.3)*\dir{<};
(0,0)*{};(-2.5,4.33)*{} **\dir{-}?(.3)*\dir{<};
(0,0)*{};(-2.5,-4.33)*{} **\dir{-}?(.3)*\dir{<};
\endxy
\]
correspond to the unique (up to scalar multiple) maps $V^*\otimes V^* \rightarrow V$ and $V\otimes V \rightarrow V^*$ 
(we shall read all diagrams as mapping from the domain on the left to the codomain on the right).
We will call these morphisms and their duals $Y$-webs. All other morphisms 
can be obtained from $Y$-webs and $U$-webs via composition and tensor product. The main result of 
\cite{Kuperberg} is that under this correspondence $\mathcal{S}$ is equivalent to the full subcategory of 
$\mathcal{U}_q(\mathfrak{sl}_3)$ representations generated by $V$. One consequence of this result is that there exist 
only finitely many non-elliptic webs with a given boundary since the corresponding $\Hom$-set is a finite-dimensional 
vector space.

\subsection{$\mathfrak{sl}_3$ projectors}
Kuperberg introduced `internal clasps' in his initial study of the $\mathfrak{sl}_3$ spider \cite{Kuperberg}. Under the equivalence 
outlined above, these clasps are idempotent elements $P_w\in\Hom_{\bullet}(w,w)$ which correspond to projection 
onto the highest weight irreducible summand (and then inclusion). Here $w$ is a word of $+$'s and $-$'s.

We shall denote these projectors graphically by
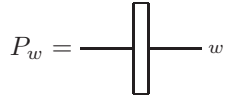
\[
P_w =
\xy
(-8,0); (-1,0) **\dir{-};
(-1,-6)*{}; (-1,6)*{} **\dir{-};
(-1,-6)*{}; (1,-6)*{} **\dir{-};
(-1,6)*{}; (1,6)*{} **\dir{-};
(1,-6)*{}; (1,6)*{} **\dir{-};
(1,0); (8,0) **\dir{-};
(10,0)*{_{w}};
\endxy
\]
when we don't wish to specify the word $w$ and by orienting the strands and labeling them with numbers 
corresponding to their multiplicities when we do. 
For instance, 
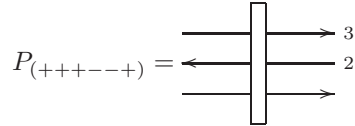
\[
P_{(+++--+)}=
\xy
(12,4)*{_3};
(12,0)*{_2};
(-10,-4)*{}; (-1,-4)*{} **\dir{-};
(-10,0)*{}; (-1,0)*{} **\dir{-}?(0)*\dir{<};
(-10,4)*{}; (-1,4)*{} **\dir{-};
(-1,-8)*{}; (-1,8)*{} **\dir{-};
(-1,-8)*{}; (1,-8)*{} **\dir{-};
(-1,8)*{}; (1,8)*{} **\dir{-};
(1,-8)*{}; (1,8)*{} **\dir{-};
(1,-4)*{}; (10,-4)*{} **\dir{-}?(1)*\dir{>};
(1,0)*{}; (10,0)*{} **\dir{-};
(1,4)*{}; (10,4)*{} **\dir{-}?(1)*\dir{>};
\endxy
\]
where unlabeled strands have multiplicity one.

The projectors are defined as follows. Let the weight of a word $w$ be given by 
\[ \wt(w)=(w{_+},w{_-}) \in \Z^2_{\geq0}\] where $w_{\pm}$ denotes the number of $\pm$ signs appearing in the word. 
There is a partial order on words generated by the relations 
\begin{align*}
(w_+,w_-) &> (w{_+}+1,w{_-}-2) \\ 
(w_+,w_-) &> (w{_+}-2,w{_-}+1),
\end{align*}
corresponding to the partial order on the weight lattice for $\mathfrak{sl}_3$. The projector $P_w$ is the unique 
non-zero idempotent element in $\Hom_{\bullet}(w,w)$ satisfying the condition that if 
$\wt(v) < \wt(w)$ then $P_w\bullet W_1 = 0$ 
for any $W_1 \in \Hom_{\bullet}(w,v)$ and $W_2\bullet P_w = 0$ for any $W_2 \in \Hom_{\bullet}(v,w)$
(recall our conventions for the order of tangle composition!). It follows that $(P_w)^* = P_{{w^*}}$.

Of particular importance are the segregated projectors, those of the form 
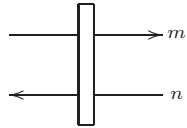
\[
 \xy
(12,4)*{_m};
(12,-4)*{_n};
(-10,-4)*{}; (-1,-4)*{} **\dir{-}?(0)*\dir{<};
(-10,4)*{}; (-1,4)*{} **\dir{-};
(-1,-8)*{}; (-1,8)*{} **\dir{-};
(-1,-8)*{}; (1,-8)*{} **\dir{-};
(-1,8)*{}; (1,8)*{} **\dir{-};
(1,-8)*{}; (1,8)*{} **\dir{-};
(1,-4)*{}; (10,-4)*{} **\dir{-};
(1,4)*{}; (10,4)*{} **\dir{-}?(1)*\dir{>};
\endxy
\]
for $m,n\geq0$; we will refer to the domain ($=$ codomain) of a segregated projector as a segregated word. 
All other projectors can be obtained from these by inserting `$H$-webs,' those of the form
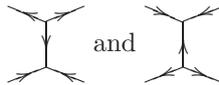
\[
 \xy
(0,3)*{};(-5,5)*{} **\dir{-}?(.6)*\dir{>};
(0,3)*{};(5,5)*{} **\dir{-}?(.6)*\dir{>};
(0,-3)*{};(0,3)*{} **\dir{-}?(.4)*\dir{<};
(-5,-5)*{};(0,-3)*{} **\dir{-}?(.6)*\dir{>};
(5,-5)*{};(0,-3)*{} **\dir{-}?(.6)*\dir{>};
\endxy \text{   and   }
\xy
(0,-3)*{};(-5,-5)*{} **\dir{-}?(.6)*\dir{>};
(0,-3)*{};(5,-5)*{} **\dir{-}?(.6)*\dir{>};
(0,3)*{};(0,-3)*{} **\dir{-}?(.4)*\dir{<};
(-5,5)*{};(0,3)*{} **\dir{-}?(.6)*\dir{>};
(5,5)*{};(0,3)*{} **\dir{-}?(.6)*\dir{>};
\endxy
\]
to permute the order of the $+$'s and $-$'s. For example $P_{(+-+)}$ can be obtained from $P_{(++-)}$ as follows:
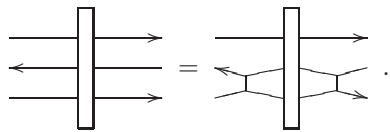
\[
 \xy
(-10,-4)*{}; (-1,-4)*{} **\dir{-};
(-10,0)*{}; (-1,0)*{} **\dir{-}?(0)*\dir{<};
(-10,4)*{}; (-1,4)*{} **\dir{-};
(-1,-8)*{}; (-1,8)*{} **\dir{-};
(-1,-8)*{}; (1,-8)*{} **\dir{-};
(-1,8)*{}; (1,8)*{} **\dir{-};
(1,-8)*{}; (1,8)*{} **\dir{-};
(1,-4)*{}; (10,-4)*{} **\dir{-}?(1)*\dir{>};
(1,0)*{}; (10,0)*{} **\dir{-};
(1,4)*{}; (10,4)*{} **\dir{-}?(1)*\dir{>};
\endxy \ = \
\xy
(12,0)*{}; 
(-10,0)*{}; (-6,-1)*{} **\dir{-}?(0)*\dir{<};
(-10,-4)*{}; (-6,-3)*{} **\dir{-};
(-6,-1)*{}; (-6,-3)*{} **\dir{-};
(-6,-1)*{}; (-1,0)*{} **\dir{-};
(-6,-3)*{}; (-1,-4)*{} **\dir{-};
(-10,4)*{}; (-1,4)*{} **\dir{-};
(-1,-8)*{}; (-1,8)*{} **\dir{-};
(-1,-8)*{}; (1,-8)*{} **\dir{-};
(-1,8)*{}; (1,8)*{} **\dir{-};
(1,-8)*{}; (1,8)*{} **\dir{-};
(1,4)*{}; (10,4)*{} **\dir{-}?(1)*\dir{>};
(1,0)*{}; (6,-1)*{} **\dir{-};
(1,-4)*{}; (6,-3)*{} **\dir{-};
(6,-1)*{}; (6,-3)*{} **\dir{-};
(6,-1)*{}; (10,0)*{} **\dir{-};
(6,-3)*{}; (10,-4)*{} **\dir{-}?(1)*\dir{>};
\endxy.
\]
The following result, proved in \cite{OhtsukiYamada}, gives a recursive formula for the segregated projectors.
\begin{prop}\label{projrecur}
For $m>0$,
\begin{equation}\label{(m,0)projector}
\xy
(13,0)*{_{m}};
(-10,0)*{}; (-1,0)*{} **\dir{-};
(-1,-6)*{}; (-1,6)*{} **\dir{-};
(-1,-6)*{}; (1,-6)*{} **\dir{-};
(-1,6)*{}; (1,6)*{} **\dir{-};
(1,-6)*{}; (1,6)*{} **\dir{-};
(1,0)*{}; (10,0)*{} **\dir{-}?(1)*\dir{>};
\endxy =
\xy
(14,4)*{_{m-1}};
(-1,8)*{}; (1,8)*{} **\dir{-};
(-1,0)*{}; (-1,8)*{} **\dir{-};
(1,8)*{}; (1,0)*{} **\dir{-};
(-1,0)*{}; (1,0)*{} **\dir{-};
(-10,4)*{}; (-1,4)*{} **\dir{-};
(1,4)*{}; (10,4)*{} **\dir{-}?(1)*\dir{>};
(-10,-4)*{}; (10,-4)*{} **\dir{-}?(1)*\dir{>};
\endxy
- \frac{[m-1]}{[m]}
\xy
(-12,0)*{}; 
(14,4)*{_{m-1}};
(0,9)*{_{m-2}};
(-10,4)*{}; (-6,4)*{} **\dir{-};
(6,4)*{}; (10,4)*{} **\dir{-}?(1)*\dir{>};
(-4,6)*{}; (4,6)*{} **\dir{-};
(4,8)*{}; (6,8)*{} **\dir{-};
(4,0)*{}; (4,8)*{} **\dir{-};
(4,0)*{}; (6,0)*{} **\dir{-};
(6,0)*{}; (6,8)*{} **\dir{-};
(-4,8)*{}; (-6,8)*{} **\dir{-};
(-4,0)*{}; (-4,8)*{} **\dir{-};
(-4,0)*{}; (-6,0)*{} **\dir{-};
(-6,0)*{}; (-6,8)*{} **\dir{-};
(-2,-2)*{}; (2,-2)*{} **\dir{-};
(-4,2)*{}; (-2,-2)*{} **\dir{-};
(2,-2)*{}; (4,2)*{} **\dir{-};
(-10,-4)*{}; (-2,-2)*{} **\crv{(-6,-4)};
(2,-2)*{}; (10,-4)*{} **\crv{(6,-4)}?(1)*\dir{>};
\endxy
\end{equation}
and for $m,n>0$
\begin{equation}
\xy
(12,4)*{_m};
(12,-4)*{_n};
(-10,-4)*{}; (-1,-4)*{} **\dir{-}?(0)*\dir{<};
(-10,4)*{}; (-1,4)*{} **\dir{-};
(-1,-8)*{}; (-1,8)*{} **\dir{-};
(-1,-8)*{}; (1,-8)*{} **\dir{-};
(-1,8)*{}; (1,8)*{} **\dir{-};
(1,-8)*{}; (1,8)*{} **\dir{-};
(1,-4)*{}; (10,-4)*{} **\dir{-};
(1,4)*{}; (10,4)*{} **\dir{-}?(1)*\dir{>};
\endxy = \sum_{k=0}^{\min(m,n)}(-1)^{k} \frac{[m]![n]![m+n-k+1]!}{[m-k]![n-k]![m+n+1]![k]!}
\xy
(-12,0)*{}; 
(4,2)*{}; (6,2)*{} **\dir{-};
(6,2)*{}; (6,10)*{} **\dir{-};
(6,10)*{}; (4,10)*{} **\dir{-};
(4,10)*{}; (4,2)*{} **\dir{-};
(-4,2)*{}; (-6,2)*{} **\dir{-};
(-6,2)*{}; (-6,10)*{} **\dir{-};
(-6,10)*{}; (-4,10)*{} **\dir{-};
(-4,10)*{}; (-4,2)*{} **\dir{-};
(4,-2)*{}; (6,-2)*{} **\dir{-};
(6,-2)*{}; (6,-10)*{} **\dir{-};
(6,-10)*{}; (4,-10)*{} **\dir{-};
(4,-10)*{}; (4,-2)*{} **\dir{-};
(-4,-2)*{}; (-6,-2)*{} **\dir{-};
(-6,-2)*{}; (-6,-10)*{} **\dir{-};
(-6,-10)*{}; (-4,-10)*{} **\dir{-};
(-4,-10)*{}; (-4,-2)*{} **\dir{-};
(6,6)*{}; (10,6)*{} **\dir{-}?(1)*\dir{>};
(6,-6)*{}; (10,-6)*{} **\dir{-};
(-6,-6)*{}; (-10,-6)*{} **\dir{-}?(1)*\dir{>};
(-6,6)*{}; (-10,6)*{} **\dir{-};
(-4,8)*{}; (4,8)*{} **\dir{-};
(-4,-8)*{}; (4,-8)*{} **\dir{-};
(4,4)*{}; (4,-4)*{} **\crv{(0,0)};
(-4,4)*{}; (-4,-4)*{} **\crv{(0,0)};
(0,0)*{_k};
(0,10)*{_{m-k}};
(0,-6)*{_{n-k}};
(12,6)*{_m};
(12,-6)*{_n};
\endxy.
\end{equation}
\end{prop}
From these formulas one can show that a projector is the sum of a lone identity web, 
denoted $\Id_w$ for the duration, in quantum degree zero 
with a $\C(q)$-linear combination of 
non-identity webs. The next proposition, which follows from Proposition \ref{projrecur} and the definition of 
the projectors, gives a characterization of $P_w$ that we will eventually categorify.
\begin{prop}\label{charproj}
The following properties characterize $P_w$.
\begin{enumerate}
\item $P_w = \Id_w + \sum_{i=1}^{r} f_i(q) \cdot W_i$ with $f_i \in \C(q)$ and where $W_i \in \Hom_{\bullet}(w,w) \smallsetminus \Id_w$ are 
non-elliptic webs.
\item If $\wt(v) < \wt(w)$ then $P_w\bullet W_1 = 0$ 
for any $W_1 \in \Hom_{\bullet}(w,v)$ and $W_2\bullet P_w = 0$ for any $W_2 \in \Hom_{\bullet}(v,w)$.
\item $P_w \bullet P_w = P_w$.
\end{enumerate} 
\end{prop}
In fact, one can show that the third property follows from the first two, but perhaps a better way to view this proposition is 
that the latter two properties characterize the projector, with the first property serving as a non-degeneracy condition (indeed, the 
zero morphism is the only other morphism satisfying the second two properties).

If the projector is segregated, Proposition \ref{projrecur} shows that the non-identity webs in the sum
take the form $V_1 \bullet W \bullet V_2$ where $V_1$ and $V_2$ are (the tensor product of identity webs with) 
$U$-webs or $Y$-webs, and $W$ is an arbitrary web.
This observation, together with the semisimplicity of $\Rep \mathcal{U}_q(\mathfrak{sl}_3)$, implies
that the second defining property above can be replaced by the following in the case of a segregated projector:
\begin{enumerate}
\item[(2')] $P_w$ annihilates $Y$-webs and $U$-webs (when two of the boundary points are attached to $P_w$).
\end{enumerate}
This can also be deduced from the following result of Kuperberg which shall be used in the sequel.
\begin{prop}[\cite{Kuperberg}]\label{Kuprop}
If $w$ is a segregated word and $v$ is a word of lower or incomparable weight, then any non-elliptic web in 
$\Hom_{\bullet}(w,v)$ factors through a $Y$-web or a $U$-web, i.e. has a $Y$-web or $U$-web with two of its boundary 
points attached to $w$.
\end{prop}
A similar result holds for $\Hom_{\bullet}(v,w)$ by taking duals.

\subsection{$\mathfrak{sl}_3$ knot invariants}
In \cite{Kuperberg}, Kuperberg introduced skein relations for the $\mathfrak{sl}_3$ spider which lead to the quantum 
$\mathfrak{sl}_3$ invariant of framed tangles (see \cite{Kuperberg2} for a detailed discussion of the combinatorial 
approach to this invariant and \cite{ReshetikhinTuraev1} for the original construction using quantum groups). We will 
use the convention for these relations which is categorified in \cite{Khovanov3}:
\begin{align*}
\left\langle \
\xy 
(-5,-5)*{};(5,5)*{} **\dir{-}?(1)*\dir{>}  \POS?(.5)*{\hole}="h1";
(-5,5)*{}; "h1" **\dir{-}?(0)*\dir{<};
(5,-5)*{}; "h1" **\dir{-};
\endxy \ \right\rangle &= q^2
\xy
(-5,-5)*{};(-5,5)*{} **\crv{(0,0)}?(1)*\dir{>};
(5,-5)*{};(5,5)*{} **\crv{(0,0)}?(1)*\dir{>};
\endxy - q^3
\xy
(0,3)*{};(-5,5)*{} **\dir{-}?(.6)*\dir{>};
(0,3)*{};(5,5)*{} **\dir{-}?(.6)*\dir{>};
(0,-3)*{};(0,3)*{} **\dir{-}?(.4)*\dir{<};
(-5,-5)*{};(0,-3)*{} **\dir{-}?(.6)*\dir{>};
(5,-5)*{};(0,-3)*{} **\dir{-}?(.6)*\dir{>};
\endxy \\
\left\langle \
\xy 
(-5,5)*{};(5,-5)*{} **\dir{-}?(0)*\dir{<}  \POS?(.5)*{\hole}="h1";
(-5,-5)*{}; "h1" **\dir{-};
(5,5)*{}; "h1" **\dir{-}?(0)*\dir{<};
\endxy \ \right\rangle &= q^{-2}
\xy
(-5,-5)*{};(-5,5)*{} **\crv{(0,0)}?(1)*\dir{>};
(5,-5)*{};(5,5)*{} **\crv{(0,0)}?(1)*\dir{>};
\endxy - q^{-3}
\xy
(0,3)*{};(-5,5)*{} **\dir{-}?(.6)*\dir{>};
(0,3)*{};(5,5)*{} **\dir{-}?(.6)*\dir{>};
(0,-3)*{};(0,3)*{} **\dir{-}?(.4)*\dir{<};
(-5,-5)*{};(0,-3)*{} **\dir{-}?(.6)*\dir{>};
(5,-5)*{};(0,-3)*{} **\dir{-}?(.6)*\dir{>};
\endxy.
\end{align*}
Using these relations, we can view (framed) tangles as morphisms in $\mathcal{S}$. In particular, links are morphisms in 
$\Hom_{\bullet}(\emptyset,\emptyset)\cong \C(q)$ so this gives a $\C(q)$-valued invariant of (framed) links. In fact, 
the local and skein relations show that this link invariant is $\Z[q^{-1},q]$-valued. 
Moreover, one can check that using the 
convention above for the skein relations actually leads to an invariant which is independent of framing. 
This does not, however, give an invariant of tangled webs (webs with crossings) since we have the relation
\begin{equation}\label{knottedspider}
\left\langle
\xy
(-2,-5)*{};(-2,5)*{} **\crv{(-2,0)}?(1)*\dir{>} \POS?(.25)*{\hole}="h1" \POS?(.75)*{\hole}="h2";
(-5,-5)*{};"h1" **\dir{-}?(0)*\dir{<};
(-5,5)*{};"h2" **\dir{-}?(0)*\dir{<};
(1,0)*{};"h1" **\dir{-};
(1,0)*{};"h2" **\dir{-};
(1,0)*{}; (6,0)*{} **\dir{-}?(1)*\dir{>};
\endxy \right\rangle = q^8
\left\langle
\xy
(2,-5);(2,5) **\dir{-}?(1)*\dir{>} \POS?(.5)*{\hole}="h1";
(-1,0); "h1" **\dir{-};
"h1"; (5,0) **\dir{-}?(1)*\dir{>};
(-5,5);(-1,0) **\dir{-}?(0)*\dir{<};
(-5,-5);(-1,0) **\dir{-}?(0)*\dir{<};
(7,0)*{};
\endxy \right\rangle .
\end{equation}
There are similar corrections of $q^8$ for other orientations of this diagram with the factor of $q^8$ always 
appearing on the side of the equation with smaller writhe.

Using the $\mathfrak{sl}_3$ projectors we can extend this invariant to give a combinatorial description of the 
$\mathfrak{sl}_3$ Reshetikhin-Turaev invariant of framed tangles, also known as the colored $\mathfrak{sl}_3$ 
invariant. This is an invariant of framed tangles with each component labeled by a
finite dimensional irreducible representation of 
$\mathfrak{sl}_3$. To compute this invariant $\langle T \rangle_{(w_1,\ldots,w_r)}$ for an $r$-component tangle $T$, 
consider any word $w_{t}$ corresponding to the highest weight of the 
irreducible representation labeling the component $t$. Take the cable of the tangle corresponding to the tangle's framing 
with strands directed according to $w_{t}$, inserting the relevant projector somewhere 
along the component. Finally, use the skein relations to evaluate the (sum of) tangled webs.

\begin{ex}
Since we have
\[
\xy 
(-6,2); (-1,2) **\dir{-};
(-6,-2); (-1,-2) **\dir{-}**\dir{-}?(0)*\dir{<};
(-1,-4)*{}; (-1,4)*{} **\dir{-};
(-1,-4)*{}; (1,-4)*{} **\dir{-};
(-1,4)*{}; (1,4)*{} **\dir{-};
(1,-4)*{}; (1,4)*{} **\dir{-};
(1,2); (6,2) **\dir{-}?(1)*\dir{>};
(1,-2); (6,-2) **\dir{-};
\endxy =
\xy
(-3,3); (3,3) **\crv{(0,1)}?(1)*\dir{>};
(3,-3); (-3,-3) **\crv{(0,-1)}?(1)*\dir{>};
\endxy - \frac{1}{[3]} \
\xy
(3,-3); (3,3) **\crv{(1,0)}?(1)*\dir{>};
(-3,3); (-3,-3) **\crv{(-1,0)}?(1)*\dir{>};
\endxy
\]
we find that
\[
\left\langle
\xy
(4,5)*{_{0}};
(0,0)*\xycircle(5,5){-};
(5,0); (5,-0.1) **\dir{-}?(1)*\dir{>};
\endxy
\right\rangle_{(+-)}
= \ \
\xy
(-10,-1)*{}="a";
(-10,1)*{}="b";
(-2,-1)*{}="c";
(-2,1)*{}="d";
"a";"b" **\dir{-};
"b";"d" **\dir{-};
"d";"c" **\dir{-};
"c";"a" **\dir{-};
(-8,-1)*{}="e";
(-8,1)*{}="f";
(-4,-1)*{}="g";
(-4,1)*{}="h";
(8,0)*{}="p";
(4,0)*{}="q";
"e";"p" **\crv{(-8,-10)&(8,-10)};
"f";"p" **\crv{(-8,10)&(8,10)}?(1)*\dir{>};
"g";"q" **\crv{(-4,-6)&(4,-6)};
"h";"q" **\crv{(-4,6)&(4,6)}?(1)*\dir{<};
\endxy \ \ = [3]^2-1 .
\]
\end{ex}
It is possible to show that for each labeling this invariant depends only on the regular isotopy class of the tangle and 
in particular does not depend on where we place the projector on each component. Again, see \cite{ReshetikhinTuraev1} for 
the original construction of this invariant.

\subsection{$\mathfrak{sl}_3$ knot homology}
Using the cohomology rings of projective space and flag varieties and certain singular surfaces called foams, Khovanov 
constructed a categorification of the quantum $\mathfrak{sl}_3$ knot invariant in \cite{Khovanov3}. This construction 
gives a bigraded homology theory for links from which the quantum $\mathfrak{sl}_3$ link invariant can be obtained by 
taking the graded Euler characteristic. In \cite{MackaayVaz}, Mackaay and Vaz gave a geometric reformulation of this 
theory in the spirit of \cite{BarNatan2}, which was later refined by Morrison and Nieh in \cite{MorrisonNieh} to an 
invariant of tangles. This latter theory is the setting for our categorification.

We now briefly outline Morrison and Nieh's construction, referring the reader to their work for complete details. The 
invariant 
takes values in the homotopy category of bounded complexes over a graded, additive category 
$\mathcal{F}$ of webs and surfaces 
with singularities, denoted $\K^b(\mathcal{F})$ for the duration. In more detail, the objects of 
the category\footnote{This category is denoted $\mathbf{Mat}(\mathcal{C}ob(\mathfrak{su}_3))$ in Morrison and 
Nieh's work.} $\mathcal{F}$ are formal direct sums of $q$-graded webs and morphisms between webs are matrices 
of $\C$-linear 
combinations of isotopy classes of degree-zero foams - surfaces with singular arcs which locally look like the product of the 
letter $Y$ and an interval - having the appropriate webs as boundary. The degree of a foam $F:q^{k_1}W_1 \to q^{k_2}W_2$ is given 
by
\[
\deg(F) = 2\chi(F) - \abs{\partial} + \frac{\abs{V}}{2} + k_2 - k_1
\]
where $\chi$ is the Euler characteristic, $\partial$ is the boundary of $W_1$ (or $W_2$ - they agree), and $V$ is the 
set of trivalent vertices in $W_1 \coprod W_2$.
Certain (degree homogeneous) local relations are imposed on these foams, see \cite{MorrisonNieh}.

Morrison and Nieh emphasize the fact that $\mathcal{F}$ has the structure of a canopolis; 
informally, this is a planar algebra enriched over $\mathcal{C}at$, the category of categories. This setting is 
appropriate since they view $\mathcal{S}$ as a planar algebra, using the pivotal structure to ignore the distinction 
between domain and codomain. Since we view $\mathcal{S}$ 
as a pivotal category, we will consider $\mathcal{F}$ as a (weak) tensor $2$-category. Objects are words, $1$-morphisms 
are formal direct sums of $q$-graded webs, and $2$-morphisms are matrices of isotopy classes of foams. We will 
denote vertical 
composition, matrix multiplication via gluing of foams along webs, by $\circ$ and horizontal composition, gluing of 
webs along their boundaries, by $\bullet$. Tensor product is defined in a similar fashion as for $\mathcal{S}$ and will 
be denoted as before by $\otimes$. This tensor $2$-categorical structure naturally extends to $\K^b(\mathcal{F})$, 
which is constructed by taking the homotopy category of complexes in each $\Hom_{\bullet}$-category. Horizontal composition 
and tensor product are defined by taking the total complex of the corresponding double complex, similar to the 
construction of the tensor product of complexes of abelian groups.

The invariant of a tangle $T$, denoted $\llbracket T \rrbracket$, is given on crossings by
\begin{equation}\label{zip}
\left\llbracket \ \xy 
(-5,-5)*{};(5,5)*{} **\dir{-}?(1)*\dir{>}  \POS?(.5)*{\hole}="h1";
(-5,5)*{}; "h1" **\dir{-}?(0)*\dir{<};
(5,-5)*{}; "h1" **\dir{-};
\endxy \
\right\rrbracket
= \left(
\xymatrix{q^2 \underline{\xy
(-5,-5)*{};(-5,5)*{} **\crv{(0,0)}?(1)*\dir{>};
(5,-5)*{};(5,5)*{} **\crv{(0,0)}?(1)*\dir{>};
(-5,-7)*{};  
\endxy } \ar[rr]^{\xy
(-6,1)*{}="a"; 
(-4,5)*{}="b"; 
(-2,3)*{}="c"; 
(2,3)*{}="d"; 
(4,1)*{}="e"; 
(6,5)*{}="f"; 
(-6,-5)*{}="g"; 
(-4,-1)*{}="h"; 
(4,-5)*{}="i"; 
(6,-1)*{}="j"; 
"a";"c" **\dir{-} \POS?(.5)*{\hole}="h1";
"b";"c" **\dir{-};
"c";"d" **\dir{-};
"d";"e" **\dir{-};
"d";"f" **\dir{-};
"b";"h1" **\dir{-};
"h1";"h" **\dir{-};
"a";"g" **\dir{-};
"e";"i" **\dir{-} \POS?(.4)*{\hole}="h2";
"f";"j" **\dir{-};
"g";"i" **\crv{(1,-2)};
"h";"h2" **\crv{(-1,-2)};
"h2";"j" **\crv{(4,-2)};
"c";"d" **\crv{(0,-1)};
\endxy}
& &
q^3 \xy
(0,3)*{};(-5,5)*{} **\dir{-}?(.6)*\dir{>};
(0,3)*{};(5,5)*{} **\dir{-}?(.6)*\dir{>};
(0,-3)*{};(0,3)*{} **\dir{-}?(.4)*\dir{<};
(-5,-5)*{};(0,-3)*{} **\dir{-}?(.6)*\dir{>};
(5,-5)*{};(0,-3)*{} **\dir{-}?(.6)*\dir{>};
\endxy}
\right)
\end{equation}
\begin{equation}\label{unzip}
\left\llbracket \ \xy 
(-5,5)*{};(5,-5)*{} **\dir{-}?(0)*\dir{<}  \POS?(.5)*{\hole}="h1";
(-5,-5)*{}; "h1" **\dir{-};
(5,5)*{}; "h1" **\dir{-}?(0)*\dir{<};
\endxy \
\right\rrbracket
= \left(
\xymatrix{q^{-3} \xy
(0,3)*{};(-5,5)*{} **\dir{-}?(.6)*\dir{>};
(0,3)*{};(5,5)*{} **\dir{-}?(.6)*\dir{>};
(0,-3)*{};(0,3)*{} **\dir{-}?(.4)*\dir{<};
(-5,-5)*{};(0,-3)*{} **\dir{-}?(.6)*\dir{>};
(5,-5)*{};(0,-3)*{} **\dir{-}?(.6)*\dir{>};
\endxy \ar[rr]^{\xy
(-6,-5)*{}="a"; 
(-4,-1)*{}="b"; 
(-2,-3)*{}="c"; 
(2,-3)*{}="d"; 
(5,-5)*{}="e"; 
(8,-1)*{}="f"; 
(-6,1)*{}="g"; 
(-4,5)*{}="h"; 
(5,1)*{}="i"; 
(8,5)*{}="j"; 
"a";"c" **\dir{-};
"b";"c" **\dir{-};
"c";"d" **\dir{-};
"d";"e" **\dir{-};
"e";"i" **\dir{-} \POS?(.5)*{\hole}="h2";
"d";"h2" **\dir{-};
"h2";"f" **\dir{-};
"g";"i" **\crv{(1,2)} \POS?(.15)*{\hole}="h1";
"b";"h1" **\dir{-};
"h1";"h" **\dir{-};
"a";"g" **\dir{-};
"f";"j" **\dir{-};
"h";"j" **\crv{(-1,2)};
"c";"d" **\crv{(0,1)};
\endxy}
& &
q^{-2} \underline{ \xy
(-5,-5)*{};(-5,5)*{} **\crv{(0,0)}?(1)*\dir{>};
(5,-5)*{};(5,5)*{} **\crv{(0,0)}?(1)*\dir{>};
(0,-7)*{}; 
\endxy }  }
\right)
\end{equation}
and extended to all tangles using horizontal composition and tensor product. 
Note that these foams, and hence the differentials in all complexes, have degree zero.
We shall refer to the foam in equation \eqref{zip} as a zip and the foam in \eqref{unzip} as an unzip, 
often denoting these morphisms by $z$ and $u$ respectively.
Here and for the duration we will underline the term sitting in homological degree zero; if no underline is 
present then the leftmost term is assumed to be in homological degree zero.
Applying a Reidemeister move 
to a tangle changes the corresponding complex by a homotopy equivalence, so we obtain a 
$\K^b(\mathcal{F})$-valued invariant of tangles. Taking the graded Euler characteristic of the complex 
assigned to a tangle gives the quantum $\mathfrak{sl}_3$ invariant.

\subsection{Categorification}

$\mathcal{F}$ categorifies the $\mathfrak{sl}_3$ spider. To formulate this statement precisely, we must consider the 
subcategory $\mathcal{S}' \subset \mathcal{S}$ whose morphisms are $\Z[q^{-1},q]$-linear combinations of webs. 
The local relations \eqref{spidereqn1}, \eqref{spidereqn2}, and \eqref{spidereqn3} show that $\mathcal{S}'$ is indeed a 
subcategory. In \cite{MorrisonNieh}, Morrison 
and Nieh show that the graded split Grothendieck group of $\mathcal{F}$ corresponds with $\mathcal{S}'$.
In particular, categorified versions of 
equations \eqref{spidereqn1}, \eqref{spidereqn2}, and \eqref{spidereqn3} hold:
\begin{align}
\label{catspidereqn1}
\xy 0;/r.1pc/:
(-10,0)*{}; (10,0)*{} **\crv{(0,8)}?(.45)*\dir{<};
(-10,0)*{}; (10,0)*{} **\crv{(0,-8)}?(.45)*\dir{<};
(-20,0)*{}; (-10,0)*{}; **\dir{-}?(.3)*\dir{<};
(10,0)*{}; (20,0)*{}; **\dir{-}?(.3)*\dir{<};
\endxy
&\cong  q \hspace{1mm}
\xy
(-5,0)*{}; (5,0)*{} **\dir{-}?(.55)*\dir{>};
\endxy
\oplus
q^{-1} \hspace{1mm}
\xy
(-5,0)*{}; (5,0)*{} **\dir{-}?(.55)*\dir{>};
\endxy \\
\label{catspidereqn2}
\xy 0;/r.15pc/:
(-5,5)*{}="a";
(5,5)*{}="b";
(-5,-5)*{}="c";
(5,-5)*{}="d";
(-10,10)*{}="e";
(10,10)*{}="f";
(-10,-10)*{}="g";
(10,-10)*{}="h";
"a";"b" **\dir{-}?(.6)*\dir{>};
"d";"b" **\dir{-}?(.6)*\dir{>};
"a";"c" **\dir{-}?(.6)*\dir{>};
"d";"c" **\dir{-}?(.6)*\dir{>};
"a";"e" **\dir{-}?(.6)*\dir{>};
"f";"b" **\dir{-}?(.6)*\dir{>};
"g";"c" **\dir{-}?(.6)*\dir{>};
"d";"h" **\dir{-}?(.6)*\dir{>};
\endxy
&\cong
\xy 0;/r.13pc/:
(-10,10)*{}="e";
(10,10)*{}="f";
(-10,-10)*{}="g";
(10,-10)*{}="h";
"f";"e" **\crv{-}?(.55)*\dir{>};
"g";"h" **\crv{-}?(.55)*\dir{>};
\endxy \oplus
\xy 0;/r.13pc/:
(-10,10)*{}="e";
(10,10)*{}="f";
(-10,-10)*{}="g";
(10,-10)*{}="h";
"f";"h" **\crv{-}?(.55)*\dir{>};
"g";"e" **\crv{-}?(.55)*\dir{>};
\endxy\\
\label{catspidereqn3}
\xy 0;/r.15pc/:
(0,0)*\xycircle(5,5){-};
{\ar@{->}(0,5)*{};(1,5)*{}}; 
\endxy 
&\cong q^2 \emptyset \oplus q^0 \emptyset \oplus q^{-2} \emptyset
\end{align}
where $\cong$ denotes isomorphism in $\mathcal{F}$. 

The categorified version of equation $\eqref{knottedspider}$ also holds:
\begin{equation}\label{catknottedspider}
\left\llbracket
\xy
(-2,-5)*{};(-2,5)*{} **\dir{-}?(.6)*\dir{>} \POS?(.25)*{\hole}="h1" \POS?(.75)*{\hole}="h2";
(-5,-5)*{};"h1" **\dir{-}?(0)*\dir{<};
(-5,5)*{};"h2" **\dir{-}?(0)*\dir{<};
(1,0)*{};"h1" **\dir{-};
(1,0)*{};"h2" **\dir{-};
(1,0)*{}; (6,0)*{} **\dir{-}?(1)*\dir{>};
\endxy
\right\rrbracket
 \simeq \left\llbracket
\xy
(2,-5);(2,5) **\dir{-}?(1)*\dir{>} \POS?(.5)*{\hole}="h1";
(-1,0); "h1" **\dir{-};
"h1"; (5,0) **\dir{-}?(1)*\dir{>};
(-5,5);(-1,0) **\dir{-}?(0)*\dir{<};
(-5,-5);(-1,0) **\dir{-}?(0)*\dir{<};
\endxy
\right\rrbracket [2]\{8\}.
\end{equation}
Here and throughout $\simeq$ denotes homotopy equivalence, $[a]$ denotes a shift up in 
homological degree by $a$, and $\{b\}$ denotes a shift up in quantum degree by $b$. Similar relations hold 
for other orientations of \eqref{catknottedspider} and, as in the decategorified case, the shifts occur on the side of the equation with 
smaller writhe.

A priori, we have a commutative diagram 
\[
\xymatrix{\mathcal{T} \ar[drr]^-{\llbracket - \rrbracket} \ar[ddr]_-{\langle - \rangle} & & \\
& \mathcal{F} \ar@{^{(}->}[r] \ar[d] & \K^b(\mathcal{F}) \ar[d] \\
& \mathcal{S}' \ar[r] & \mathcal{K}_{\triangle}(\K^b(\mathcal{F}))}
\]
where $\mathcal{T}$ is the category of tangles and $\mathcal{K}_{\triangle}$ denotes the triangulated Grothendieck group 
of a triangulated category. The bottom map exists since it is easy to see that the split Grothendieck group of 
an additive category always maps (surjectively) to the triangulated Grothendieck group of the homotopy category 
of complexes over that category. In fact, we show in \cite{Rose1} that this map is an isomorphism, giving the 
desired diagram
\begin{equation}\label{eulereq}
\xymatrix{& & \K^b(\mathcal{F}) \ar[d]^{\chi} \\
& \mathcal{T} \ar[r]^{\langle - \rangle} \ar[ur]^{\llbracket - \rrbracket} & \mathcal{S}' \\}
\end{equation}
where the map $\chi$ is taking the `Euler characteristic', the alternating sum of terms 
in a complex viewed as an element of $\mathcal{S}'$. 
Another way of stating this result from \cite{Rose1} is 
that isomorphic complexes in $\K^b(\mathcal{F})$ 
(i.e. homotopy equivalent complexes) have the same Euler characteristic.

We will later see that a diagram similar to \eqref{eulereq} exists for an extension of the category 
$\K^b(\mathcal{F})$ containing 
certain semi-infinite complexes. This will be the natural setting for the decategorification of our categorified 
projectors.

\subsection{Summary of results}

The main result of this paper is the following:
\begin{thm}\label{1}
For each word $w$ there exists a complex $\tilde{P}_{w}$
consisting of webs in $\Hom_{\bullet}(w,w)$ supported in non-negative homological degree so that
\begin{enumerate}
\item $\Id_w$ appears only once in $\tilde{P}_w$ and does so in quantum and homological degree zero,
\item all other webs in the complex factor through words of lower weight,
 \item if $\wt(v) < \wt(w)$ then $\tilde{P}_w\bullet W_1 \simeq 0$ for any $W_1 \in \Hom_{\bullet}(w,v)$ and 
 $W_2\bullet \tilde{P}_w \simeq 0$ for any $W_2 \in \Hom_{\bullet}(v,w)$, and
 \item $\tilde{P}_w \bullet \tilde{P}_w \simeq \tilde{P}_w$.
\end{enumerate}
\end{thm}
It follows from this description that such a complex is unique up to homotopy equivalence.
We construct such complexes as the stable limit (up to homotopy) of the complexes 
\[
\left\llbracket \ \twistw{k} \right\rrbracket [k\cdot c_{-}] \{k \left(3c_{-} - 2c_{+}\right)\}
\]
as $k \to \infty$. The notation indicates that there are $k$ full twists on strands directed according to the 
word $w$ and
$c_{\pm}$ is the number of $\pm$ crossings in one twist. 
Although the complexes $\tilde{P}_w$ will be semi-infinite, we will show that it is possible to take 
their graded Euler characteristic and that the next result holds.
\begin{thm}\label{2}
$\chi(\tilde{P}_w) = P_w$.
\end{thm}
Theorems \ref{1} and \ref{2} will be proved in Section \ref{sectioncatproj}.
As a consequence of the above we obtain an alternate characterization of the $\mathfrak{sl}_3$ projectors, 
known to experts in the field. The author is unaware of a proof appearing in the literature.

\begin{cor}
Let $w$ be a word, then
\[
\xy
(-8,0); (-1,0) **\dir{-};
(-1,-6)*{}; (-1,6)*{} **\dir{-};
(-1,-6)*{}; (1,-6)*{} **\dir{-};
(-1,6)*{}; (1,6)*{} **\dir{-};
(1,-6)*{}; (1,6)*{} **\dir{-};
(1,0); (8,0) **\dir{-};
(10,0)*{_{w}};
\endxy = 
\lim_{k\to\infty} q^{k\left(3c_{-}-2c_{+}\right)} \left\langle \ \twistw{k} \right\rangle \ .
\]
\end{cor}
This limit is a finite sum of webs with coefficients 
in $\Z[q^{-1},q]]$ and corresponds to the left hand side using the inclusion of coefficients 
$\C(q)\hookrightarrow \C[q^{-1},q]]$.

Finally, we shall use the categorified projectors $\tilde{P}_w$ to give a categorification of the $\mathfrak{sl}_3$ 
Reshetikhin-Turaev invariants of framed tangles. Let $\K^+(\mathcal{F})$ denote the homotopy category of bounded 
below complexes in $\mathcal{F}$.
\begin{thm}\label{3}
Let $T$ be an $r$-component framed tangle and let $w_1,\dots,w_r$ be words labeling the components of $T$. 
There exists a complex $\left\llbracket T \right\rrbracket_{(w_1,\dots,w_r)}$ in $\K^+(\mathcal{F})$, 
invariant up to homotopy under regular isotopy, which gives a categorification of the $\mathfrak{sl}_3$ 
Reshetikhin-Turaev invariant, i.e. the diagram
\[
\xymatrix{& & & \K^+(\mathcal{F}) \ar[d]^{\chi} \\
& \mathcal{T} \ar[rr]_{\langle - \rangle_{(w_1,\dots,w_r)}} \ar[urr]^{\llbracket - \rrbracket_{(w_1,\dots,w_r)}} 
& & \mathcal{S} \\ }
\]
commutes.
\end{thm}
This result is proved in Section \ref{sectioninvt}.

\section{Homological algebra}\label{homoalg}

In this section we present the requisite homological algebra for our results. The first subsection contains standard 
results from homological algebra. The following two subsections present the 
`calculus' of chain complexes. Almost all of
these results (or rather the dual statements) are taken from \cite{Rozansky}, but we repeat them in the interest of
giving a self contained treatment and in order to provide some proofs omitted there.

\subsection{Standard results}\label{homoalg1}
Let $\mathcal{A}$ be an
additive category. We will consider both the category $\Ch(\mathcal{A})$ of (cochain) complexes of objects in
$\mathcal{A}$ and the category $\K(\mathcal{A})$, the homotopy category of $\Ch(\mathcal{A})$. 
We will use the superscripts $b$ and $+$ to denote the full subcategories 
of these categories consisting of bounded and bounded below complexes. For instance, $\K^+(\mathcal{A})$ denotes the 
homotopy category of bounded below complexes in $\mathcal{A}$. We will use $\cong$ to indicate isomorphism in 
$\Ch(\mathcal{A})$ 
and $\simeq$ to denote isomorphism in $\K(\mathcal{A})$, that is, homotopy equivalence.

The first result, a technical tool from \cite{BarNatan3}, concerns Gaussian elimination homotopy equivalences. 
Such homotopy equivalences are ubiquitous in the study of Khovanov homology and its generalizations.
\begin{prop}[Gaussian elimination]\label{GE}
Let
\[
\xymatrix{ &\cdots \ar[r] &A \ar[r]^-{\left(\begin{smallmatrix} \ast \\ \alpha \end{smallmatrix}\right)} 
&B\oplus C \ar[rr]^-{\left(\begin{smallmatrix} \psi & \beta \\ \gamma & \delta \end{smallmatrix}\right)} 
&
&D\oplus E \ar[r]^-{\left(\begin{smallmatrix} \ast & \epsilon \end{smallmatrix}\right)} 
&F \ar[r] &\cdots
}
\]
be a complex in $\Ch(\mathcal{A})$ where $\psi:B \to D$ is an isomorphism, then this complex is 
homotopy equivalent to the following complex:
\[
\xymatrix{
&\cdots \ar[r] &A \ar[r]^{\alpha} &C \ar[rr]^{\delta - \gamma \circ \psi^{-1} \circ \beta} 
& &E \ar[r]^{\epsilon} &F \ar[r] &\cdots
} \ \ .
\]
Moreover, if $\mathcal{A}$ is graded and the differentials in the complex are degree $0$ then so is the 
homotopy equivalence.
\end{prop}

Given a complex $(A^{\cdot},d_A)$ in $\Ch(\mathcal{A})$, we denote by $A[n]^{\cdot}$ 
the complex shifted up by the integer $n$, that is the complex with
\[
A[n]^i = A^{i-n}
\]
and with differential given by $(-1)^n d_A$. If $f:A^{\cdot}\to B^{\cdot}$ is a chain map, 
the complex $\cone(f)$ is given by
\[
\cone(f)^i = A^{i+1} \oplus B^i
\]
with differential $\left(\begin{smallmatrix} -d_A & 0 \\ -f & d_B \end{smallmatrix} \right)$. It is a standard fact, 
proved in \cite{Spanier} for the case of abelian groups,
that the cone of a chain map detects if the map is a homotopy equivalence. The proof described there carries over to 
arbitrary additive categories

\begin{prop}\label{conenull}
A chain map $\varphi$ is a homotopy equivalence iff $\cone(\varphi)\simeq 0$.
\end{prop}

Recall now that the category $\K(\mathcal{A})$ is triangulated, with 
distinguished triangles given by those isomorphic to triangles of the form
\[
\xymatrix{
A^{\cdot} \ar[r]^-{f} &B^{\cdot} \ar[r]^-{\iota} &\cone(f) \ar[r]^-{\delta} &A[-1]^{\cdot}  
}.
\]
Using the above we can deduce the next result.
\begin{prop}\label{conehtpy}
Let $\varphi$ and $\psi$ be homotopy equivalences and $\alpha$ be a chain map, then 
$\cone(\varphi \circ \alpha \circ \psi) \simeq \cone(\alpha)$.
\end{prop}
\begin{proof}
Given chain maps $f$ and $g$, there is a homotopy equivalence
\[
\cone\left(\xymatrix{\cone(f) \ar[rr]^-{\left(\begin{smallmatrix} id & 0 \\ 0 & g\end{smallmatrix}\right)} 
& &\cone(g \circ f)}\right) \simeq \cone(g) 
\]
so there is a distinguished triangle
\[
\xymatrix{
\cone(f) \ar[r] &\cone(g\circ f) \ar[r] &\cone(g) \ar[r] &\cone(f)[-1]  
}.
\]
Considering the rotations of this triangle, the result follows from Proposition \ref{conenull} assuming in turn 
that $f$ or $g$ is a homotopy equivalence.
\end{proof}

Let $\{A^{i,j}\}$ be a double complex. By convention, the horizontal $d_h$ and vertical $d_v$ 
differentials anti-commute. Given a double complex we can obtain an element in $\Ch(\Ch(\mathcal{A}))$ by negating 
the differentials in every other row, and vice-versa. We will use this trick to show the following.

\begin{prop}[Replacement]\label{replacement}
Let $\{A^{i,j}\}$ be a double complex with $0 \leq i \leq \infty$ and $0 \leq j \leq m$ (a triply-bounded 
double complex). Suppose that for each $j$ there exist complexes $D^{\cdot,j}$ and
homotopy equivalences $\varphi_j: A^{\cdot,j} \simeq D^{\cdot,j}$, 
then $\Tot(\{A^{i,j}\})$ is homotopy equivalent to a complex $D_m$ which 
has $\oplus_{i+j=k} D^{i,j}$ in homological degree $k$.
\end{prop}

\begin{proof}
We proceed via induction on $m$. The case $m=0$ is obvious. We will show the $m=1$ case as this informs the 
proof of the general case. We consider a double complex $\{A^{i,j}\}$ of the form
\begin{equation}\label{m=1dc}
\xymatrix{A^{0,1} \ar[r]^{d_h} &A^{1,1} \ar[r]^{d_h} &A^{2,1} \ar[r]^{d_h} &\cdots \\
A^{0,0} \ar[u]^{d_v} \ar[r]^{d_h} &A^{1,0} \ar[u]^{d_v} \ar[r]^{d_h} &A^{2,0} \ar[u]^{d_v} 
\ar[r]^{d_h} &\cdots
}
\end{equation}
where each square anti-commutes. We find that
\begin{equation}\label{m=1tot}
\Tot(\{A^{i,j}\}) =
\xymatrix{ \underline{A^{0,0}} \ar[r]^-{\left(\begin{smallmatrix}d_h \\ d_v \end{smallmatrix}\right)}
&A^{1,0} \oplus A^{0,1} \ar[rr]^-{\left(\begin{smallmatrix} d_h & 0 \\ d_v & d_h \end{smallmatrix}\right)}
& & A^{2,0} \oplus A^{1,1} \ar[r]^-{\left(\begin{smallmatrix} d_h & 0 \\ d_v & d_h \end{smallmatrix}\right)}
&  \cdots
}.
\end{equation}
Negating the top row of equation \eqref{m=1dc} to view $d_v$ as a chain map between the complexes 
$(A^{\cdot,0},d_h)$ and $(A^{\cdot,1},-d_h)$ we find that $\Tot(\{A^{i,j}\}) = \cone(d_v)[1]$.
Consider the composition $\varphi_1 \circ d_v \circ \varphi_0^{-1}: D^{\cdot,0} \to D^{\cdot,1}$. 
Proposition \ref{conehtpy} shows that
\[
\cone(\varphi_1 \circ d_v \circ \varphi_0^{-1}) \simeq \cone(d_v)
\]
which gives the result since the degree $k$ term of $\cone(\varphi_1 \circ d_v \circ \varphi_0^{-1})[1]$ 
is $D^{k,0} \oplus D^{k-1,1}$.

We now prove the general case. Let $\{A^{i,j}\}_{j\leq m+1}$ be a double complex with $0
\leq i \leq \infty$ and $0 \leq j \leq m+1$ and let $\{A^{i,j}\}_{j\leq m}$ be the double complex obtained by 
truncating the last row of $\{A^{i,j}\}_{j\leq m+1}$. We find that
\[
\Tot(\{A^{i,j}\}_{j\leq m+1}) \cong \cone\left(
\xymatrix{\Tot(\{A^{i,j}\}_{j\leq m})\ar[r]^-{d_v 
} &A^{\cdot,m+1}}[m]\right)[1]
\]
where we have negated the differential on $A^{\cdot,m+1}[m]$ in the case that $m$ is even to view $d_v$ as 
a chain map (recall how $[-]$ acts on differentials). By induction, there exists a homotopy equivalence 
$\psi:D_m \to \Tot(\{A^{i,j}\}_{j\leq m})$ and the result follows from 
\[
\cone(\varphi_{m+1}[m] \circ d_v \circ \psi) \simeq \cone(d_v)
\]
as above.
\end{proof}

We will typically apply this result to double complexes of the form $\{A^i \bullet B^j\}$ where 
$A^{\cdot}$ and $B^{\cdot}$ are complexes in $\Ch(\mathcal{F})$ and  the complexes $A^i \bullet B^{\cdot}$ 
can be simplified using Gaussian elimination.

The final result we shall need describes how tensor products interact with cones and homotopy equivalence.

\begin{prop}\label{tensorcone}
Suppose $\mathcal{A}$ is a tensor category, then
\[
\cone\left(\xymatrix{A^{\cdot} \ar[r]^-{f} &B^{\cdot}}\right)\otimes C^{\cdot} =
\cone\left(\xymatrix{A^{\cdot}\otimes C^{\cdot} \ar[rr]^-{f\otimes id_C} & &B^{\cdot}\otimes C^{\cdot}}\right).
\]
and if $A^{\cdot} \simeq B^{\cdot}$ then $A^{\cdot} \otimes C^{\cdot} \simeq B^{\cdot} \otimes C^{\cdot}$.
\end{prop}

Similar results holds for other operations that behave like the tensor product of complexes. In particular, we will 
use the analogous result for horizontal composition in $\Ch(\mathcal{F})$.

\subsection{Homological calculus in $\Ch(\mathcal{A})$}\label{homoalg2}

In this section, we study the calculus of chain complexes in $\Ch(\mathcal{A})$. 
Let $A^{\cdot}$ be such a complex.

\begin{defn} The \emph{$k^{th}$ truncation} of $A^{\cdot}$ is the complex $t_{\leq k}A^{\cdot}$ with
\[ t_{\leq k}A^i = \left\{
     \begin{array}{ll}
       A^i & \text{ if \ \ }  i\leq k\\
       0 & \text{ if \ \ }   i > k
     \end{array}
   \right.\]
and the obvious differentials.
\end{defn}

Extend $t_{\leq k}$ to a functor $\Ch(\mathcal{A}) \rightarrow \Ch(\mathcal{A})$ by defining
$t_{\leq k}f$ for a chain map $f:A^{\cdot} \to B^{\cdot}$ to be the obvious map $t_{\leq k}A^{\cdot} \to
t_{\leq k}B^{\cdot}$. We will say that a chain complex $A^{\cdot}$ is $\Oh^h(k)$ if $A^i = 0$ for all $i<k$.
Equivalently, $A^{\cdot}$ is $\Oh^h(k)$ if and only if $t_{\leq (k-1)}A^{\cdot} = 0$.

\begin{defn}
The \emph{isomorphism order} of a chain map $f:A^{\cdot} \to B^{\cdot}$ is given by
\[ \abs{f}_{\cong} = \sup_k \{k \ | \ t_{\leq k}f \text{ is an isomorphism}\}\]
\end{defn}
Clearly, a chain map $f$ is a chain isomorphism if and only if $\abs{f}_{\cong} = \infty$.

\begin{prop}\label{triangleprop}
Let $A^{\cdot}\overset{f}{\longrightarrow}B^{\cdot}\overset{\iota}{\longrightarrow}\cone(f)
\overset{\delta}{\longrightarrow}A[-1]^{\cdot}$ be a distinguished triangle and suppose 
$B^{\cdot}$ is $\Oh^h(k)$,
then $\abs{\delta}_{\cong} \geq k-1$.
\end{prop}

\begin{proof}
The result follows since $\delta$ is given by the diagram
\[\xymatrix{\cdots \ar[r] & A^{i+1} \oplus B^i \ar[r] \ar[d]^{=} &\cdots \ar[r] &A^k \oplus
B^{k-1} \ar[r] \ar[d]^{=} &A^{k+1}\oplus B^k \ar[r] \ar[d]^{?} &\cdots\\ 
\cdots \ar[r] & A^{i+1} \ar[r] & \cdots \ar[r] & A^k \ar[r] & A^{k+1} \ar[r] & \cdots\\} 
\]
noticing that the degree $i$ term of $\cone(f)$ is $A^{i+1} \oplus B^i$.
\end{proof}

Define an inverse system in $\Ch(\mathcal{A})$ as a sequence of chain complexes linked by chain maps:
\[\mathbf{A} = \left(A_0^{\cdot} \overset{f_0}{\longleftarrow} A_1^{\cdot} \overset{f_1}{\longleftarrow}
\cdots\right).\]

\begin{defn}\label{defstabilizing}
An inverse system $\mathbf{A}$ is \emph{stabilizing} if $\lim_{l\to\infty}\abs{f_l}_{\cong} = \infty$.
\end{defn}

\begin{defn}\label{defchainlimit}
An inverse system has a \emph{$\Ch$-limit}, denoted $\lim_{\Ch}\mathbf{A}$, if there exist chain maps
$\lim_{\Ch}\mathbf{A} \overset{\tilde{f_l}}{\longrightarrow}A_l^{\cdot}$ so that
\begin{equation}\label{chlimitdiag}
\xymatrix{&\lim_{\Ch}\mathbf{A} \ar[ld]_{\tilde{f}_{l-1}} \ar[dr]^{\tilde{f}_l}& \\ A_{l-1}^{\cdot} 
&& A_l^{\cdot} \ar[ll]_{f_{l-1}} \\}
\end{equation}
commutes and $\lim_{l\to\infty}\abs{\tilde{f}_l}_{\cong} = \infty$.
\end{defn}

In fact, these preceding two notions coincide.

\begin{thm}\label{chlimiffstab}
An inverse system $\mathbf{A}$ has a $\Ch$-limit if and only if it is stabilizing. Such a limit is unique up 
to isomorphism.
\end{thm}

\begin{proof}
First, assume that $\mathbf{A} = \left(A_0^{\cdot} \overset{f_0}{\longleftarrow} A_1^{\cdot} 
\overset{f_1}{\longleftarrow} \cdots\right)$ is a stabilizing inverse system. We shall give a direct construction of 
$A^{\cdot} = \lim_{\Ch}\mathbf{A}$.

For each homological degree $k$, there exists a minimal $l(k)$ so that $f_l^k:A_{l+1}^k \to A_l^k$ is 
an isomorphism for all $l\geq l(k)$; let $A^k = A_{l(k)}^k$. Since all of the $f_l$ are chain maps, 
there is an obvious choice of boundary map $A^k \to A^{k+1}$. Indeed, the construction is demonstrated 
in the following diagram:
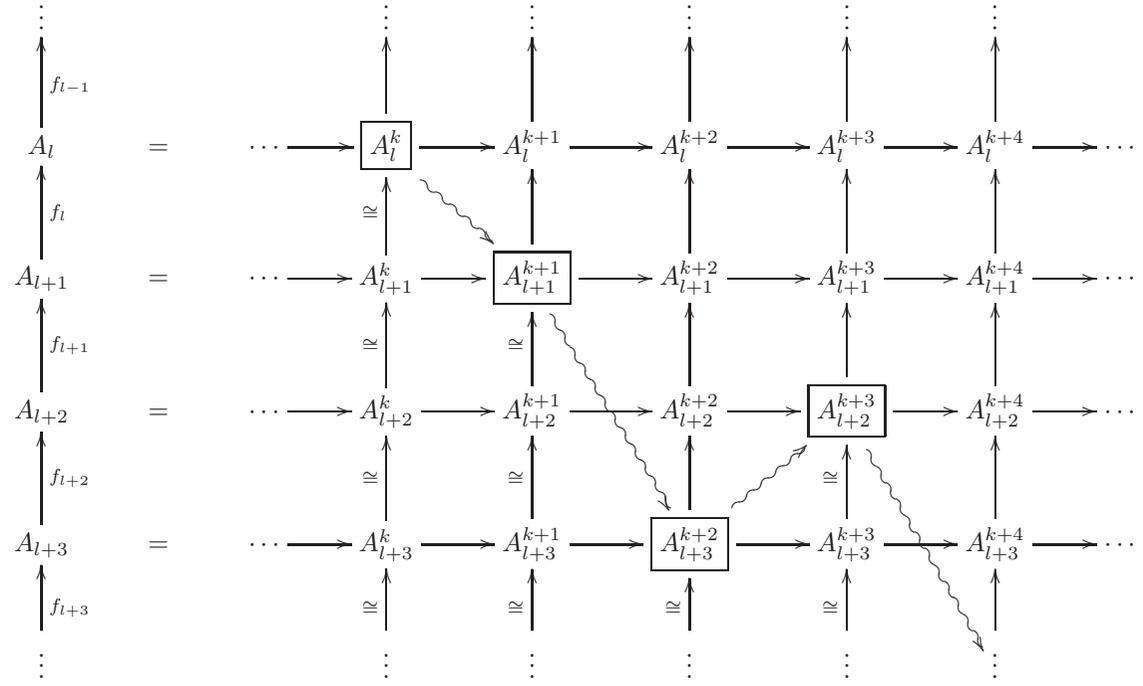
\begin{landscape}
\[\xymatrix{
 \vdots &&&\vdots & \vdots & \vdots & \vdots & \vdots & \\
A_{l} \ar[u]_{f_{l-1}} &=&\cdots \ar[r] & \boxed{A_{l}^{k}} \ar@{~>}[dr] \ar[r] \ar[u] 
& A_{l}^{k+1} \ar[r] \ar[u] & A_{l}^{k+2} \ar[r] \ar[u] & A_{l}^{k+3} \ar[r] \ar[u] 
& A_{l}^{k+4} \ar[r] \ar[u] & \cdots  \\
A_{l+1} \ar[u]_{f_{l}} &=&\cdots \ar[r] & A_{l+1}^{k} \ar[r] \ar[u]^{\cong} 
& \boxed{A_{l+1}^{k+1}} \ar@{~>}[ddr] \ar[r] \ar[u] & A_{l+1}^{k+2} \ar[r] \ar[u] 
& A_{l+1}^{k+3} \ar[r] \ar[u] & A_{l+1}^{k+4} \ar[r] \ar[u] & \cdots \\
A_{l+2} \ar[u]_{f_{l+1}} &=&\cdots \ar[r] & A_{l+2}^{k} \ar[r] \ar[u]^{\cong} 
& A_{l+2}^{k+1} \ar[r] \ar[u]^{\cong} & A_{l+2}^{k+2} \ar[r] \ar[u] 
& \boxed{A_{l+2}^{k+3}} \ar@{~>}[ddr] \ar[r] \ar[u] & A_{l+2}^{k+4} \ar[r] \ar[u] & \cdots  \\
A_{l+3} \ar[u]_{f_{l+2}} &=&\cdots \ar[r] & A_{l+3}^{k} \ar[r] \ar[u]^{\cong} 
& A_{l+3}^{k+1} \ar[r] \ar[u]^{\cong} & \boxed{A_{l+3}^{k+2}} \ar@{~>}[ur] \ar[r] \ar[u] 
& A_{l+3}^{k+3} \ar[r] \ar[u]^{\cong} & A_{l+3}^{k+4} \ar[r] \ar[u] & \cdots  \\
\vdots \ar[u]_{f_{l+3}}&&&\vdots \ar[u]^{\cong} & \vdots \ar[u]^{\cong} & \vdots \ar[u]^{\cong} 
& \vdots \ar[u]^{\cong} & \vdots \ar[u] & \\
}
\]
\end{landscape}
\noindent where the boundary maps are uniquely specified as any path between the selected 
`nodes'; commutativity of the above grid ensures that these maps are well-defined and 
square to zero. Define maps $\tilde{f}_l:A^{\cdot} \to A_l^{\cdot}$ using the obvious maps from $A^k = A_{l(k)}^k 
\longrightarrow A_l^k$, noting that these determine chain maps with respect to the differential on $A^{\cdot}$ 
described above, again by commutativity.

We now show that $A^{\cdot}$ is a $\Ch$-limit for $\mathbf{A}$. Since $\mathbf{A}$ is stabilizing, for every 
$N$ there exists $l_N$ so that $l>l_N$ implies $\abs{f_l}_{\cong} \geq N$. This in turn implies that if $l>l_N$ then 
$\abs{\tilde{f}_l}_{\cong} \geq N$, showing that $\lim_{l\to\infty}\abs{\tilde{f}_l}_{\cong} = \infty$. The result now 
follows since the diagram \eqref{chlimitdiag} commutes by construction.

Next, assume that $\mathbf{A}$ has a $\Ch$-limit. Commutativity of \eqref{chlimitdiag} shows that 
\[\abs{f_l}_{\cong} \geq \min\left(\abs{\tilde{f}_l}_{\cong} ,\abs{\tilde{f}_{l+1}}_{\cong}\right) \] so 
$\lim_{l\to\infty}\abs{f_l}_{\cong} = \infty$, i.e. $\mathbf{A}$ is stabilizing.

Finally, we show that $\Ch$-limits are unique up to isomorphism in $\Ch(\mathcal{A})$. Suppose that both 
$A^{\cdot}$ and $(A')^{\cdot}$ are $\Ch$-limits. For each homological degree $k_0$, there exists $m(k_0)$ so that 
for all $m\geq m(k_0)$ the maps $\tilde{f}_m^k$ and $\tilde{f'}_m^k$ are isomorphisms for all $k\leq k_0$. Define the 
isomorphism $(A')^{\cdot} \overset{\cong}{\longrightarrow} A^{\cdot}$ in the $k_0^{th}$ homological degree by
\[(\tilde{f}_{m(k_0)}^k)^{-1}\circ \tilde{f'}_{m(k_0)}^k : (A')^k \to A^k. \]
Commutativity of \eqref{chlimitdiag} implies that this gives a chain isomorphism, well defined independent of the choice 
of $m(k_0)$.
\end{proof}

\subsection{Homological calculus in $\K(\mathcal{A})$}\label{homoalg3}

We now describe the extension of the definitions and results from the previous section to the homotopy category 
$\K(\mathcal{A})$.

Let $A^{\cdot}$ be a complex in $\K(\mathcal{A})$. Define the homological order of $A^{\cdot}$ via
\[\abs{A^{\cdot}}_h = \sup \{k \ | \ A^{\cdot} \simeq B^{\cdot} \text{ where } 
B^{\cdot} \text{ is } \Oh^h(k) \}.\]
We think of a complex as homologically negligible if $\abs{A^{\cdot}}_h$ is large. In the same vein, we view 
complexes $A^{\cdot}$ and $B^{\cdot}$ as homologically close if there is a chain map $A^{\cdot} 
\overset{f}{\longrightarrow} B^{\cdot}$ so that $\cone(f)$ is homologically negligible. 

The next two definitions generalize the notions of stabilizing inverse system and $\Ch$-limit.

\begin{defn}\label{defcauchy}
An inverse system $\mathbf{A} = \left(A_0^{\cdot} \overset{f_0}{\longleftarrow} A_1^{\cdot} 
\overset{f_1}{\longleftarrow} \cdots\right)$ is \emph{Cauchy} if $\lim_{l\to\infty}\abs{\cone(f_l)}_h = \infty$.
\end{defn}

\begin{defn}\label{defKlim}
An inverse system $\mathbf{A} = \left(A_0^{\cdot} \overset{f_0}{\longleftarrow} A_1^{\cdot} 
\overset{f_1}{\longleftarrow} \cdots\right)$ has a \emph{$\K$-limit}, denoted $\lim_{\K}\mathbf{A}$, if 
there exist chain maps $\lim_{\K}\mathbf{A} \overset{\tilde{f_l}}{\longrightarrow}A_l^{\cdot}$ so that
\begin{equation}\label{hlimitdiag}
\xymatrix{&\lim_{\K}\mathbf{A} \ar[ld]_{\tilde{f}_{l-1}} \ar[dr]^{\tilde{f}_l}& \\ 
A_{l-1}^{\cdot} && A_l^{\cdot} \ar[ll]_{f_{l-1}} \\}
\end{equation}
commutes in $\K(\mathcal{A})$ and $\lim_{l\to\infty}\abs{\cone(\tilde{f}_l)}_h = \infty$.
\end{defn}

We now aim to state and prove the analog of Theorem \ref{chlimiffstab} in the homotopy category. Before doing so, we need 
a preparatory lemma.

\begin{lem}\label{trianglelem}
Suppose we have a commutative triangle
\[\xymatrix{A^{\cdot} \ar[d]^{f} \ar[dr]^{k} & \\
B^{\cdot} \ar[r]^{g} & C^{\cdot}} \]
in $\K(\mathcal{A})$ (i.e. $k \simeq g\circ f$), then
\[\abs{\cone(g)}_h \geq \min \{\abs{\cone(k)}_h,\abs{\cone(f)}_h-1\},\]
\[\abs{\cone(f)}_h \geq \min \{\abs{\cone(g)}_h+1,\abs{\cone(k)}_h\},\]
and
\[\abs{\cone(k)}_h \geq \min \{\abs{\cone(f)}_h,\abs{\cone(g)}_h\}.\]
\end{lem}

\begin{proof}
The result follows from the various rotations of the distinguished triangle
\[\cone(f) \to \cone(g\circ f) \to \cone(g) \to \cone(f)[-1],\]
noting that for any chain map $X^{\cdot} \overset{\alpha}{\longrightarrow} Y^{\cdot}$ we have the inequality 
$\abs{\cone(\alpha)}_h \geq \min \{\abs{Y^{\cdot}}_h,\abs{X^{\cdot}}_h-1\}$.
\end{proof}

\begin{thm}\label{klimitiffcauchy}
An inverse system $\mathbf{A}$ has a $\K$-limit if and only if it is Cauchy.
\end{thm}
The following proof is essentially taken from \cite{Rozansky}, but we reproduce the argument in our context as a 
construction contained therein will be used later.
\begin{proof}
First, assume that $\mathbf{A} = \left(A_0^{\cdot} \overset{f_0}{\longleftarrow} A_1^{\cdot} 
\overset{f_1}{\longleftarrow} \cdots\right)$ is Cauchy. This implies that there exist complexes 
$C_l^{\cdot}$ so that
\begin{enumerate}
\item $C_l[-1]^{\cdot} \simeq \cone(f_l)$ 
\item $C_l^{\cdot}$ is $\Oh^h(m_l)$ 
\item $\lim_{l\to\infty}m_l = \infty. $
\end{enumerate}
We now construct a new inverse system $\mathbf{B} = \left(B_0^{\cdot} \overset{\delta_0}{\longleftarrow} 
B_1^{\cdot} 
\overset{\delta_1}{\longleftarrow} \cdots\right)$ with $B_l^{\cdot} \simeq A_l^{\cdot}$ 
via the following procedure. 
Let $B_0^{\cdot} = A_0^{\cdot}$; to define $B_l^{\cdot}$ for $l\geq1$, consider the distinguished triangle
\begin{equation}\label{triangle1}
A_l^{\cdot}\overset{f_{l-1}}{\longrightarrow}A_{l-1}^{\cdot}
\overset{\iota_{l-1}}{\longrightarrow}\cone(f_{l-1}).
\end{equation}
Using the diagram
\[\xymatrix{
 A_{l-1}[1]^{\cdot} \ar[r] \ar[d]^{=} & \cone(f_{l-1})[1] \ar[r] \ar[d]^{\simeq}  
& A_l^{\cdot} \ar[r]^{f_{l-1}} \ar[d]^{=}  & A_{l-1}^{\cdot} \ar[d]^{=} \\
A_{l-1}[1]^{\cdot} \ar[r] \ar[d]^{\simeq} & C_{l-1}^{\cdot} \ar[r] \ar[d]^{=}  & A_l^{\cdot} \ar[r] \ar[d]^{=}  
& A_{l-1}^{\cdot} \ar[d]^{\simeq} \\
B_{l-1}[1]^{\cdot} \ar[r]^{j_{l-1}} \ar[d]^{=} & C_{l-1}^{\cdot} \ar[r] \ar[d]^{=}  
& A_l^{\cdot} \ar[r] \ar[d]^{\simeq}  & B_{l-1}^{\cdot} \ar[d]^{=} \\
B_{l-1}[1]^{\cdot} \ar[r]^{j_{l-1}}  & C_{l-1}^{\cdot} \ar[r]   & \cone(j_{l-1}) \ar[r]^{\delta_{l-1}}   
& B_{l-1}^{\cdot} \\
}
\]
in which every row is a distinguished triangle and
whose first row is obtained from rotating \eqref{triangle1}, we define $B_l^{\cdot} = \cone(j_{l-1})$. Note that indeed 
$B_l^{\cdot} \simeq A_l^{\cdot}$ and under these homotopies we have the commutative diagram
\begin{equation}\label{square1}
\xymatrix{
A_l^{\cdot} \ar[r]^{f_{l-1}} \ar[d]^{\simeq} & A_{l-1}^{\cdot} \ar[d]^{\simeq}\\
B_l^{\cdot} \ar[r]^{\delta_{l-1}} & B_{l-1}^{\cdot} \ \ \ \ \  . \\
}
\end{equation}

Consider now the inverse system $\mathbf{B} = \left(B_0^{\cdot} \overset{\delta_0}{\longleftarrow} B_1^{\cdot} 
\overset{\delta_1}{\longleftarrow} \cdots\right)$ whose objects fit into distinguished triangles
\[B_l[1]^{\cdot} \overset{j_l}{\longrightarrow} C_l^{\cdot} \longrightarrow B_{l+1}^{\cdot} 
\overset{\delta_l}{\longrightarrow} B_l^{\cdot}.\]
Since $C_l^{\cdot}$ is $\Oh^h(m_l)$, Proposition \ref{triangleprop} gives that $\abs{\delta_l}_{\cong} \geq m_l-1$; this 
in turn implies that $\mathbf{B}$ is stabilizing so it has a $\Ch$-limit 
$B^{\cdot} = \lim_{\Ch}\mathbf{B}$.

We now see that $B^{\cdot}$ gives a $\K$-limit for $\mathbf{A}$. The commutative square \ref{square1} together 
with the fact that $B^{\cdot}$ is a $\Ch$-limit gives the diagram
\begin{equation}\label{speciallimitdiagram}
 \xymatrix{
&& B^{\cdot} \ar[dll]_{\tilde{\delta}_0} \ar[dl]^{\tilde{\delta}_1} \ar[d]^{\tilde{\delta}_2} \ar[dr] & \\
B_0^{\cdot} \ar[d]^{\alpha_0}_{\simeq} & B_1^{\cdot} \ar[d]^{\alpha_1}_{\simeq} \ar[l]^{\delta_0} 
& B_2^{\cdot} \ar[d]^{\alpha_2}_{\simeq} \ar[l]^{\delta_1} & \cdots \ar[l]\\
A_0^{\cdot} & A_1^{\cdot} \ar[l]^{f_0} & A_2^{\cdot} \ar[l]^{f_1} & \cdots \ar[l]\\  
}
\end{equation}
where the maps $\alpha_{l}$ are determined by equation \eqref{square1}, 
$\lim_{l\to\infty}\abs{\tilde{\delta}_l}_{\cong} = \infty$, and the squares commute up to homotopy. Since 
$\alpha_l$ is a homotopy equivalence,
$\cone(\alpha_l\circ \tilde{\delta}_l) \simeq \cone(\tilde{\delta}_l)$ which implies 
$\abs{\cone(\alpha_l\circ \tilde{\delta}_l)}_h = \abs{\cone(\tilde{\delta}_l)}_h.$ Define 
$\tilde{f}_l:B^{\cdot} \rightarrow A_l^{\cdot}$ via $\tilde{f}_l = \alpha_l\circ \tilde{\delta}_l$; Gaussian 
elimination of complexes implies that $\abs{\cone(\tilde{f}_l)}_h \geq \abs{\tilde{f}_l}_{\cong}$ so we have
\begin{align*}
\lim_{l\to\infty} \abs{\cone(\tilde{f}_l)}_h &\geq \lim_{l\to\infty}\abs{\tilde{\delta}_l}_{\cong}\\
&= \infty \ .
\end{align*}
Since $f_l\circ \tilde{f}_{l+1} \simeq \tilde{f}_{l}$ this shows that $B^{\cdot}$ is a $\K$-limit for 
$\mathbf{A}$.

Now, suppose that $\mathbf{A}$ has a $\K$-limit. Equation \eqref{hlimitdiag} and Lemma \ref{trianglelem} give 
that \[\abs{\cone(f_l)}_h \geq \min \{\abs{\cone(\tilde{f}_{l})}_h,\abs{\cone(\tilde{f}_{l+1})}_h-1\}\] so 
\[\lim_{l\to\infty} \abs{\cone(f_l)}_h = \infty\] showing that $\mathbf{A}$ is Cauchy.
\end{proof}

Similar to the case of limits in $\Ch(\mathcal{A})$, there is also a uniqueness statement; however, we need a few 
preparatory lemmata before giving its proof. 


\begin{lem}\label{technicallem}
Let $\mathbf{A}$ be a Cauchy inverse system and let $B^{\cdot}$ be the $\K$-limit constructed in the proof of 
Theorem \ref{klimitiffcauchy}. If $A^{\cdot}$ is a complex and there are maps 
$A^{\cdot} \overset{a_i}{\longrightarrow} A_i^{\cdot}$ so that the diagram
\[\xymatrix{&& A^{\cdot} \ar[dll]_{a_0} \ar[dl]^{a_1} \ar[d]^{a_2} \ar[dr] & \\
A_0^{\cdot} & A_1^{\cdot} \ar[l]^{f_0} & A_2^{\cdot} \ar[l]^{f_1} & \cdots \ar[l]\\
}
\]
commutes in $\K(\mathcal{A})$, then there exists a map 
$A^{\cdot} \overset{h}{\longrightarrow} B^{\cdot}$ so that the triangles
\[\xymatrix{& A^{\cdot} \ar[d]^{h} \ar[dl]_{a_i} \\
A_i^{\cdot}  & B^{\cdot} \ar[l]^{\alpha_i \circ \tilde{\delta}_i} }  \]
commute in $\K(\mathcal{A})$.
\end{lem}

\begin{proof}
Consider the diagram

\[\xymatrix{& A^{\cdot} \ar[dl]_{a_0} \ar[d]_{a_1} \ar[dr]_{a_2} \ar[drr] & \\
A_0^{\cdot} & A_1^{\cdot} \ar[l]^{f_0} & A_2^{\cdot} \ar[l]^{f_1} & \cdots \ar[l]\\
& B^{\cdot}\ar[ul]^{b_0} \ar[u]^{b_1} \ar[ur]^{b_2} \ar[urr] 
}
\]
where $b_i = \alpha_i \circ \tilde{\delta}_i$ (in the notation from the proof of Theorem \ref{klimitiffcauchy}). Using 
\eqref{speciallimitdiagram} we can assume that $b_i = f_i \circ b_{i+1}$ and by throwing out terms and re-indexing we can 
suppose that $\abs{b_i}_{\cong} \geq i$. It suffices to show that such an $h: A^{\cdot} \to B^{\cdot}$ exists with the 
(remaining) triangles
\[\xymatrix{& A^{\cdot} \ar[d]^{h} \ar[dl]_{a_i} \\
A_i^{\cdot}  & B^{\cdot} \ar[l]^{b_i} } \]
commuting up to homotopy.

To this end, consider the sub-diagrams given by
\[\xymatrix{& A^{\cdot} \ar[dl]_{a_i} \ar[dr]^{a_{i+1}} \\
 A_i^{\cdot} && A_{i+1}^{\cdot} \ar[ll]^{f_i}\\
& B^{\cdot} \ar[ul]^{b_i} \ar[ur]_{b_{i+1}}.&}\\
\]
We begin by constructing maps $g_i: t_{\leq i}A^{\cdot} \to t_{\leq i}B^{\cdot}$ 
so that $g_{i+1}$ is homotopic to a 
map $\hat{g}_{i+1}$ with $t_{\leq j}\hat{g}_{i+1} = t_{\leq j}g_i$ for $j<i$ and so that the triangles
\[\xymatrix{& t_{\leq i}A^{\cdot} \ar[d]^{g_i} \ar[dl]_{t_{\leq i}a_i} \\
t_{\leq i} A_i^{\cdot}  & t_{\leq i}B^{\cdot} \ar[l]^{t_{\leq i}b_i} }  \]
commute. Define $g_i = (t_{\leq i}b_i)^{-1} \circ t_{\leq i}a_i$. Since $a_i \simeq f_i \circ a_{i+1}$ there exist maps 
$H_i^k$ so that
\begin{equation}\label{ahtpy}
a_i^k = f_i^k \circ a_{i+1}^k + d_{A_i} \circ H_i^{k-1} + H_i^k \circ d_{A}.
\end{equation}
Consider the map  $\hat{g}_{i+1}: t_{\leq i+1}A^{\cdot} \to t_{\leq i+1}B^{\cdot}$ 
defined in homological degree $k$ by
\[
\hat{g}_{i+1}^k = \left\{ \begin{array}{ll}
       g_{i+1}^{i+1} & \text{ if \ \ }  k = i+1\\
       g_{i+1}^i + d_B \circ (b_i^{i-1})^{-1} \circ H_i^{i-1}  & \text{ if \ \ }   k=i\\
       g_{i+1}^k + d_B \circ (b_i^{k-1})^{-1} \circ H_i^{k-1} + (b_i^{k})^{-1} \circ H_i^k \circ d_{A} 
        & \text{ if \ \ } k<i
       
     \end{array}
   \right.\ 
\]
which is a chain map homotopic to $g_{i+1}$. Note that for $k<i$
\begin{align*}
\hat{g}_{i+1}^k &= g_{i+1}^k + d_B \circ (b_i^{k-1})^{-1} \circ H_i^{k-1} 
+ (b_i^{k})^{-1} \circ H_i^k \circ d_{A} \\
&= (b_i^k)^{-1} \circ f_i^k \circ a_{i+1}^k + d_B \circ (b_i^{k-1})^{-1} \circ H_i^{k-1} 
+ (b_i^{k})^{-1} \circ H_i^k \circ d_{A} \\
&= (b_i^k)^{-1} \circ \left(f_i^k \circ a_{i+1}^k + d_{A_i} \circ H_i^{k-1} 
+ H_i^k \circ d_{A}\right) \\
&= (b_i^k)^{-1} \circ a_i^k \\
&= g_i^k
\end{align*}
so the $g_i$ have the desired properties.

We now define maps $h_i: t_{\leq i}A^{\cdot} \to t_{\leq i}B^{\cdot}$ with $h_i \simeq g_i$ and so that 
$t_{\leq j}h_{i+1}$ agrees with $t_{\leq j}h_i$ for all $j<i$ as follows. Let $h_0 = g_0$ and 
$h_1 = \hat{g}_1 \simeq g_1$; we will construct $h_{i+1}$ assuming that we have constructed $h_0,\dots,h_i$. Since 
$h_i \simeq g_i$ there exist maps $G_i^k$ so that 
\begin{align*}
 h_i^i &= g_i^i + d_{B}\circ G_i^{i-1}\\
h_i^k &= g_i^k + d_{B}\circ G_i^{k-1} + G_i^k\circ d_{A} \text{ \ \ if \ \ } k<i.
\end{align*}
Define $h_{i+1}$ via
\[
 h_{i+1}^k = \left\{ \begin{array}{ll}
                     g_{i+1}^{i+1} & \text {if \ \ } k=i+1\\
		     g_{i+1}^i + d_{B}\circ \left((b_i^{i-1})^{-1} \circ H_i^{i-1} + G_i^{i-1} \right) 
                      & \text {if \ \ } k=i\\
                     g_{i+1}^k + d_{B}\circ \left((b_i^{k-1})^{-1} \circ H_i^{k-1} + G_i^{k-1} \right) 
                       & \\
                     \ \ \ \ \ \  + \left((b_i^{k})^{-1} \circ H_i^k + G_i^k \right) \circ d_{A} & \text {if \ \ } k<i\\
                      
\end{array} \right.\ 
\]
and observe that $h_{i+1} \simeq g_{i+1}$ (and that in fact we take $G_i^{i-1}=0$). We also compute
\begin{align*}
 h_{i+1}^k &= g_{i+1}^k + d_{B}\circ \left((b_i^{k-1})^{-1} \circ H_i^{k-1} + G_i^{k-1} \right) 
                + \left((b_i^{k})^{-1} \circ H_i^k + G_i^k \right) \circ d_{A}\\
           &=\hat{g}_{i+1}^k + d_{B}\circ G_i^{k-1} + G_i^k \circ d_{A}\\
           &= g_i^k + d_{B}\circ G_i^{k-1} + G_i^k \circ d_{A}\\
           &= h_i^k
\end{align*}
for $k<i$, so the $h_i$ have the desired properties.

Finally, let $h:A^{\cdot} \to B^{\cdot}$ be defined as the stable limit of the maps $h_i$, i.e. $h^k = h_i^k$ for any 
$i>k$. It remains to check that $b_i\circ h \simeq a_i$ for all $i$. Observe first that we have the equalities
\begin{align*}
h^k &= \left\{ \begin{array}{ll}
       h_{k+1}^k & \text{ if \ \ }  k\geq i\\
       h_i^k  & \text{ if \ \ }   k<i \end{array} \right. \\
&= \left\{ \begin{array}{ll}
       g_{k+1}^k + d_{B}\circ (b_k^{k-1})^{-1} \circ H_k^{k-1} & \text{ if \ \ }  k\geq i\\
       g_i^k + d_B\circ G_i^{k-1} + G_i^k \circ d_A & \text{ if \ \ }   k<i \end{array} \right.
\end{align*}
and so
\begin{align*}
(b_i\circ h)^k &= \left\{ \begin{array}{ll}
       b_i^k\circ g_{k+1}^k + b_i^k\circ d_{B}\circ (b_k^{k-1})^{-1} \circ H_k^{k-1} 
          & \text{ if \ \ }  k\geq i\\
       b_i^k\circ g_i^k + b_i^k\circ d_B\circ G_i^{k-1} + b_i^k\circ G_i^k \circ d_A 
          & \text{ if \ \ }   k<i \end{array} \right.  \\
&= \left\{ \begin{array}{ll}
        b_i^k\circ (b_{k+1}^k)^{-1}\circ a_{k+1}^k & \\
        \ \ \ \ + d_{A_i}\circ b_i^{k-1}\circ (b_k^{k-1})^{-1} \circ H_k^{k-1}  & \text{ if \ \ }  k\geq i\\
        b_i^k\circ (b_i^k)^{-1}\circ a_i^k & \\
        \ \ \ \ +  d_{A_i}\circ b_i^{k-1}\circ G_i^{k-1} + b_i^k\circ G_i^k \circ d_A & \text{ if \ \ }   k<i
        \end{array} \right. \\
&= \left\{ \begin{array}{ll}
        f_i^k\circ \cdots \circ f_k^k\circ a_{k+1}^k & \\
        \ \ \ \ + d_{A_i}\circ f_i^{k-1}\circ \cdots \circ f_{k-1}^{k-1} \circ H_k^{k-1}  & \text{ if \ \ }  k\geq i\\
        a_i^k +  d_{A_i}\circ b_i^{k-1}\circ G_i^{k-1} + b_i^k\circ G_i^k \circ d_A & \text{ if \ \ }   k<i.
       \end{array} \right.
\end{align*}
Using \eqref{ahtpy} we compute
\begin{align*}
f_i^k\circ \cdots \circ f_k^k\circ a_{k+1}^k &= a_i^k 
 - d_{A_i}\circ \left(H_i^{k-1}+f_i^{k-1}\circ H_{i+1}^{k-1}
 +\cdots \right. \\
& \ \ \left. +f_i^{k-1}\circ \cdots \circ f_{k-1}^{k-1} \circ H_k^{k-1}\right) - \left(H_i^k+f_i^k\circ H_{i+1}^k+\cdots 
                                                                                                              \right. \\
& \ \  \left. +f_i^k\circ \cdots \circ f_{k-1}^k \circ H_k^k\right)\circ d_A
\end{align*}
and
\begin{equation*}
(b_i\circ h)^k = \left\{ \begin{array}{ll}
                    a_i^k - d_{A_i}\circ \left(H_i^{k-1}
                     +f_i^{k-1}\circ H_{i+1}^{k-1}+\cdots \right. & \\
                    \ \ \ \ \ \ \left.+f_i^{k-1}\circ \cdots \circ f_{k-2}^{k-1} \circ H_{k-1}^{k-1}\right) & \\
                    \ \ \ \ \ \ - \left(H_i^k+f_i^k\circ H_{i+1}^k+\cdots \right. & \\
                    \ \ \ \ \ \ \left.+f_i^k\circ \cdots \circ f_{k-1}^k \circ H_k^k\right)\circ d_A & 
                                                                                   \text{ if \ \ }  k\geq i+2 \\
                    a_i^{i+1} - d_{A_i}\circ H_i^i 
                     - \left(H_i^{i+1}+f_i^{i+1}\circ H_{i+1}^{i+1}\right)\circ d_A & \text{ if \ \ }  k= i+1\\ 
                    a_i^i - H_i^i\circ d_A & \text{ if \ \ }  k= i\\
                    a_i^{i-1} + d_{A_i}\circ b_i^{i-2}\circ G_i^{i-2} & \text{ if \ \ }  k= i-1\\
                    a_i^k + d_{A_i}\circ b_i^{k-1}\circ G_i^{k-1} 
                     + b_i^k\circ G_i^k \circ d_A & \text{ if \ \ }  k\leq i-2
                      \end{array}
                 \right.
 \end{equation*}
from which it is evident that $b_i\circ h\simeq a_i$.
\end{proof}

The next result shows that if a complex is `infinitely' homologically negligible, then it is contractible.

\begin{lem}\label{contractiblelem}
If $\abs{A^{\cdot}}_h = \infty$, then $A^{\cdot}$ is contractible.
\end{lem}
The following proof is taken from \cite{Rozansky}.
\begin{proof}
Since $\abs{A^{\cdot}}_h = \infty$ we have that $A^{\cdot} \simeq A_i^{\cdot}$ for complexes $A_i^{\cdot}$ which are
$O^h(m_i)$ with $\lim_{i\to \infty} m_i = \infty$. We have the following diagram in which each map is a 
homotopy equivalence
\[\xymatrix{A^{\cdot}\ar@/^/[r]^{f_0} & A_1^{\cdot} \ar@/^/[r] \ar@/^/[l]^{g_0} 
& \cdots \ar@/^/[r] \ar@/^/[l] &A_i^{\cdot} \ar@/^/[r]^{f_i} \ar@/^/[l]
& A_{i+1}^{\cdot} \ar@/^/[l]^{g_i} \ar@/^/[r] &\cdots \ar@/^/[l]} \ \ ;\]
in particular $Id_{A_i} -g_i\circ f_i = d_{A_i}\circ H_i + H_i \circ d_{A_i}$ where the $H_i$ are 
chain homotopies (and $A_0^{\cdot} = A^{\cdot}$).

Now, define maps $\tilde{f}_i = f_i \circ \cdots \circ f_0$, $\tilde{g}_i = g_0 \circ \cdots \circ g_i$, and 
$\tilde{H}_i = H_0 + \tilde{g}_0\circ H_1\circ \tilde{f}_0 + \cdots + \tilde{g}_{i-1}\circ H_i\circ \tilde{f}_{i-1}$.
which are related by
\begin{equation}\label{nullhtpy}
Id_A - \tilde{g}_i\circ \tilde{f}_i = d_A\circ \tilde{H}_i + \tilde{H}_i\circ d_A.
\end{equation}
The equality $\lim_{i\to \infty} m_i = \infty$ implies that the maps $\tilde{H}_i$ stabilize in each homological degree
as $i\to \infty$, so we can define their stable limit $\tilde{H}$. Equation \eqref{nullhtpy} stabilizes as well (for the
same reason) to give
\[ Id_A = d_A\circ \tilde{H} + \tilde{H}\circ d_A\]
which shows that $A^{\cdot}$ is contractible.
\end{proof}

We can now prove a uniqueness result concerning limits in $\K(\mathcal{A})$.

\begin{prop}\label{Klimunique}
The limit of a Cauchy sequence $\mathbf{A}$ is unique up to homotopy equivalence. 
\end{prop}

\begin{proof}
Since $\mathbf{A}$ is Cauchy we can construct the limit $B^{\cdot}$ as in the proof of Theorem \ref{klimitiffcauchy}.
If $A^{\cdot}$ is another $\K$-limit, then Lemma \ref{technicallem} gives a map 
$A^{\cdot} \overset{h}{\longrightarrow} B^{\cdot}$ so that the triangles
\[\xymatrix{& A^{\cdot} \ar[d]^{h} \ar[dl]_{a_i} \\
A_i^{\cdot}  & B^{\cdot} \ar[l]^{b_i} }  \]
commute (up to homotopy) for all $i$.

Lemma \ref{trianglelem} gives that
\[\abs{\cone(h)}_h \geq \min \{\abs{\cone(b_i)}_h+1,\abs{\cone(a_i)}_h\}\]
for all $i$. Taking the limit of the right hand side as $i\to\infty$ and using the fact that both $A^{\cdot}$ and
$B^{\cdot}$ are $\K$-limits, we find $\abs{\cone(h)}_h = \infty$. Lemma \ref{contractiblelem} then shows that
$\cone(h)$ is contractible, which is equivalent to $h$ being a homotopy equivalence.
\end{proof}

We conclude this section with two easy results concerning $\K$-limits.

\begin{prop}\label{easy1}
If $\mathbf{A} = \left( \xymatrix{A_0^{\cdot} & A_1^{\cdot} \ar[l]_-{f_0} & \cdots \ar[l]_-{f_1}}\right)$ is a 
Cauchy system and $\lim_{l\to\infty}\abs{A_l^{\cdot}}_h = \infty$, then $\lim_{\K}\mathbf{A} \simeq 0$.
\end{prop}
\begin{proof}
The maps $\tilde{f}_l = \left(\xymatrix{0 \ar[r] & A_l^{\cdot}}\right)$ satisfy the condition required in 
Definition \ref{defKlim}. The result then follows from Proposition \ref{Klimunique}.
\end{proof}

\begin{prop}\label{easy2}
If $\mathbf{A} = \left( \xymatrix{A_0^{\cdot} & A_1^{\cdot} \ar[l]_-{f_0} & \cdots \ar[l]_-{f_1}}\right)$ 
is a Cauchy system in 
$\K(\mathcal{A})$ and $\mathcal{A}$ is a tensor category then 
\[
B^{\cdot} \otimes \mathbf{A} = \left(\xymatrix{B^{\cdot} \otimes A_0^{\cdot} 
& B^{\cdot} \otimes A_1^{\cdot} \ar[l]_-{id\otimes f_0} & \cdots \ar[l]_-{id \otimes f_1}}\right)
\]
is a Cauchy system and 
$\lim_{\K}\left( B^{\cdot} \otimes \mathbf{A}\right) \simeq B^{\cdot} \otimes \lim_{\K} \mathbf{A}$. 
\end{prop}
\begin{proof}
Consider the maps $\xymatrix{\lim_{\K} \mathbf{A} \ar[r]^-{\tilde{f}_l} & A_l^{\cdot}}$ which satisfy 
\[\lim_{l\to\infty} \abs{\cone(\tilde{f}_l)}_h = \infty.\] These give maps 
$\xymatrix{B^{\cdot} \otimes \lim_{\K} \mathbf{A} \ar[r]^-{id\otimes \tilde{f}_l} & B^{\cdot} \otimes A_l^{\cdot}}$
which also satisfy \[\lim_{l\to\infty} \abs{\cone(id\otimes\tilde{f}_l)}_h = \infty\] since 
$\cone(id\otimes\tilde{f}_l) = B^{\cdot} \otimes \cone(\tilde{f}_l)$. The result then follows from Proposition \ref{Klimunique}.
\end{proof}

\section{Categorified $\mathfrak{sl}_3$ projectors}\label{sectioncatproj}

In this section we construct the categorified projectors and prove Theorems \ref{1} and \ref{2}. 
Subsection \ref{scp1} contains the construction of $\tilde{P}_w$ for $w=(+ \cdots +)$; in Subsection \ref{scp2} we show that 
in this case $\chi(\tilde{P}_w) = P_w$. The case $w=(+\cdots+-\cdots-)$ is treated in Subsection \ref{scp3} and 
the results for general $w$ are given in Subsection \ref{scp4}.

\subsection{$\tilde{P}_w$ for $w=(+ \cdots +)$}\label{scp1}

We begin by constructing the categorified projectors $\tilde{P}_w$ and giving a proof of Theorem \ref{1} 
when $w=(+ \cdots +)$. The general case differs from this one only in the technical details.

We will refer to the process of applying equations \eqref{catspidereqn1}, \eqref{catspidereqn2}, and \eqref{catspidereqn3} 
to express a web in terms of the direct sum of webs with fewer digon, square, and circular faces as reduction. Recall 
that a web which has no digon, square, or circular faces is called non-elliptic and that any 
web can be reduced to a direct sum of non-elliptic webs. When we write the complex 
$\left \llbracket D \right \rrbracket$ for a tangle diagram $D$ we will assume that we have reduced all webs appearing to 
direct sums of non-elliptic webs. If we would like to consider the complex with terms unreduced 
we will denote it by 
$\left \llbracket D \right \rrbracket^{un}$.

A shifted version of $\mathfrak{sl}_3$ knot homology will be useful for our considerations. 
Given a tangle diagram $D$, define the shifted complex by 
\[
\left\llbracket D \right\rrbracket_{s} = \left\llbracket D \right\rrbracket [c_{-}] \{3c_{-} - 2c_{+} \}
\]
where $c_{\pm}$ is the number of $\pm$ crossings in $D$; the complex $\left\llbracket D \right\rrbracket_{s}^{un}$ 
is defined similarly. 
This complex is not an invariant of the tangle corresponding to $D$ as it acquires shifts 
in both homological and quantum degree under $R1$ and $R2$ Reidemeister moves (but is invariant up to homotopy 
under $R3$). 
Nevertheless, this shifting convention will prove 
useful. In particular, for any diagram $D$ the shifted complex $\left\llbracket D \right\rrbracket_{s}$ is supported in 
non-negative homological degree.

We begin with a basic result describing the complex assigned to a $Y$-web attached to a 
positive crossing.
\begin{lem}\label{untwist} 
There are homotopy equivalences
\[
\left\llbracket \xy
(-3,0)*{}="a";
(.25,0)*{}="b";
(3,2.5)*{}="c";
(3,-2.5)*{}="d";
{\ar "b"; "a" };
{\ar "b"; "c" };
{\ar "b"; "d" };
\endxy \bullet
\smalltwohalftwist \ \right\rrbracket_s \simeq 
\left\llbracket \xy
(-3,0)*{}="a";
(0.25,0)*{}="b";
(3,2.5)*{}="c";
(3,-2.5)*{}="d";
{\ar "b"; "a" };
{\ar "b"; "c" };
{\ar "b"; "d" };
\endxy \ \right\rrbracket_s [1] \{2\} 
\]
and
\[
\left\llbracket
\smalltwohalftwist \bullet
\smallYli
\right\rrbracket_s \simeq
\left\llbracket
\smallYli
\right\rrbracket_s [1]\{2\} .
\]
\end{lem}
\begin{proof}
We have
\[
\left\llbracket \smallYro \bullet \smalltwohalftwist \ \right\rrbracket_s = 
\xymatrix{ \smallYro
\ar[r]^-{\left(\begin{smallmatrix} id \\ \ast \end{smallmatrix}\right)}
& \left( \smallYro \oplus q^2 \smallYro \right)
}
\]
and
\[
\left\llbracket
\smalltwohalftwist \bullet
\smallYli
\right\rrbracket_s = 
\xymatrix{ \smallYli
\ar[r]^-{\left(\begin{smallmatrix} id \\ \ast \end{smallmatrix}\right)}
& \left( \smallYli \oplus q^2 \smallYli \right)} .
\]
The result then follows from Proposition \ref{GE}.
\end{proof}

Now consider the complex $\left\llbracket \smalltwistm{m}{} \right\rrbracket_s^{un}$ assigned to the 
diagram of a full twist on $m$ strands and note that every web $W$ appearing in the 
complex except the lone identity web in homological degree zero takes the form
\begin{equation}\label{goodwebform}
W = 
\xy
(0,0)*{\smallYli};
{\ar (-3,4);(3,4)};
{\ar (-3,-4);(3,-4)};
\endxy \bullet
W' \bullet 
\xy
(0,0)*{\smallYro};
{\ar (-3,4);(3,4)};
{\ar (-3,-4);(3,-4)};
\endxy
\end{equation}
where we have omitted multiplicities from the identity strands (we will often do this to simplify notation). 
Since reduction cannot affect such a decomposition, we find that 
\begin{align*}
\xymatrix{\left\llbracket \smalltwistm{m}{} \right\rrbracket_s} &=
\xymatrix{\left(\id{m}\right) \ar[r]^-{z} &C^1 \ar[r] &C^2 \ar[r] &\cdots} \\
&= \cone\left(\xymatrix{\left\llbracket \id{m} \right\rrbracket_s \ar[r]^-{z} & C[-1]^{\cdot}}\right)[1]
\end{align*}
where every web appearing in each $C^i$ is of the form \eqref{goodwebform} and non-elliptic 
(by definition we take $C^h=0$ for $h\leq0$). We thus have 
the distinguished triangle 
\[
\xymatrix{\left\llbracket \id{m} \right\rrbracket_s \ar[r]^-{z} & C[-1]^{\cdot} \ar[r] 
&\left\llbracket \smalltwistm{m}{} \right\rrbracket_s[-1] \ar[r]^-{g[-1]} 
& \left\llbracket \id{m} \right\rrbracket_s[-1]}
\]
where the map $g$ is the identity in homological degree zero and zero in all other homological degrees.
This implies that there is a homotopy equivalence
\[
C[-1]^{\cdot} \simeq \cone\left(
\xymatrix{\left\llbracket \smalltwistm{m}{} \right\rrbracket_s \ar[r]^-{g} 
& \left\llbracket \id{m} \right\rrbracket_s} \right).
\]

We now consider the inverse system 
\begin{equation}
\mathbf{T}_w = \xymatrix{\left\llbracket \id{m} \right\rrbracket_s 
& \left\llbracket \smalltwistm{m}{} \right\rrbracket_s \ar[l]_-{g_0}
& \left\llbracket \smalltwistm{m}{2} \right\rrbracket_s \ar[l]_-{g_1}
& \cdots \ar[l]
}
\end{equation}
for $w=\underbrace{(+ \cdots +)}_{m}$ where $g_k$ is defined as
\[
\xymatrix{\left\llbracket \smalltwistm{m}{} \right\rrbracket_s 
\bullet
\left\llbracket 
\smalltwistm{m}{k} \right\rrbracket_s
\ar[rr]^-{g\bullet id} 
& & \left\llbracket \id{m} \right\rrbracket_s
\bullet
\left\llbracket 
\smalltwistm{m}{k} \right\rrbracket_s
}.
\]
Here $\smalltwistm{m}{k}$ denotes $k$ full twists on $m$ strands.

\begin{prop}\label{plusCauchy}
The inverse system $\mathbf{T}_w$ is Cauchy.
\end{prop}
\begin{proof}

We inductively construct complexes $C_k^{\cdot}$ satisfying the following conditions.
\begin{enumerate}
\item $C_k[-1]^{\cdot} \simeq \cone(g_k)$.
\item $C_k^{\cdot}$ is $\Oh^h(2k+1)$.
\item Every web appearing in $C_k^{\cdot}$ takes the form \eqref{goodwebform}.
\end{enumerate}
Let $C_0^{\cdot} = C^{\cdot}$ and suppose we have constructed $C_0^{\cdot},\ldots,C_{k-1}^{\cdot}$ as above. 
Proposition \ref{tensorcone} gives that
\begin{align*}
\cone(g_k) &= \cone(g_{k-1}) \bullet \xymatrix{\left\llbracket \smalltwistm{m}{} \right\rrbracket_s} \\
&\simeq C_{k-1}[-1]^{\cdot} \bullet \xymatrix{\left\llbracket \smalltwistm{m}{} \right\rrbracket_s}
\end{align*}
so we must show that $C_{k-1}^{\cdot} \bullet \xymatrix{\left\llbracket \smalltwistm{m}{} \right\rrbracket_s}$
is homotopy equivalent to a complex satisfying the second and third of the above conditions.

To this end, consider the tangled web
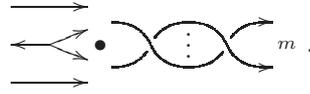
\[
\xy
{\ar (-5,5)*{}; (5,5)*{}};
{\ar (-5,-5)*{}; (5,-5)*{}};
(-5,0)*{}="a";
(0,0)*{}="b";
(5,2)*{}="c";
(5,-2)*{}="d";
{\ar "b"; "a" };
{\ar "b"; "c" };
{\ar "b"; "d" };
\endxy \bullet
\twistm{m}{} \ .
\]
Using Reidemeister $3$ moves, we can pull the crossings on the two strands aligning with the $Y$-web 
through so that they take place before any other crossings, giving the tangled web
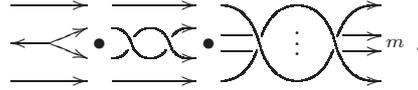
\[
\xy
{\ar (-5,5)*{}; (5,5)*{}};
{\ar (-5,-5)*{}; (5,-5)*{}};
(-5,0)*{}="a";
(0,0)*{}="b";
(5,2)*{}="c";
(5,-2)*{}="d";
{\ar "b"; "a" };
{\ar "b"; "c" };
{\ar "b"; "d" };
\endxy \bullet
\twisttwo{} \bullet
\twistds{m}{} \ .
\]
The rightmost tangle above is the result after pulling the crossings on the two strands through and out of the twist; 
we shall denote such a tangle in this way for the duration. Lemma \ref{untwist} gives the homotopy 
equivalence
\[
\left\llbracket 
\xy
{\ar (-5,5)*{}; (5,5)*{}};
{\ar (-5,-5)*{}; (5,-5)*{}};
(-5,0)*{}="a";
(0,0)*{}="b";
(5,2)*{}="c";
(5,-2)*{}="d";
{\ar "b"; "a" };
{\ar "b"; "c" };
{\ar "b"; "d" };
\endxy \bullet
\twisttwo{} \bullet
\twistds{m}{} \right\rrbracket_s \simeq
\left\llbracket 
\xy
{\ar (-5,5)*{}; (5,5)*{}};
{\ar (-5,-5)*{}; (5,-5)*{}};
(-5,0)*{}="a";
(0,0)*{}="b";
(5,2)*{}="c";
(5,-2)*{}="d";
{\ar "b"; "a" };
{\ar "b"; "c" };
{\ar "b"; "d" };
\endxy \bullet
\twistds{m}{} \right\rrbracket_s [2]\{4\}
\]
so we have homotopy equivalences
\begin{equation}\label{pt}
\left\llbracket W \bullet \smalltwistm{m}{} \right\rrbracket_s \simeq
\left\llbracket W \bullet \twistds{m}{} \right\rrbracket_s [2]\{4\}
\end{equation}
for each web $W$ appearing in $C_{k-1}^{\cdot}$.
Proposition \ref{replacement} now gives that the complex
\[
C_{k-1}^{\cdot} \bullet \left\llbracket \smalltwistm{m}{} \right\rrbracket_s
\]
is homotopy equivalent to a complex which is $\Oh^h(2k+1)$ and 
whose terms come from complexes taking the same form as the right side of equation \eqref{pt};
define $C_k^{\cdot}$ to be this complex. 
Since all such webs take the form \eqref{goodwebform}, the result follows.
\end{proof}

Since $\mathbf{T}_w$ is Cauchy, Theorem \ref{klimitiffcauchy} implies that $\lim_{\K}\mathbf{T}_w$ exists.
Let $\tilde{P}_w$ denote the limiting complex for $\mathbf{T}_w$ explicitly constructed using the proof of Theorem 
\ref{klimitiffcauchy}. The next result follows from analysis of the details of that proof.

\begin{prop}\label{plus1}
Let $w=(+ \cdots +)$. The web $\Id_w = \left( \id{m} \right)$ appears only once in $\tilde{P}_w$ and does so in quantum and homological 
degree zero. All other webs appearing in $\tilde{P}_w$ take the form 
$
\xy
(0,0)*{\smallYli};
{\ar (-3,4);(3,4)};
{\ar (-3,-4);(3,-4)};
\endxy \bullet
W \bullet 
\xy
(0,0)*{\smallYro};
{\ar (-3,4);(3,4)};
{\ar (-3,-4);(3,-4)};
\endxy
$.
\end{prop}
\begin{proof}
The limit of the Cauchy sequence is given as the $\Ch$-limit of the stabilizing system $\mathbf{B}$ constructed
in the proof of Theorem \ref{klimitiffcauchy}. 
In the current case, we see that
\begin{align*}
B_0^{\cdot} &= \llbracket \id{m} \rrbracket_s    \\
B_1^{\cdot} &= \cone\left( \llbracket \id{m} \rrbracket_s [1] \to C_{0}^{\cdot} \right) \\
& \ \ \vdots \\
B_k^{\cdot} &= \cone\left( B_{k-1}^{\cdot} [1] \to C_{k-1}^{\cdot} \right).
\end{align*}
The result now follows from our description of the complexes $C_k^{\cdot}$ above.
\end{proof}

\begin{prop}\label{plus2}
Let $w=(+ \cdots +)$. If $\wt(v) < \wt(w)$ then $\tilde{P}_w\bullet W_1 \simeq 0$ for any 
$W_1 \in \Hom_{\bullet}(w,v)$ and $W_2\bullet \tilde{P}_w \simeq 0$ for any 
$W_2 \in \Hom_{\bullet}(v,w)$.
\end{prop}
\begin{proof}
Let $W_1 \in \Hom_{\bullet}(w,v)$ and note that it suffices to consider the case when $W_1$ is non-elliptic. Proposition 
\ref{Kuprop} then implies that
\[
W_1 = 
\xy
(0,0)*{\smallYli};
{\ar (-3,4);(3,4)};
{\ar (-3,-4);(3,-4)};
\endxy \bullet
W_1'
\]
for some $W_1'$.
Lemma \ref{untwist} gives the homotopy equivalence
\[
\left\llbracket \smalltwistm{m}{k}\bullet
\xy
(0,0)*{\smallYli};
{\ar (-3,4);(3,4)};
{\ar (-3,-4);(3,-4)};
\endxy \ \right\rrbracket_s \simeq
\left\llbracket \twistds{m}{k} \bullet
\xy
(0,0)*{\smallYli};
{\ar (-3,4);(3,4)};
{\ar (-3,-4);(3,-4)};
\endxy \
\right\rrbracket_s [2k]\{4k\}
\]
so we have
\[
\abs{\left\llbracket 
\smalltwistm{m}{k}\bullet
W_1
\right\rrbracket_s
}_h \geq 2k .
\]
Propositions \ref{easy1} and \ref{easy2} now give
\begin{align*}
\tilde{P}_w \bullet W_1 &=\left(\lim_{\K}\mathbf{T}_w\right)\bullet W_1 \\ 
&\simeq \lim _{\K}\left(\mathbf{T}_w\bullet W_1 \right) \\
&\simeq 0 .
\end{align*}
The proof concerning $W_2$ is completely analogous.
\end{proof}

\begin{prop}\label{plus3}
Let $w=(+ \cdots +)$, then 
\[
\tilde{P}_w \bullet \tilde{P}_w \simeq \tilde{P}_w .
\]
\end{prop}
\begin{proof}
Proposition \ref{plus1} gives that 
\[
\tilde{P}_w = \left( \xymatrix{\left(\id{m}\right) \ar[r]^-{z} &D^1 \ar[r] &D^2 \ar[r] &\cdots} \right)
\]
where each web appearing in $D^i$ takes the form \eqref{goodwebform}. Setting $D^i=0$ for $i\leq 0$ we have
\[
\tilde{P}_w = \cone \left( \llbracket \id{m} \rrbracket_s \to D[-1]^{\cdot} \right)[1]
\]
which gives the distinguished triangle
\[
\xymatrix{\tilde{P}_w \ar[r] & \llbracket \id{m} \rrbracket_s \ar[r] & D[-1]^{\cdot}}.
\]
This in turn gives the distinguished triangle
\[
\xymatrix{\tilde{P}_w \bullet \tilde{P}_w \ar[r] & \tilde{P}_w \ar[r] & D[-1]^{\cdot}\bullet \tilde{P}_w}
\]
so by Proposition \ref{conenull} it suffices to show that $D^{\cdot} \bullet \tilde{P}_w \simeq 0$.

We can write 
\[
D^{\cdot} = \cone \left( D^{1}[2] \to t_{\geq 2}D^{\cdot} \right)
\]
where $D^1$ stands for the complex with all terms zero except $D^1$ sitting in homological degree zero and 
$t_{\geq k}D^{\cdot}$ denotes the truncation of $D^{\cdot}$ from below. As above, this gives the distinguished 
triangle
\[
\xymatrix{t_{\geq 2}D^{\cdot} \bullet \tilde{P}_w \ar[r] & D^{\cdot} \bullet \tilde{P}_w \ar[r] 
& D^1[1] \bullet \tilde{P}_w}.
\]
Proposition \ref{plus2} implies that $D^1[1] \bullet \tilde{P}_w \simeq 0$ so 
$t_{\geq 2}D^{\cdot} \bullet \tilde{P}_w \simeq D^{\cdot} \bullet \tilde{P}_w$. Repeating this procedure gives
\[
t_{\geq k}D^{\cdot} \bullet \tilde{P}_w \simeq D^{\cdot} \bullet \tilde{P}_w
\]
for all $k>0$. Lemma \ref{contractiblelem} then implies that $D^{\cdot} \bullet \tilde{P}_w \simeq 0$.
\end{proof}

Propositions \ref{plus1}, \ref{plus2}, and \ref{plus3} give the proof of Theorem \ref{1} when $w=(+\cdots +)$. 
Moreover, uniqueness of $\tilde{P}_w$ follows from the argument used in the proof of Proposition \ref{plus3}. 
Indeed, if $\tilde{P}'_w$ is another complex supported in non-negative homological degree satisfying Propositions 
\ref{plus1} and \ref{plus2} then similar reasoning shows that the complex $\tilde{P}_w \bullet \tilde{P}'_w$ 
is homotopy equivalent to both $\tilde{P}_w$ and $\tilde{P}'_w$.

\subsection{Decategorification for $w=(+ \cdots +)$}\label{scp2}

We now aim to show that $\chi(\tilde{P}_w) = P_w$ when $w=(+\cdots +)$. Observe that doing so requires two 
steps. First, we must show that it is possible to define and compute $\chi(\tilde{P}_w)$ since 
$\chi$ is generally not well-defined for complexes in $\K^+(\mathcal{F})$. Second, we must show the desired equality.

To resolve the first issue, we consider a full subcategory of $\K^+(\mathcal{F})$ where we restrict the
support of complexes $A^{\cdot}$. By definition, $\supp(A^{\cdot})$ is the 
set of pairs $(h,l)\in \Z^2$ for which $A^h$ has a non-zero summand in quantum degree $q^l$. We shall identify 
$\supp(A^{\cdot})$ 
with the corresponding discrete subset in $\R^2$ and abuse notation slightly by calling this the $(h,q)$-plane. 

If we wish to translate a subset of $\R^2$ we will use the same notation which we use to shift complexes, 
viewing homological degree as the horizontal direction and quantum degree as the vertical direction. For 
instance,
\[
\{(h,q)| h\geq 2 \text{ and } q \geq 1\}[2]\{1\} = \{(h,q)| h\geq 4 \text{ and } q \geq 2\}.
\]
Consider now the subset of $\R^2$ given by
\[
R_t = \{(h,q)\in \R^2 | h\geq0 \text{ and } q\geq t\cdot h\}
\]
and let $\hat{\mathcal{S}'}$ denote the $\mathfrak{sl}_3$ spider considered over the ring 
$\Z[q^{-1},q]] := \Z[[q]][\frac1q]$.
The following conditions are sufficient to guarantee that the Euler characteristic $\chi(A^{\cdot})$ of a
complex $A^{\cdot}$ in $\K^+(\mathcal{F})$ is a well defined element in $\hat{\mathcal{S}'}$: 
\begin{enumerate}
\item $\supp(A^{\cdot})\subset R_t[a]\{b\}$ for some $t > 0$ and $a,b \in \Z$. 
\item All webs appearing in $A^{\cdot}$ are non-elliptic. 
\item Only finitely many distinct webs appear in $A^{\cdot}$.
\end{enumerate}
Denote by $\K^{\angle}(\mathcal{F})$ the full subcategory of $\K^+(\mathcal{F})$ whose objects satisfy the above 
conditions. Note that this subcategory is closed under taking direct sums, cones, and tensor product. The horizontal 
composition of two complexes in $\K^{\angle}(\mathcal{F})$ is isomorphic (in $\K^{+}(\mathcal{F})$) to a complex 
in $\K^{\angle}(\mathcal{F})$ via reduction. In this sense, we can view $\K^{\angle}(\mathcal{F})$ as closed under 
horizontal composition.

We now aim to show that $\tilde{P}_w$ is an object in $\K^{\angle}(\mathcal{F})$. 
Before doing so, we need some preparatory lemmata.
Our first result enables us to bound the quantum degree of the webs appearing when we express a web as a direct sum of 
non-elliptic webs.
\begin{lem}\label{lem1}
Let $W$ be a web with no closed components. When reducing $W$ to a direct sum of non-elliptic webs
we can assume that no closed component forms. 
If $W$ has $r$ faces, $W \cong \oplus_{i=1}^{s} q^{k_{i}} \cdot W_i$ 
is the direct sum decomposition into non-elliptic webs resulting 
from reduction, and $r_i$ is the number of faces in $W_i$, then $k_i \geq r_i-r$.
\end{lem}
\begin{proof}
Suppose that a reduction does produce a closed component. We can assume that no reductions are possible which do not 
split off a closed component; otherwise, perform these reductions. Since a closed component can only form upon application of 
equation \eqref{catspidereqn2}, we have that
\[
W = \left( U \bullet 
\xy
(-1,1)*{}="a";
(1,1)*{}="b";
(-1,-1)*{}="c";
(1,-1)*{}="d";
(-2.5,2.5)*{}="e";
(2.5,2.5)*{}="f";
(-2.5,-2.5)*{}="g";
(2.5,-2.5)*{}="h";
"a";"b" **\dir{-};
"d";"b" **\dir{-};
"a";"c" **\dir{-};
"d";"c" **\dir{-};
"a";"e" **\dir{-};
"f";"b" **\dir{-};
"g";"c" **\dir{-};
"d";"h" **\dir{-};
\endxy \bullet
V \right)
\]
and hence the isomorphism
\[
W \cong
\Big(U \bullet
\xy
(-2,2)*{}="a";
(2,2)*{}="b";
(-2,-2)*{}="c";
(2,-2)*{}="d";
"a";"b" **\crv{-};
"c";"d" **\crv{-};
\endxy \bullet
V \Big) \oplus 
\Big(U \bullet
\xy
(-2,2)*{}="a";
(2,2)*{}="b";
(-2,-2)*{}="c";
(2,-2)*{}="d";
"a";"c" **\crv{-};
"b";"d" **\crv{-};
\endxy \bullet
V \Big) .
\]
We can assume that
$
\xy
(2,2)*{}="a";
(2,-2)*{}="b";
"a";"b" **\crv{-};
\endxy \bullet V
$
has no internal digon faces. Indeed, if a digon face is formed then we must have
$
V = 
\xy
(-2,2)*{}="a";
(0,1)*{}="b";
(2,2)*{}="c";
(0,-1)*{}="d";
(-2,-2)*{}="e";
(2,-2)*{}="f";
"a";"b" **\dir{-};
"b";"c" **\dir{-};
"b";"d" **\dir{-};
"e";"d" **\dir{-};
"d";"f" **\dir{-};
\endxy \bullet
V'
$
so we can instead consider the reduction corresponding to 
$W =
U' \bullet 
\xy
(-1,1)*{}="a";
(1,1)*{}="b";
(-1,-1)*{}="c";
(1,-1)*{}="d";
(-2.5,2.5)*{}="e";
(2.5,2.5)*{}="f";
(-2.5,-2.5)*{}="g";
(2.5,-2.5)*{}="h";
"a";"b" **\dir{-};
"d";"b" **\dir{-};
"a";"c" **\dir{-};
"d";"c" **\dir{-};
"a";"e" **\dir{-};
"f";"b" **\dir{-};
"g";"c" **\dir{-};
"d";"h" **\dir{-};
\endxy \bullet
V'
$ where
$U' = U \bullet
\xy
(-2,2)*{}="a";
(0,1)*{}="b";
(2,2)*{}="c";
(0,-1)*{}="d";
(-2,-2)*{}="e";
(2,-2)*{}="f";
"a";"b" **\dir{-};
"b";"c" **\dir{-};
"b";"d" **\dir{-};
"e";"d" **\dir{-};
"d";"f" **\dir{-};
\endxy$. Similarly, we can assume that 
$
\xy
(2,2)*{}="a";
(2,-2)*{}="b";
"a";"b" **\crv{-};
\endxy \bullet V
$ has at most one internal square face since otherwise we have
$V = V'' \bullet 
\xy
(-1,1)*{}="a";
(1,1)*{}="b";
(-1,-1)*{}="c";
(1,-1)*{}="d";
(-2.5,2.5)*{}="e";
(2.5,2.5)*{}="f";
(-2.5,-2.5)*{}="g";
(2.5,-2.5)*{}="h";
"a";"b" **\dir{-};
"d";"b" **\dir{-};
"a";"c" **\dir{-};
"d";"c" **\dir{-};
"a";"e" **\dir{-};
"f";"b" **\dir{-};
"g";"c" **\dir{-};
"d";"h" **\dir{-};
\endxy \bullet
V'
$
and again we can consider the reduction corresponding to
$W =
U'' \bullet 
\xy
(-1,1)*{}="a";
(1,1)*{}="b";
(-1,-1)*{}="c";
(1,-1)*{}="d";
(-2.5,2.5)*{}="e";
(2.5,2.5)*{}="f";
(-2.5,-2.5)*{}="g";
(2.5,-2.5)*{}="h";
"a";"b" **\dir{-};
"d";"b" **\dir{-};
"a";"c" **\dir{-};
"d";"c" **\dir{-};
"a";"e" **\dir{-};
"f";"b" **\dir{-};
"g";"c" **\dir{-};
"d";"h" **\dir{-};
\endxy \bullet
V'
$ where
$
U'' = U \bullet 
\xy
(-1,1)*{}="a";
(1,1)*{}="b";
(-1,-1)*{}="c";
(1,-1)*{}="d";
(-2.5,2.5)*{}="e";
(2.5,2.5)*{}="f";
(-2.5,-2.5)*{}="g";
(2.5,-2.5)*{}="h";
"a";"b" **\dir{-};
"d";"b" **\dir{-};
"a";"c" **\dir{-};
"d";"c" **\dir{-};
"a";"e" **\dir{-};
"f";"b" **\dir{-};
"g";"c" **\dir{-};
"d";"h" **\dir{-};
\endxy \bullet
V''$.

We analyze the closed web
\[ C =
\xy
(2,2)*{}="a";
(2,-2)*{}="b";
"a";"b" **\crv{-};
\endxy \bullet V
\]
which has at most one internal square face and no internal digon faces. 
The orientations of edges around vertices shows that all faces must have an even 
number of edges.
Considering the web on the surface of the $2$-sphere creates an external face which may have any (even) number of edges. 
Since every edge borders two regions we compute
\[
e_C \geq \frac{1}{2}(6(f_C-2)+4+2) = 3(f_C-1)
\]
where $e_C$ is the number of edges in $C$ and $f_C$ is the number of faces bounded on the $2$-sphere by $C$.
Since $C$ is trivalent we have $v_C = \frac{2}{3} e_C$ where $v_C$ is the number of vertices in C. We thus find
\begin{align*}
2 &= f_C - e_C + v_C \\
&= f_C - \frac{1}{3} e_C \\
&\leq f_C + (1-f_C) = 1,
\end{align*}
a contradiction.

The second statement follows from the first by noticing that each reduction lowers the number of internal faces and that 
\eqref{catspidereqn3} need never be used.
\end{proof}

Given a diagram $D$, define the $0$-resolution as the unique web appearing in 
$\left\llbracket D \right\rrbracket_{s}^{un}$ 
in homological degree zero. Concretely, this is the web obtained by taking the smooth resolution
$\xy
(-2,2)*{}="a";
(2,2)*{}="b";
(-2,-2)*{}="c";
(2,-2)*{}="d";
"a";"c" **\crv{-};
"b";"d" **\crv{-};
\endxy$
of each positive crossing and the singular resolution
$\xy
(-2,2)*{}="a";
(0,1)*{}="b";
(2,2)*{}="c";
(0,-1)*{}="d";
(-2,-2)*{}="e";
(2,-2)*{}="f";
"a";"b" **\dir{-};
"b";"c" **\dir{-};
"b";"d" **\dir{-};
"e";"d" **\dir{-};
"d";"f" **\dir{-};
\endxy$
of each negative crossing. Define the smooth resolution of $D$ to be the web obtained by taking the smooth resolution 
of both positive and negative crossings.

\begin{lem}\label{lem2}
Let $D$ be a tangle diagram and let the smooth resolution of $D$ have no closed components. Let $r$ be the number of 
internal faces in the $0$-resolution of $D$ and $c_{+}$ be the number of positive crossings in $D$, then the complex 
$\llbracket D \rrbracket_{s}$ 
satisfies
\[
\supp\left(\llbracket D \rrbracket_{s}\right) \subset R_{1/c_{+}}\{-r-1\}.
\]
If a web $W$ appearing in 
$\llbracket D \rrbracket_{s}$ has $r_W$ internal faces then that term is supported in $R_{1/c_{+}}\{r_W-r-1\}$. 
Moreover, if $D$ has no negative crossings then the $-1$ shifts can be omitted from both statements.
\end{lem}
\begin{proof}
Since any web in $\llbracket D \rrbracket_s^{un}$ is obtained from the smooth resolution by switching smoothings of 
crossings to singular resolutions, no web appearing in this complex has closed components. Also, note that changing 
the resolution of a crossing from the smooth resolution to the singular resolution produces at most one new internal face 
while changing from the singular resolution to the smooth resolution cannot produce new internal faces.

It follows that a web $V$ appearing in $\llbracket D \rrbracket_s^{un}$ in homological degree $h$ has at most 
$r+\min(h,c_{+})$ internal faces. Next, note that 
the complex $\llbracket D \rrbracket_s^{un}$ is supported along the line $h=q$ in the $(h,q)$-plane. 
Hence by Lemma 
\ref{lem1}, if $q^l W$ appears in $\llbracket D \rrbracket_s$ in homological degree $h$ then 
\begin{align*}
l &\geq h+r_W-r-\min(h,c_{+}) \\
&\geq r_W-r-1+\frac{h}{c_{+}} 
\end{align*}
which gives the result.

For the final statement, note that if $D$ has no negative crossing then the $0$-resolution and the smooth resolution 
agree. All webs appearing in $\llbracket D \rrbracket_s^{un}$ are thus obtained from the smooth resolution by 
changing a crossing to the 
singular resolution. Since no internal faces are formed when the first resolution is changed, we find that a web $V$ 
appearing in $\llbracket D \rrbracket_s^{un}$ in homological degree $h>0$ has at most 
$r+\min(h,c_{+})-1$ internal faces. The result then follows as above.
\end{proof}

\begin{prop}\label{pluscatprojgoodform}
Let $w=(+\cdots +)$. The categorified projector $\tilde{P}_w$ lies in $\K^{\angle}(\mathcal{F})$.
\end{prop}

\begin{proof}
Recall that there are only finitely many non-elliptic basis webs in \[\Hom_{\bullet}(w,w).\]
Since the terms of $\tilde{P}_w$ are the stable limit of the terms from the complexes $B_k$ given in the proof of 
Proposition \ref{plus1} and all webs appearing there are non-elliptic, 
it suffices to show that the complexes $C_k^{\cdot}$ from the proof of Proposition \ref{plusCauchy} can (additionally) 
be chosen to lie in $A_{1/t}$ for some fixed $t>0$. 

Lemma \ref{lem2} gives that each web $W$ appearing in
$
\llbracket \smalltwistm{m}{} \rrbracket_s
$
is supported in $A_{1/M}\{r_W\}$ where $M = m(m-1)$ and $r_W$ is the number of internal faces in $W$.
It follows that the same is true for webs appearing in $C_0^{\cdot}$. We now show, via induction, that the same result 
holds for the complexes $C_k^{\cdot}$.

Recalling how $C_k^{\cdot}$ is constructed from $C_{k-1}^{\cdot} \bullet \llbracket \smalltwistm{m}{} \rrbracket_s$, 
it suffices to show that if $W$ is a web 
appearing in $C_{k-1}^{\cdot}$ then any web $V$ appearing in 
\[
\left\llbracket W \bullet \twistds{m}{} \right\rrbracket_s [2]\{4\}
\]
is supported in $A_{1/M}\{r_V - r_W\}$.
Since the zero ($=$ smooth) resolution of 
\[
W \bullet \twistds{m}{}
\]
is simply $W$ there are 
$r_W$ internal faces and no closed components. The result then follows from Lemma \ref{lem2}.
\end{proof}

We hence can consider the $\Z[q^{-1},q]]$-linear combination of webs $\chi(\tilde{P}_w)$. 
A slight extension of the results from \cite{Rose1} shows that Euler characteristic 
is invariant under homotopy in $\K^{\angle}(\mathcal{F})$; Proposition \ref{plus3} then gives that
\[
\chi(\tilde{P}_w) \bullet \chi(\tilde{P}_w) = \chi(\tilde{P}_w)
\]
and Proposition \ref{plus2} gives that if $\wt(v) < \wt(w)$ then 
\[
\chi(\tilde{P}_w)\bullet W_1 = 0
\]
for any 
$W_1 \in \Hom_{\bullet}(w,v)$ and 
\[
W_2\bullet \chi(\tilde{P}_w) = 0
\]
for any 
$W_2 \in \Hom_{\bullet}(v,w)$. Propositions \ref{plus1} and \ref{pluscatprojgoodform} show that
\[
\chi(\tilde{P}_w) = \Id_w + \sum_{i=1}^{r} f_i(q) \cdot W_i
\]
with $f_i \in \Z[q^{-1},q]]$ and where $W_i \in \Hom_\bullet(w,w) \smallsetminus \Id_w$ are non-elliptic webs. These facts, together with Proposition 
\ref{charproj}, give a proof of Theorem \ref{2} in the case that $w=(+\cdots +)$.

\begin{ex}\label{catex1}
The computations given in \cite[Section 6.1]{MorrisonNieh} show that 
\[
\tilde{P}_{(++)} = \ \ \ \ \
\xymatrix{
\xy 
(0,0);
(0,1)*{\xy
(-3,3); (3,3) **\crv{(0,1)};
(-3,-3); (3,-3) **\crv{(0,-1)};
(.5,2.1); (.6,2.1) **\dir{-}?(1)*\dir{>};
(.5,-2); (.6,-2) **\dir{-}?(1)*\dir{>};
\endxy} ;
\endxy
\ \ \ar[r]^-{z}
& q \xy
(-4,4); (-2.5,0) **\dir{-}?(.7)*\dir{>};
(-4,-4); (-2.5,0) **\dir{-}?(.7)*\dir{>};
(-2.5,0); (2.5,0) **\dir{-}?(.4)*\dir{<};
(2.5,0); (4,4) **\dir{-}?(.7)*\dir{>};
(2.5,0); (4,-4) **\dir{-}?(.7)*\dir{>};
\endxy \ar[r]^-{\psi_{-}}
& q^3 \xy
(-4,4); (-2.5,0) **\dir{-}?(.7)*\dir{>};
(-4,-4); (-2.5,0) **\dir{-}?(.7)*\dir{>};
(-2.5,0); (2.5,0) **\dir{-}?(.4)*\dir{<};
(2.5,0); (4,4) **\dir{-}?(.7)*\dir{>};
(2.5,0); (4,-4) **\dir{-}?(.7)*\dir{>};
\endxy \ar[r]^-{\psi_{+}}
& q^5 \xy
(-4,4); (-2.5,0) **\dir{-}?(.7)*\dir{>};
(-4,-4); (-2.5,0) **\dir{-}?(.7)*\dir{>};
(-2.5,0); (2.5,0) **\dir{-}?(.4)*\dir{<};
(2.5,0); (4,4) **\dir{-}?(.7)*\dir{>};
(2.5,0); (4,-4) **\dir{-}?(.7)*\dir{>};
\endxy \ar[r]^-{\psi{-}}
& q^7 \xy
(-4,4); (-2.5,0) **\dir{-}?(.7)*\dir{>};
(-4,-4); (-2.5,0) **\dir{-}?(.7)*\dir{>};
(-2.5,0); (2.5,0) **\dir{-}?(.4)*\dir{<};
(2.5,0); (4,4) **\dir{-}?(.7)*\dir{>};
(2.5,0); (4,-4) **\dir{-}?(.7)*\dir{>};
\endxy \ar[r]^-{\psi_{+}}
&\cdots
}
\]
where $\psi_{\pm}$ are the morphisms defined in that section. Note that
\begin{align*}
\xymatrix{\chi(\tilde{P}_{(++)})} &= \ \ \ \  \xymatrix{\xy
(0,0);
(0,1)*{\xy
(-3,3); (3,3) **\crv{(0,1)};
(-3,-3); (3,-3) **\crv{(0,-1)};
(.5,2.1); (.6,2.1) **\dir{-}?(1)*\dir{>};
(.5,-2); (.6,-2) **\dir{-}?(1)*\dir{>};
\endxy} ;
\endxy \ \  - (q-q^3+q^5-q^7+\cdots) \ \xy
(-4,4); (-2.5,0) **\dir{-}?(.7)*\dir{>};
(-4,-4); (-2.5,0) **\dir{-}?(.7)*\dir{>};
(-2.5,0); (2.5,0) **\dir{-}?(.4)*\dir{<};
(2.5,0); (4,4) **\dir{-}?(.7)*\dir{>};
(2.5,0); (4,-4) **\dir{-}?(.7)*\dir{>};
\endxy} \\
&= \ \ \ \ \xymatrix{\xy
(0,0);
(0,1)*{\xy
(-3,3); (3,3) **\crv{(0,1)};
(-3,-3); (3,-3) **\crv{(0,-1)};
(.5,2.1); (.6,2.1) **\dir{-}?(1)*\dir{>};
(.5,-2); (.6,-2) **\dir{-}?(1)*\dir{>};
\endxy} ;
\endxy \ \ - \frac{1}{[2]} \ \xy
(-4,4); (-2.5,0) **\dir{-}?(.7)*\dir{>};
(-4,-4); (-2.5,0) **\dir{-}?(.7)*\dir{>};
(-2.5,0); (2.5,0) **\dir{-}?(.4)*\dir{<};
(2.5,0); (4,4) **\dir{-}?(.7)*\dir{>};
(2.5,0); (4,-4) **\dir{-}?(.7)*\dir{>};
\endxy} \\
&= P_{(++)}
\end{align*}
by Proposition \ref{projrecur}.
\end{ex}

\subsection{$\tilde{P}_w$ and decategorification for $w=(+ \cdots + - \cdots -)$}\label{scp3}

We now construct $\tilde{P}_w$ for $w=(\underbrace{+ \cdots +}_{m} \underbrace{- \cdots -}_{n})$. 
Since we have already considered the case $n=0$ (and $m=0$ by taking duals) we assume $m,n>0$.
We proceed by mimicking the proof for $w=(+\cdots +)$. Let $\smalltwistmn{m}{n}{}$ denote a full twist on 
$m+n$ strands directed as indicated. The first step is to construct a map
\[
\xymatrix{\left \llbracket \smalltwistmn{m}{n}{} \right \rrbracket_s \ar[r]^-{g} 
& \left \llbracket \idmn{m}{n} \right \rrbracket_s}
\]
and use it to build the inverse system
\[
\mathbf{T}_w = \xymatrix{\left\llbracket \idmn{m}{n} \right\rrbracket_s 
& \left\llbracket \smalltwistmn{m}{n}{} \right\rrbracket_s \ar[l]_-{g_0}
& \left\llbracket \smalltwistmn{m}{n}{2} \right\rrbracket_s \ar[l]_-{g_1}
& \cdots \ar[l]
}.
\]
This is trivial when $w=(+ \cdots +)$ since the degree zero term of the complex assigned to a single twist is the 
identity tangle. This fails for $w=(+ \cdots + - \cdots -)$, but we shall see that it holds up to homotopy, which is 
sufficient to define the map $g$.

We begin with a lemma showing this for the case $m=1=n$, which will also be of use later.

\begin{lem}\label{m=1=n}
\[
\left \llbracket \ \smalltwotwistoneone \ \right\rrbracket_s \simeq 
\left(
\xymatrix{
\xy
(0,1)*{
\xy
(-2,2)*{}="a";
(2,2)*{}="b";
(-2,-2)*{}="c";
(2,-2)*{}="d";
"a";"b" **\crv{(0,1)};
"d";"c" **\crv{(0,-1)};
(1.9,1.97);(2,2)**\dir{-}?(1)*\dir{>};
(-1.9,-1.97);(-2,-2)**\dir{-}?(1)*\dir{>};
\endxy};
\endxy \ar[r]^-{s}
& q^2 \xy
(-2,2)*{}="a";
(2,2)*{}="b";
(-2,-2)*{}="c";
(2,-2)*{}="d";
"a";"c" **\crv{-};
"d";"b" **\crv{-};
(1.9,1.85);(2,2)**\dir{-}?(1)*\dir{>};
(-1.9,-1.87);(-2,-2)**\dir{-}?(1)*\dir{>};
\endxy \ar[r]^-{ct_{-}}
& q^4 \xy
(-2,2)*{}="a";
(2,2)*{}="b";
(-2,-2)*{}="c";
(2,-2)*{}="d";
"a";"c" **\crv{-};
"d";"b" **\crv{-};
(1.9,1.85);(2,2)**\dir{-}?(1)*\dir{>};
(-1.9,-1.87);(-2,-2)**\dir{-}?(1)*\dir{>};
\endxy
}
\right)
\]
\end{lem}

Here $s$ is a saddle cobordism and $ct_{-} = ct_{L} - ct_{R}$ where $ct_{L}$ denotes the foam which is the identity foam 
on the right arc and has a `choking torus:' 
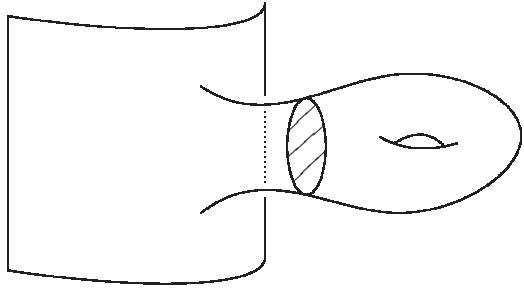
\[
\xy 0;/r.2pc/:
(-20,18); (20,20) **\crv{(0,15)&(20,15)};
(-20,-22); (-20,18) **\dir{-};
(-20,-22); (20,-20) **\crv{(0,-25)&(20,-25)};
(10,7); (10,-13) **\crv{(20,0)&(40,15)&(70,0)&(40,-20)&(20,-5)};
(20,4.2)*{\hole}="h1";
(20,-9.3)*{\hole}="h2";
(20,20); "h1" **\dir{-};
"h1"; "h2" **\dir{.};
"h2"; (20,-20) **\dir{-};
(38,-1); (50,-2) **\crv{(43,-4)}\POS?(.2)="a" \POS?(.85)="b";
"a"; "b" **\crv{(45,1)};
(26.5,-2.5)*\xycircle(3,7.6){-};
(23.7,0); (28,4) **\dir{-};
(23.6,-4.2); (29.1,.85) **\dir{-};
(24.4,-8); (29.5, -3.15) **\dir{-};
\endxy
\]
on the left arc. The foam
$ct_R$ is defined similarly.

\begin{proof}
We have
\[
\left \llbracket \ \smalltwotwistoneone \ \right\rrbracket_s =
\xymatrix{ \xy
(0,1)*{\xy
(-2,2)*{}="a";
(2,2)*{}="b";
(-2,-2)*{}="c";
(2,-2)*{}="d";
"a";"b" **\crv{(0,1)}?(1);
"d";"c" **\crv{(0,-1)}?(1);
(1.9,1.97);(2,2)**\dir{-}?(1)*\dir{>};
(-1.9,-1.97);(-2,-2)**\dir{-}?(1)*\dir{>};
\endxy};
\endxy \oplus
 \xy
(-2,2)*{}="a";
(2,2)*{}="b";
(-2,-2)*{}="c";
(2,-2)*{}="d";
"a";"c" **\crv{-};
"d";"b" **\crv{-};
(1.9,1.85);(2,2)**\dir{-}?(1)*\dir{>};
(-1.9,-1.87);(-2,-2)**\dir{-}?(1)*\dir{>};
\endxy \ar[r]^-{A}
& q^{2} \xy
(-2,2)*{}="a";
(2,2)*{}="b";
(-2,-2)*{}="c";
(2,-2)*{}="d";
"a";"c" **\crv{-};
"d";"b" **\crv{-};
(1.9,1.85);(2,2)**\dir{-}?(1)*\dir{>};
(-1.9,-1.87);(-2,-2)**\dir{-}?(1)*\dir{>};
\endxy \oplus 
 \xy
(-2,2)*{}="a";
(2,2)*{}="b";
(-2,-2)*{}="c";
(2,-2)*{}="d";
"a";"c" **\crv{-};
"d";"b" **\crv{-};
(1.9,1.85);(2,2)**\dir{-}?(1)*\dir{>};
(-1.9,-1.87);(-2,-2)**\dir{-}?(1)*\dir{>};
\endxy \oplus 
q^{2} \xy
(-2,2)*{}="a";
(2,2)*{}="b";
(-2,-2)*{}="c";
(2,-2)*{}="d";
"a";"c" **\crv{-};
"d";"b" **\crv{-};
(1.9,1.85);(2,2)**\dir{-}?(1)*\dir{>};
(-1.9,-1.87);(-2,-2)**\dir{-}?(1)*\dir{>};
\endxy \oplus 
 \xy
(-2,2)*{}="a";
(2,2)*{}="b";
(-2,-2)*{}="c";
(2,-2)*{}="d";
"a";"c" **\crv{-};
"d";"b" **\crv{-};
(1.9,1.85);(2,2)**\dir{-}?(1)*\dir{>};
(-1.9,-1.87);(-2,-2)**\dir{-}?(1)*\dir{>};
\endxy \ar[r]^-{B}
& q^{4} \xy
(-2,2)*{}="a";
(2,2)*{}="b";
(-2,-2)*{}="c";
(2,-2)*{}="d";
"a";"c" **\crv{-};
"d";"b" **\crv{-};
(1.9,1.85);(2,2)**\dir{-}?(1)*\dir{>};
(-1.9,-1.87);(-2,-2)**\dir{-}?(1)*\dir{>};
\endxy \oplus 
q^{2} \xy
(-2,2)*{}="a";
(2,2)*{}="b";
(-2,-2)*{}="c";
(2,-2)*{}="d";
"a";"c" **\crv{-};
"d";"b" **\crv{-};
(1.9,1.85);(2,2)**\dir{-}?(1)*\dir{>};
(-1.9,-1.87);(-2,-2)**\dir{-}?(1)*\dir{>};;
\endxy \oplus
 \xy
(-2,2)*{}="a";
(2,2)*{}="b";
(-2,-2)*{}="c";
(2,-2)*{}="d";
"a";"c" **\crv{-};
"d";"b" **\crv{-};
(1.9,1.85);(2,2)**\dir{-}?(1)*\dir{>};
(-1.9,-1.87);(-2,-2)**\dir{-}?(1)*\dir{>};
\endxy}
\]
where
\[
A = \begin{pmatrix}s & \ast \\ 0 & \Id \\ s & \ast \\ 0 & \Id\end{pmatrix}
\]
and
\[
B = \begin{pmatrix}\ast & \ast & \ast & \ast \\ \ast & \ast & \Id & \ast \\ 0 & \ast & 0 & -\Id\end{pmatrix}.
\]
Using Gaussian elimination, we find that
\[
\left \llbracket \ \smalltwotwistoneone \ \right\rrbracket_s \simeq 
\xymatrix{
\xy
(0,1)*{
\xy
(-2,2)*{}="a";
(2,2)*{}="b";
(-2,-2)*{}="c";
(2,-2)*{}="d";
"a";"b" **\crv{(0,1)};
"d";"c" **\crv{(0,-1)};
(1.9,1.97);(2,2)**\dir{-}?(1)*\dir{>};
(-1.9,-1.97);(-2,-2)**\dir{-}?(1)*\dir{>};
\endxy};
\endxy \ar[r]^-{s}
& q^2 \xy
(-2,2)*{}="a";
(2,2)*{}="b";
(-2,-2)*{}="c";
(2,-2)*{}="d";
"a";"c" **\crv{-};
"d";"b" **\crv{-};
(1.9,1.85);(2,2)**\dir{-}?(1)*\dir{>};
(-1.9,-1.87);(-2,-2)**\dir{-}?(1)*\dir{>};
\endxy \ar[r]
& q^4 \xy
(-2,2)*{}="a";
(2,2)*{}="b";
(-2,-2)*{}="c";
(2,-2)*{}="d";
"a";"c" **\crv{-};
"d";"b" **\crv{-};
(1.9,1.85);(2,2)**\dir{-}?(1)*\dir{>};
(-1.9,-1.87);(-2,-2)**\dir{-}?(1)*\dir{>};
\endxy
}.
\]
We deduce the second map since, up to scalar multiple, it is the only degree zero map 
which makes the diagram a complex. A computation shows that it is indeed non-zero.
\end{proof}

We now prove the general case.

\begin{prop}\label{mnstart}
\[
\left \llbracket \smalltwistmn{m}{n}{} \right \rrbracket_s \simeq 
\xymatrix{ \left(\idmn{m}{n}\right) \ar[r] & C^1 \ar[r] & C^2 \ar[r] &\cdots} 
\]
where every web appearing in $C^{\cdot}$ is non-elliptic and takes the form 
\begin{equation}\label{goodwebform2}
W_L \bullet W \bullet W_R
\end{equation}
for 
\[
W_L = 
\xy  
(-3,5); (3,5) **\dir{-}?(1)*\dir{>};
(-3,-3); (3,-3) **\dir{-}?(1)*\dir{>};
(-3,-5); (3,-5) **\dir{-}?(0)*\dir{<};
(0,1)*{\smallYli};
\endxy \text{ \ \ or \ \ }
\xy  
(-3,5); (3,5) **\dir{-}?(1)*\dir{>};
(-3,3); (3,3) **\dir{-}?(0)*\dir{<};
(-3,-5); (3,-5) **\dir{-}?(0)*\dir{<};
(0,-1)*{\smallYlo};
\endxy \text{ \ \ or \ \ }
\xy
(-3,5); (3,5) **\dir{-}?(1)*\dir{>};
(-3,3); (-3,-3) **\crv{(3,0)}?(1)*\dir{>};
(-3,-5); (3,-5) **\dir{-}?(0)*\dir{<};
\endxy
\]
and
\[
W_R = 
\xy  
(-3,5); (3,5) **\dir{-}?(1)*\dir{>};
(-3,-3); (3,-3) **\dir{-}?(1)*\dir{>};
(-3,-5); (3,-5) **\dir{-}?(0)*\dir{<};
(0,1)*{\smallYro};
\endxy \text{ \ \ or \ \ }
\xy  
(-3,5); (3,5) **\dir{-}?(1)*\dir{>};
(-3,3); (3,3) **\dir{-}?(0)*\dir{<};
(-3,-5); (3,-5) **\dir{-}?(0)*\dir{<};
(0,-1)*{\smallYri};
\endxy \text{ \ \ or \ \ }
\xy
(-3,5); (3,5) **\dir{-}?(1)*\dir{>};
(3,3); (3,-3) **\crv{(-3,0)}?(1)*\dir{>};
(-3,-5); (3,-5) **\dir{-}?(0)*\dir{<};
\endxy \ \
\]
where the strands involved in the $Y$-webs and $U$-webs have multiplicity one (and the other strands can have 
higher multiplicities).
\end{prop}

\begin{proof}
We proceed via induction on $m$,
noting that the result holds 
trivially in the cases $m=0$ or $n=0$. 
We have
\[
\left \llbracket \twistmn{m}{n}{} \right \rrbracket_s \simeq 
\left \llbracket \onewrapmn{m-1}{n}{} \right \rrbracket_s \bullet 
\left \llbracket 
\xy
(-14,6)*{}; (8,6)*{} **\dir{-}?(1)*\dir{>};
(0,-1)*{\twistmn{m-1}{n}{}};
\endxy
\right \rrbracket_s
\]
where the first tangle diagram on the right hand side indicates one strand wrapping around the others and the 
second denotes the tensor product of a strand with a full twist on $m+n-1$ strands, directed as indicated. By induction, 
the second complex has the desired form. 
Since 
the composition of two webs of the form \eqref{goodwebform2} is isomorphic to a $q$-linear direct sum of non-elliptic 
webs of this form, it suffices to show that the first complex has the desired form. We have
\[
\left \llbracket \onewrapmn{m}{n}{} \right \rrbracket_s = 
\left \llbracket
\xy
(-4,-3)*{};(2,-3)*{} **\dir{-}?(0)*\dir{<};
(2,2)*{\smallhalftwistmn{m-1}{}};
(4,-3.5)*{^{_{n}}};
\endxy
\right \rrbracket_s \bullet
\left \llbracket
\xy
(-5,4)*{};(5,4)*{} **\dir{-}?(1)*\dir{>};
(0,-1)*{\smalltwotwistoneone};
(8,4)*{^{_{m-1}}};
(6,-4)*{^{_{n}}};
\endxy
\right \rrbracket_s \bullet
\left \llbracket
\xy
(-4,-3)*{};(2,-3)*{} **\dir{-}?(0)*\dir{<};
(2,2)*{\smallhalftwistmn{}{m-1}};
(4,-3.5)*{^{_{n}}};
\endxy
\right \rrbracket_s
\]
where the middle term on the right side is the tensor product of $m-1$ strands with a single strand wrapping 
around $n$ strands (which do not twist themselves - note the subtle difference in notation!). Since the two outside 
terms on the right side have the desired form, it now suffices to show that 
\[
\left \llbracket
\xy
(0,0)*{\smalltwotwistoneone};
(6,-3)*{^{_{n}}};
\endxy
\right \rrbracket_s
\]
has this form. We claim that this complex
is homotopy equivalent to a complex
\begin{equation}\label{claimeq}
\xymatrix{\left(\idmn{}{n}\right) \ar[r] & D^{1} \ar[r] &D^{2} \ar[r] & \cdots}
\end{equation}
where each web in $D^{i}$ takes the form
\begin{equation}\label{Dweb}
\xy
(0,13)*{\vdots};
(-30,20)*{}; (30,20)*{} **\dir{-}?(.5)*\dir{<};
(-40,17)*{\Big(\smooth \text{ \ or \ } \singular \Big)};
(40,17)*{\Big(\smooth \text{ \ or \ } \singular \Big)};
(-15,0)*{\Big(\smooth \text{ \ or \ } \singular \Big)};
(15,0)*{\Big(\smooth \text{ \ or \ } \singular \Big)};
(-30,14)*{}; (-29,11.75)*{} **\dir{-};
(-26,5.25)*{}; (-25,3)*{} **\dir{-};
(-28,9.7)*{\cdot};
(-27.5,8.6)*{\cdot};
(-27,7.5)*{\cdot};
(30,14)*{}; (29,11.75)*{} **\dir{-};
(26,5.25)*{}; (25,3)*{} **\dir{-};
(28,9.7)*{\cdot};
(27.5,8.6)*{\cdot};
(27,7.5)*{\cdot};
(-5,3)*{}; (5,3)*{} **\dir{-}?(.5)*\dir{<};
(-5,-3)*{}; (-50,-10) **\crv{(5,-10)}?(1)*\dir{>};
(5,-3)*{}; (50,-10) **\crv{(-5,-10)}?(.7)*\dir{<};
(-50,-15)*{}; (50,-15)*{} **\dir{-}?(0)*\dir{<};
(-25,-2)*{}; (-50,-2)*{} **\dir{-}?(1)*\dir{>};
(25,-2)*{}; (50,-2)*{} **\dir{-}?(.5)*\dir{<};
(52,-15)*{_{l}};
\endxy
\end{equation}
for $l\geq0$.

We proceed via induction on $n$. The $n=1$ case follows from Lemma \ref{m=1=n}. We now compute
\[
\left \llbracket
\xy
(0,0)*{\smalltwotwistoneone};
(6,-3)*{^{_{n}}};
\endxy
\right \rrbracket_s =
\left \llbracket \
\xy 
(-2,2)*{\crossdown};
(-5,-3)*{}; (1,-3)*{} **\dir{-}?(0)*\dir{<};
(5,-3.5)*{^{_{n-1}}};
\endxy
\right \rrbracket_s \bullet
\left \llbracket \
\xy
(-5,4)*{}; (5,4)*{} **\dir{-}?(0)*\dir{<};
(0,-1)*{\smalltwotwistoneone};
(8,-3.5)*{^{_{n-1}}};
\endxy
\right \rrbracket_s \bullet
\left \llbracket \
\xy
(-2,2)*{\crossup};
(-5,-3)*{}; (1,-3)*{} **\dir{-}?(0)*\dir{<};
(5,-3.5)*{^{_{n-1}}};
\endxy
\right \rrbracket_s
\]
\[\simeq \left \llbracket
\xy
(-2,2)*{\crossdown};
(-5,-3)*{}; (1,-3)*{} **\dir{-}?(0)*\dir{<};
(5,-3.5)*{^{_{n-1}}};
\endxy
\right \rrbracket_s \bullet
\cone\left( \xymatrix{\left(\ \xy
(-3,2)*{}; (3,2)*{} **\dir{-}?(0)*\dir{<};
(-3,0)*{}; (3,0)*{} **\dir{-}?(1)*\dir{>};
(-3,-2)*{}; (3,-2)*{} **\dir{-}?(0)*\dir{<};
(6,-2.5)*{^{_{n-1}}};
\endxy \ \right) \ar[r] & 
\xy
(-5,2)*{}; (-1,2)*{} **\dir{-}?(0)*\dir{<};
(0,-1)*{D^{\cdot}[-1]};
\endxy} \right) [1] \bullet
\left \llbracket \
\xy
(-2,2)*{\crossup};
(-5,-3)*{}; (1,-3)*{} **\dir{-}?(0)*\dir{<};
(5,-3.5)*{^{_{n-1}}};
\endxy
\right \rrbracket_s .
\]
Here $\xy
(-2,2)*{}; (2,2)*{} **\dir{-}?(0)*\dir{<};
(0,0)*{D^{\cdot}};
\endxy$ denotes the tensor product of the complex $D^{\cdot}$ with a single strand.
By Proposition \ref{tensorcone}, the above complex is homotopy equivalent to
\[
\cone \left( \xymatrix{
\left \llbracket
\xy
(0,2)*{\smalltwotwistoneone};
(-5,-2)*{}; (5,-2)*{} **\dir{-}?(0)*\dir{<};
(8,-2.5)*{^{_{n-1}}};
\endxy
\right \rrbracket_s \ar[r]
&
\left \llbracket
\xy
(-2,2)*{\crossdown};
(-5,-3)*{}; (1,-3)*{} **\dir{-}?(0)*\dir{<};
(5,-3.5)*{^{_{n-1}}};
\endxy
\right \rrbracket_s \bullet
\xy
(-5,2)*{}; (-1,2)*{} **\dir{-}?(0)*\dir{<};
(0,-1)*{D^{\cdot}[-1]};
\endxy \bullet
\left \llbracket \
\xy
(-2,2)*{\crossup};
(-5,-3)*{}; (1,-3)*{} **\dir{-}?(0)*\dir{<};
(5,-3.5)*{^{_{n-1}}};
\endxy
\right \rrbracket_s}
\right)
\]
which has the form \eqref{claimeq}. 
Since reducing a web of the form \eqref{Dweb} gives webs of the form 
\eqref{goodwebform2}, the result follows.
\end{proof}

There thus exists a map 
$\left \llbracket \smalltwistmn{m}{n}{} \right \rrbracket_s \stackrel{g}{\longrightarrow}
 \left \llbracket \idmn{m}{n} \right \rrbracket_s$
which we use to construct the inverse system 
\begin{equation}
\mathbf{T}_w = \xymatrix{\left\llbracket \idmn{m}{n} \right\rrbracket_s 
& \left\llbracket \smalltwistmn{m}{n}{} \right\rrbracket_s \ar[l]_-{g_0}
& \left\llbracket \smalltwistmn{m}{n}{2} \right\rrbracket_s \ar[l]_-{g_1}
& \cdots \ar[l]
}
\end{equation}
for $w=(+\cdots+-\cdots-)$; $g_k$ is defined as
\[
\xymatrix{\left\llbracket \smalltwistmn{m}{n}{} \right\rrbracket_s 
\bullet
\left\llbracket 
\smalltwistmn{m}{n}{k} \right\rrbracket_s
\ar[rr]^-{g\bullet id} 
& & \left\llbracket \idmn{m}{n} \right\rrbracket_s
\bullet
\left\llbracket 
\smalltwistmn{m}{n}{k} \right\rrbracket_s
}.
\]
In the case $w=(+\cdots+)$ both the proof that the limit of the system exists 
(i.e. that the system is Cauchy) and that the limit lies in $\K^{\angle}(\mathcal{F})$ involve analysis of the 
complexes $\cone(g_k)$. We thus combine them into one result.

\begin{prop}\label{hardproof}
Let $w=(+\cdots+-\cdots-)$. The inverse system $\mathbf{T}_w$ is Cauchy and its limit lies in 
$\K^{\angle}(\mathcal{F})$.
\end{prop}
Of course, since limits are unique only up to homotopy, we mean that there is a representative of the homotopy class 
of the limit which lies in $\K^{\angle}(\mathcal{F})$.

\begin{proof}
As before, it will suffice to construct complexes $C_k^{\cdot}$ satisfying the following conditions.
\begin{enumerate}
 \item $C_k^{\cdot}[-1] \simeq \cone(g_k)$.
\item $C_k^{\cdot}$ is $\Oh^h(2k+1)$.
\item Every web appearing in $C_k^{\cdot}$ is non-elliptic and takes the form \eqref{goodwebform2}.
\item $C_k^{\cdot}$ is supported in $A_{1/M}$ for some fixed $M$. 
Moreover, if a web $W$ in $C_k^{\cdot}$ has $r_W$ internal faces then that web is supported in $A_{1/M}\{r_W\}$.
\end{enumerate}
We shall see that it suffices to take $M=2(m+n)^2$, so fix this value.

We begin by reconsidering the complex $\left \llbracket \smalltwistmn{m}{n}{} \right \rrbracket_s$ and showing that
it is homotopic to a complex of the form
\[
\xymatrix{ \left(\idmn{m}{n}\right) \ar[r] & C^1 \ar[r] & C^2 \ar[r] &\cdots}
\]
which is supported in $A_{1/M}$ and where each web in $C^i$ takes the form \eqref{goodwebform2}. 
Using only $R3$ moves we find that the tangle $\smalltwistmn{m}{n}{}$ is isotopic to the tangle
\begin{equation}\label{tangle1}
\onewrapmn{m-1}{n} \bullet
\xy
(-10,5)*{}; (3,5)*{} **\dir{-}?(1)*\dir{>};
(0,-2)*{\medonewrapmn{m-2}{n}};
\endxy \bullet \cdots \bullet
\xy
(-7,5)*{};(8,5)*{} **\dir{-}?(1)*\dir{>};
(0,-1)*{\medtwotwistoneone};
(12,5)*{^{_{m-1}}};
(8,-5)*{^{_{n}}};
\endxy \bullet
\xy
(-9,5)*{};(6,5)*{} **\dir{-}?(1)*\dir{>};
(0,-1)*{\medtwistn{n}{}};
(9,5)*{_{m}};
\endxy
\end{equation}
where all the strands directed to the right wrap (one by one, starting with the top strand) 
around the strands directed to the left at the beginning. We consider the complex assigned to 
this tangle. We next use Lemma \ref{m=1=n} and Proposition \ref{tensorcone} to express this complex in terms of complexes 
assigned to tangles obtained from \eqref{tangle1} by replacing the tangle
\[
\xy
(-7,5)*{};(8,5)*{} **\dir{-}?(1)*\dir{>};
(0,-1)*{\medtwotwistoneone};
(12,5)*{^{_{m-1}}};
(8,-5)*{^{_{n}}};
\endxy
\]
with tangles of the form 
\[
\xy 
(-10,7.5)*{}; (10,7.5)*{} **\dir{-}?(1)*\dir{>};
(-10,5)*{}; (-10,-5)*{} **\crv{(5,0)}?(1)*\dir{>} \POS?(.5)*{\hole}="h1";
(-10,0)*{}; "h1" **\dir{-}?(0)*\dir{<};
(-10,-7.5)*{}; (10,-7.5)*{} **\dir{-}?(0)*\dir{<};
"h1"; (10,0)*{} **\dir{-};
(3,0)*{\hole}="h2";
(10,-5)*{}; "h2" **\crv{(5,-5)};
"h2"; (10,5)*{} **\crv{(5,5)}?(1)*\dir{>};
(14,7.5)*{^{_{m-1}}};
(15,-8)*{^{_{n-p-1}}};
(12,0)*{^{_{p}}};
\endxy.
\]
We can then use $R2$ moves to slide the left strand which `turns back' through the entire tangle. We now repeat the procedure
for all terms in \eqref{tangle1}, moving leftward from
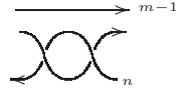
\[
\xy
(-7,5)*{};(8,5)*{} **\dir{-}?(1)*\dir{>};
(0,-1)*{\medtwotwistoneone};
(12,5)*{^{_{m-1}}};
(8,-5)*{^{_{n}}};
\endxy
\]
and one by one expressing complexes assigned to tangles of the form 
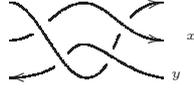
\[
\onewrapmn{x}{y}
\]
in terms of complexes assigned to
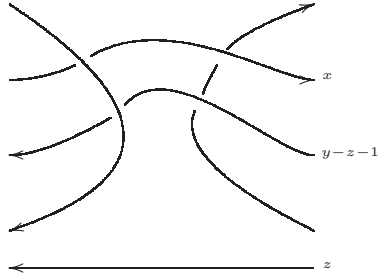
\[
\xy
(-20,-15)*{}; (20,-15)*{} **\dir{-}?(0)*\dir{<};
(-20,20)*{}; (-20,-10)*{} **\crv{(10,0)} \POS?(.2)*{\hole}="h1"
\POS?(.4)*{\hole}="h2" ?(1)*\dir{>};
(-20,10)*{}; "h1" **\crv{(-15,10)};
(-20,0)*{}; "h2" **\crv{(-15,0)}?(0)*\dir{<};
"h1"; (20,10)*{} **\crv{(0,20)&(15,10)}?(1)*\dir{>};
"h2"; (20,0)*{} **\crv{(0,15)&(15,0)};
(5,7)*{\hole}="h4";
(8,13)*{\hole}="h3";
(20,-10)*{}; "h4" **\crv{(0,0)};
"h4"; "h3" **\crv{(6,10)};
"h3"; (20,20)*{} **\crv{(9,16)}?(1)*\dir{>};
(22,10)*{^{_{x}}};
(25,0)*{^{_{y-z-1}}};
(22,-15)*{^{_{z}}};
\endxy  
\]
then sliding the left strand which turns around through the tangle. 
In the end we find that the terms in the complex assigned to \eqref{tangle1} come from the 
complexes assigned to tangles $\tau$ which, for example, take the form
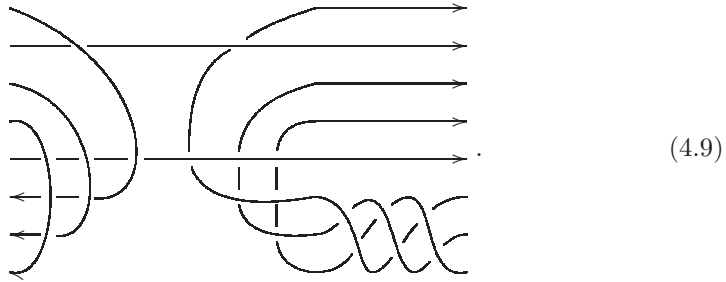
\begin{equation}\label{tangle2}
 \xy
(-20,.85)*{\turnback};
(10,0)*{\threetwist};
\endxy \ .
\end{equation}
It hence suffices to show that $\left \llbracket \tau \right \rrbracket_s$, suitably shifted to
take into account the shifts in quantum and homological degree that arise from the Gaussian elimination 
homotopies and $R2$ moves, is supported in $A_{1/M}$. 

We first establish some notation. Noting that
\[
\left \llbracket \
\xy 0;/r.1pc/:
(-5,5)*{}="a";
(5,5)*{}="b";
(-5,-5)*{}="c";
(5,-5)*{}="d";
"a"; "c" **\crv{(15,0)} \POS?(.15)*{\hole}="h1" \POS?(.85)*{\hole}="h2";
"b"; "h1" **\crv{(5,5)};
"d"; "h2" **\crv{(5,-5)};
"h1"; "h2" **\crv{(-10,0)};
\endxy \
\right \rrbracket_s
\simeq \left \llbracket \ \smooth \ \right \rrbracket_s[1]\{1\}
\]
for all possible orientations of the strands, we call an $R2$ move which reduces the number of crossings 
in a tangle a \emph{good} $R2$ move. These moves are `good' in the sense that the corresponding Gaussian 
elimination homotopy equivalence yields a complex whose support is (properly) contained in that of the original 
complex.

Now, if $m'$ is the number of strands which turn back on the left side of $\tau$ then, assuming $m'>1$, 
we make at least 
\[
2\sum_{i=1}^{m'-1}i = (m')^2 - m'
\]
good $R2$ moves to arrive at such a presentation. The right hand side of this formula also works in the cases 
when $m'=0,1$.

Next, let $l$ denote half the total number of negative crossings in $\tau$. We can apply $l$ good $R2$ moves to 
eliminate all negative crossings involving the strands which turn back on the right to produce a tangle of the form 
\begin{equation}\label{tangleform}
\rho =
\xy
(10,-5)*{}; (10,5)*{} **\crv{(0,-5)&(0,5)}?(.45)*\dir{>};
(-10,5)*{}; (-10,-5)*{} **\crv{(0,5)&(0,-5)}?(.45)*\dir{>};
(-12.5,5)*{_{\sigma_1}};
(-10,6.5)*{}; (-10,3.5)*{} **\dir{-};
(-10,6.5)*{}; (-15,6.5)*{} **\dir{-};
(-15,6.5)*{}; (-15,3.5)*{} **\dir{-};
(-15,3.5)*{}; (-10,3.5)*{} **\dir{-};
(-12.5,-6)*{\sigma_2};
(-10,-8.5)*{}; (-10,-3.5)*{} **\dir{-};
(-10,-8.5)*{}; (-15,-8.5)*{} **\dir{-};
(-15,-8.5)*{}; (-15,-3.5)*{} **\dir{-};
(-15,-3.5)*{}; (-10,-3.5)*{} **\dir{-};
(-10,-7)*{}; (10,-7)*{} **\dir{-}?(.5)*\dir{<};
(12.5,-6)*{P_3};
(10,-8.5)*{}; (10,-3.5)*{} **\dir{-};
(10,-8.5)*{}; (15,-8.5)*{} **\dir{-};
(15,-8.5)*{}; (15,-3.5)*{} **\dir{-};
(15,-3.5)*{}; (10,-3.5)*{} **\dir{-};
(12.5,6)*{P_2};
(10,8.5)*{}; (10,3.5)*{} **\dir{-};
(10,8.5)*{}; (15,8.5)*{} **\dir{-};
(15,8.5)*{}; (15,3.5)*{} **\dir{-};
(15,3.5)*{}; (10,3.5)*{} **\dir{-};
(-22.5,6)*{P_1};
(-20,8.5)*{}; (-20,3.5)*{} **\dir{-};
(-20,8.5)*{}; (-25,8.5)*{} **\dir{-};
(-25,8.5)*{}; (-25,3.5)*{} **\dir{-};
(-25,3.5)*{}; (-20,3.5)*{} **\dir{-};
(-20,5)*{}; (-15,5)*{} **\dir{-}?(.6)*\dir{>};
(-20,7.5)*{}; (10,7.5)*{} **\dir{-}?(.5)*\dir{>};
(-15,-6)*{}; (-30,-6)*{} **\dir{-}?(1)*\dir{>};
(-30,6)*{}; (-25,6)*{} **\dir{-}?(.6)*\dir{>};
(15,6)*{}; (23,6)*{} **\dir{-}?(1)*\dir{>};
(15,-6)*{}; (23,-6)*{} **\dir{-}?(.4)*\dir{<};
(25,6)*{m};
(25,-6)*{n};
(0,0)*{_{m'}};
\endxy
\end{equation}
where the $P_i$ are positive braids. The $\sigma_i$ are negative braids, but of a particular sort - 
every braid determines an element of the symmetric group and these braids are the simplest ones
corresponding to their particular element (i.e. there is no twisting). The braid $\sigma_2$ has the further 
property that all crossings consist of a strand leaving $P_3$ crossing under a strand leaving $\sigma_1$.

All of the internal faces in the $0$-resolution of \eqref{tangleform} come from the tangle
\begin{equation}
N=
\xy
(-10,5)*{}; (-10,-5)*{} **\crv{(0,5)&(0,-5)}?(.45)*\dir{>};
(-20,5)*{}; (-15,5)*{} **\dir{-}?(.6)*\dir{>};
(-12.5,5)*{_{\sigma_1}};
(-10,6.5)*{}; (-10,3.5)*{} **\dir{-};
(-10,6.5)*{}; (-15,6.5)*{} **\dir{-};
(-15,6.5)*{}; (-15,3.5)*{} **\dir{-};
(-15,3.5)*{}; (-10,3.5)*{} **\dir{-};
(-12.5,-6)*{\sigma_2};
(-10,-8.5)*{}; (-10,-3.5)*{} **\dir{-};
(-10,-8.5)*{}; (-15,-8.5)*{} **\dir{-};
(-15,-8.5)*{}; (-15,-3.5)*{} **\dir{-};
(-15,-3.5)*{}; (-10,-3.5)*{} **\dir{-};
(-15,-6)*{}; (-20,-6)*{} **\dir{-}?(1)*\dir{>};
(-10,-7)*{}; (5,-7)*{} **\dir{-}?(.5)*\dir{<};
(-20,7.5)*{}; (5,7.5)*{} **\dir{-}?(1)*\dir{>};
(10,7.5)*{_{m-m'}};
(10,-7)*{_{n-m'}};
(0,0)*{_{m'}};
\endxy
\end{equation}
and it follows from inspection that if $r$ is the number of such faces then
\[
r \leq l-1
\]
unless $l=0$ in which case there are no negative crossings in \eqref{tangleform} and no internal faces 
in the zero resolution. We assume for now that we are not in the exception case $l=0$.

For later use, we'll bound the values for $l$. The number of crossings in $\sigma_1$ is bounded by
\[
\binom{m'}{2} = \frac{1}{2} m'(m'-1),
\]
the length of the longest element in the symmetric group. The remaining crossings come from $\sigma_2$ 
where some of the $n-m'$ strands pass underneath some of the $m'$ strands leaving $\sigma_1$, 
producing at most $m'(n-m')$ crossings. We thus have
\[
l \leq \frac{1}{2} m'(m'-1) + m'(n-m') = \frac{1}{2}m'(2n-m'-1) < mn.
\] 

Lemma \ref{lem2} gives that any web $W$ appearing in $\left \llbracket \rho \right \rrbracket_s$ is supported in
\[
A_{1/c_+}\{-r-1+r_W\} \subseteq A_{1/c_+}\{-l+1-1+r_W\}
\]
where $r_W$ is the number of internal faces in $W$ and $c_+$ is the number of positive crossings in $\rho$. Since 
\[c_+ \leq (m+n)(m+n-1) < M
\]
we see such a web is supported in 
\[
A_{1/M}\{-l+r_W\}.
\]
Considering all the shifts due to Gaussian elimination homotopies and good $R2$ moves, we see that the contribution to 
$\left \llbracket \smalltwistmn{m}{n}{} \right \rrbracket_s$ is given by
\[
\left \llbracket \rho \right \rrbracket_s [(m')^2-m'+a]\{(m')^2-m'+a\}[m'+b]\{2m'+2b\}[l]\{l\}
\]
where $a \leq m^2 - (m')^2 - m + m' \leq m^2$ and $b \leq m' \leq m$.
Indeed, the shifts of $[1]\{1\}$ come from the good $R2$ moves and the shifts of $[1]\{2\}$ come 
via Lemma \ref{m=1=n} from the strands which turn back. The web $W$ is hence supported in
\[
A_{1/M}[(m')^2+a+b+l]\{(m')^2+m'+a+2b+r_W\}.
\]
We have
\[
(m')^2+a+b+l \leq 2m^2 + m + mn \leq M
\]
so we see that each web $W$ in the contribution to 
$\left \llbracket \smalltwistmn{m}{n}{} \right \rrbracket_s$ 
coming from $\left \llbracket \rho \right \rrbracket_s$ with $l \neq 0$ is supported in $A_{1/M}\{r_W\}$ 
(we have used the fact that $l \neq 0$ implies $m' \geq 1$ so $(m')^2+m'+a+2b \geq 1$).

In the case that $l=0$, i.e. there are no negative crossings in $\tau$, 
it follows from Lemma \ref{lem2} that any web $W$ contributing to 
$\left \llbracket \smalltwistmn{m}{n}{} \right \rrbracket_s$ is supported in $A_{1/M}\{r_W\}$.

We have thus shown that 
\[
\left \llbracket \smalltwistmn{m}{n}{} \right \rrbracket_s 
\simeq 
\xymatrix{ \left(\idmn{m}{n}\right) \ar[r] & C^1 \ar[r] & C^2 \ar[r] &\cdots}
\]
where each web $W$ in $C^i$ is supported in $A_{1/M}\{r_W\}$ and comes from the tangles 
\eqref{tangleform}. It is easy to see that all such tangles take the form \eqref{goodwebform2}. 
It will be useful for our further considerations to note that every web of the form \eqref{goodwebform2} 
is of the form 
\begin{equation}\label{goodwebform3}
W' \bullet R  \text{ \ \ , \ \ }
W' \bullet
\xy  
(-3,5); (3,5) **\dir{-}?(1)*\dir{>};
(-3,3); (3,3) **\dir{-}?(0)*\dir{<};
(-3,-5); (3,-5) **\dir{-}?(0)*\dir{<};
(0,-1)*{\smallYri};
\endxy \text{ \ , \ or \ \ }
W' \bullet 
\xy  
(-3,5); (3,5) **\dir{-}?(1)*\dir{>};
(-3,-3); (3,-3) **\dir{-}?(1)*\dir{>};
(-3,-5); (3,-5) **\dir{-}?(0)*\dir{<};
(0,1)*{\smallYro};
\endxy 
\end{equation}
where $R$ is a non-elliptic web of the form
\[
\xy
(-5,5); (5,5) **\dir{-}?(1)*\dir{>};
(0,0)*{S};
(1,1.5); (5,1.5) **\dir{-}?(1)*\dir{>};
(1,-1.5); (5,-1.5) **\dir{-}?(.5)*\dir{<};
(-5,-5); (5,-5) **\dir{-}?(0)*\dir{<};
(9,5)*{_{m-u}};
(9,-5)*{_{n-u}};
(7,1.5)*{_{u}};
(7,-1.5)*{_{u}};
\endxy
\]
and $S$ is a non-elliptic web with no left boundary. 
While the latter two webs in \eqref{goodwebform3} do not preclude the first, 
we employ the convention that
if we claim a web has either of these two forms it is implicit that it does not have the first. 
We will also assume that $u$ is chosen maximal for the first type of web.

We now proceed with our construction of the complexes $C_k^{\cdot}$. Let $C_0^{\cdot}$ be the complex constructed above 
with the degree zero term truncated off. We now construct $C_k^{\cdot}$ assuming we have constructed $C_{k-1}^{\cdot}$;
we have
\begin{align*}
\cone(g_k) [1] &= \xymatrix{ \cone(g_{k-1})[1] \bullet \left \llbracket \smalltwistmn{m}{n}{} \right \rrbracket_s} \\
&\simeq \xymatrix{C_{k-1}^{\cdot} \bullet \left \llbracket \smalltwistmn{m}{n}{} \right \rrbracket_s} .
\end{align*}
Since each web $W$ in $C_{k-1}^{\cdot}$ has the form \eqref{goodwebform3} (and also \eqref{goodwebform2}), is supported 
in $A_{1/M}\{r_W\}$, and has homological degree at least $1+2(k-1)$, it suffices to show that each of the complexes
\[
\left \llbracket
W' \bullet R
\bullet \smalltwistmn{m}{n}{}
\right \rrbracket_s \text{ \ \ , \ \ }
\left \llbracket
W' \bullet
\xy  
(-3,5); (3,5) **\dir{-}?(1)*\dir{>};
(-3,3); (3,3) **\dir{-}?(0)*\dir{<};
(-3,-5); (3,-5) **\dir{-}?(0)*\dir{<};
(0,-1)*{\smallYri};
\endxy
\bullet \smalltwistmn{m}{n}{}
\right \rrbracket_s \text{ \ \ , \ \ } 
\left \llbracket
W' \bullet 
\xy  
(-3,5); (3,5) **\dir{-}?(1)*\dir{>};
(-3,-3); (3,-3) **\dir{-}?(1)*\dir{>};
(-3,-5); (3,-5) **\dir{-}?(0)*\dir{<};
(0,1)*{\smallYro};
\endxy
\bullet \smalltwistmn{m}{n}{}
\right \rrbracket_s 
\]
is homotopic to a complex with 
terms of the form \eqref{goodwebform2}, minimal homological degree $2$, and with webs $V$ supported in
$A_{1/M}\{r_V-r_W\}$ with $W = W'\bullet R$ or $W'$.
We shall analyze each case separately.\newline 

\noindent
\underline{\textbf{Case 1:}}
$\left \llbracket
W' \bullet R
\bullet \smalltwistmn{m}{n}{}
\right \rrbracket_s$ \newline

We compute
\begin{equation}\label{delta1}
\left \llbracket
W' \bullet R
\bullet \smalltwistmn{m}{n}{}
\right \rrbracket_s =
\left \llbracket
W' \bullet 
\xy
(-5,5); (5,5) **\dir{-}?(1)*\dir{>};
(0,0)*{S};
(1,1.5); (5,1.5) **\dir{-}?(1)*\dir{>};
(1,-1.5); (5,-1.5) **\dir{-}?(.5)*\dir{<};
(-5,-5); (5,-5) **\dir{-}?(0)*\dir{<};
(9,5)*{_{m-u}};
(9,-5)*{_{n-u}};
(7,1.5)*{_{u}};
(7,-1.5)*{_{u}};
\endxy \bullet \smalltwistmn{m}{n}{}
\right \rrbracket [c_{-}] \{3c_{-}-2c_{+}\}
\end{equation}
which is homotopy equivalent to the complex
\begin{equation}\label{delta2}
\left \llbracket 
W' \bullet \smalltwistmn{ \ \ \ \ m-u}{ \ \ \ \ n-u}{} \bullet
\xy
(-5,5); (5,5) **\dir{-}?(1)*\dir{>};
(0,0)*{S};
(1,1.5); (5,1.5) **\dir{-}?(1)*\dir{>};
(1,-1.5); (5,-1.5) **\dir{-}?(.5)*\dir{<};
(-5,-5); (5,-5) **\dir{-}?(0)*\dir{<};
(9,5)*{_{m-u}};
(9,-5)*{_{n-u}};
(7,1.5)*{_{u}};
(7,-1.5)*{_{u}};
\endxy \bullet
\xy
(-7,5); (3,5) **\dir{-}?(1)*\dir{>};
(-7,-5); (3,-5) **\dir{-}?(0)*\dir{<};
(0,0)*{\smalltwistmn{u}{u}{}}
\endxy
\right \rrbracket  [c_{-}-\frac23 \Delta] \{3c_{-}-2c_{+}-\frac83 \Delta\}
\end{equation}
where $c_{\pm}$ is the number of $\pm$-crossings in $\smalltwistmn{m}{n}{}$ and $\Delta$ is the change in writhe 
between the right side of \eqref{delta1} and \eqref{delta2}. The shifts involving $\Delta$ can be deduced from 
\eqref{catknottedspider}.

We have the formulae
\begin{align*}
c_{+} &= m^2-m+n^2-n \\
c_{-} &= 2mn
\end{align*}
and $\Delta = 0$. 
Taking into account the shifts due to the changes in the number of crossings, this gives that
$
\left \llbracket
W' \bullet R
\bullet \smalltwistmn{m}{n}{}
\right \rrbracket_s
$
is homotopy equivalent to
\[ 
\left \llbracket
W' \bullet \smalltwistmn{ \ \ \ \ m-u}{ \ \ \ \ n-u}{} \bullet
\xy
(-5,5); (5,5) **\dir{-}?(1)*\dir{>};
(0,0)*{S};
(1,1.5); (5,1.5) **\dir{-}?(1)*\dir{>};
(1,-1.5); (5,-1.5) **\dir{-}?(.5)*\dir{<};
(-5,-5); (5,-5) **\dir{-}?(0)*\dir{<};
(9,5)*{_{m-u}};
(9,-5)*{_{n-u}};
(7,1.5)*{_{u}};
(7,-1.5)*{_{u}};
\endxy \bullet
\xy
(-7,5); (3,5) **\dir{-}?(1)*\dir{>};
(-7,-5); (3,-5) **\dir{-}?(0)*\dir{<};
(0,0)*{\smalltwistmn{u}{u}{}}
\endxy
\right \rrbracket_s [2u(m+n-2u)]\{2u(m+n-2u)\}.
\]
Similarly, taking into account the change in writhe, we compute
\[
\left \llbracket
\xy
(0,0)*{S};
(1,1.5); (5,1.5) **\dir{-}?(1)*\dir{>};
(1,-1.5); (5,-1.5) **\dir{-}?(.5)*\dir{<};
(7,1.5)*{_{u}};
(7,-1.5)*{_{u}};
\endxy \bullet
\smalltwistmn{u}{u}{}
\right \rrbracket_s \simeq
\left \llbracket
\xy
(0,0)*{S};
(1,1.5); (5,1.5) **\dir{-}?(1)*\dir{>};
(1,-1.5); (5,-1.5) **\dir{-}?(.5)*\dir{<};
(7,1.5)*{_{u}};
(7,-1.5)*{_{u}};
\endxy
\right \rrbracket_s [2u^2]\{2u(u+2)\}
\]
and so 
$
\left \llbracket
W' \bullet R
\bullet \smalltwistmn{m}{n}{}
\right \rrbracket_s
$
is homotopy equivalent to
\begin{equation}\label{WR} 
\left \llbracket
W' \bullet \smalltwistmn{ \ \ \ \ m-u}{ \ \ \ \ n-u}{} \bullet
\xy
(-5,5); (5,5) **\dir{-}?(1)*\dir{>};
(0,0)*{S};
(1,1.5); (5,1.5) **\dir{-}?(1)*\dir{>};
(1,-1.5); (5,-1.5) **\dir{-}?(.5)*\dir{<};
(-5,-5); (5,-5) **\dir{-}?(0)*\dir{<};
(9,5)*{_{m-u}};
(9,-5)*{_{n-u}};
(7,1.5)*{_{u}};
(7,-1.5)*{_{u}};
\endxy
\right \rrbracket_s [2u(m+n-u)]\{2u(m+n-u+2)\}.
\end{equation}
Since $u\geq1$ and $m+n-u \geq 1$ and each web appearing in \eqref{WR} takes the 
form $W' \bullet R$, it suffices to show that every web $V$ appearing in this 
complex is supported in $A_{1/M}\{r_V-r_{(W'\bullet R)}\}$.

We have
\[
\supp \left(
\left \llbracket
W' \bullet \smalltwistmn{ \ \ \ \ m-u}{ \ \ \ \ n-u}{} \bullet
\xy
(-5,5); (5,5) **\dir{-}?(1)*\dir{>};
(0,0)*{S};
(1,1.5); (5,1.5) **\dir{-}?(1)*\dir{>};
(1,-1.5); (5,-1.5) **\dir{-}?(.5)*\dir{<};
(-5,-5); (5,-5) **\dir{-}?(0)*\dir{<};
(9,5)*{_{m-u}};
(9,-5)*{_{n-u}};
(7,1.5)*{_{u}};
(7,-1.5)*{_{u}};
\endxy
\right \rrbracket_s 
\right) \subseteq
\supp \left(
\left \llbracket
W' \bullet \smalltwistmn{ \ \ \ \ m-u}{ \ \ \ \ n-u}{}
\right \rrbracket_s 
\right)
\]
so we will consider complexes
\[
\left \llbracket
W' \bullet \smalltwistmn{\hat{m}}{\hat{n}}{}
\right \rrbracket_s 
\]
for $\hat{m}<m$ and $\hat{n}<n$. As before, we can express the complex assigned 
to the twist in terms of the complexes assigned to tangles 
\begin{equation}\label{tangleformhat}
\hat{\rho} =
\xy
(10,-5)*{}; (10,5)*{} **\crv{(0,-5)&(0,5)}?(.45)*\dir{>};
(-10,5)*{}; (-10,-5)*{} **\crv{(0,5)&(0,-5)}?(.45)*\dir{>};
(-12.5,5)*{_{\hat{\sigma_1}}};
(-10,6.5)*{}; (-10,3.5)*{} **\dir{-};
(-10,6.5)*{}; (-15,6.5)*{} **\dir{-};
(-15,6.5)*{}; (-15,3.5)*{} **\dir{-};
(-15,3.5)*{}; (-10,3.5)*{} **\dir{-};
(-12.5,-6)*{\hat{\sigma_2}};
(-10,-8.5)*{}; (-10,-3.5)*{} **\dir{-};
(-10,-8.5)*{}; (-15,-8.5)*{} **\dir{-};
(-15,-8.5)*{}; (-15,-3.5)*{} **\dir{-};
(-15,-3.5)*{}; (-10,-3.5)*{} **\dir{-};
(-10,-7)*{}; (10,-7)*{} **\dir{-}?(.5)*\dir{<};
(12.5,-6)*{\hat{P_3}};
(10,-8.5)*{}; (10,-3.5)*{} **\dir{-};
(10,-8.5)*{}; (15,-8.5)*{} **\dir{-};
(15,-8.5)*{}; (15,-3.5)*{} **\dir{-};
(15,-3.5)*{}; (10,-3.5)*{} **\dir{-};
(12.5,6)*{\hat{P_2}};
(10,8.5)*{}; (10,3.5)*{} **\dir{-};
(10,8.5)*{}; (15,8.5)*{} **\dir{-};
(15,8.5)*{}; (15,3.5)*{} **\dir{-};
(15,3.5)*{}; (10,3.5)*{} **\dir{-};
(-22.5,6)*{\hat{P_1}};
(-20,8.5)*{}; (-20,3.5)*{} **\dir{-};
(-20,8.5)*{}; (-25,8.5)*{} **\dir{-};
(-25,8.5)*{}; (-25,3.5)*{} **\dir{-};
(-25,3.5)*{}; (-20,3.5)*{} **\dir{-};
(-20,5)*{}; (-15,5)*{} **\dir{-}?(.6)*\dir{>};
(-20,7.5)*{}; (10,7.5)*{} **\dir{-}?(.5)*\dir{>};
(-15,-6)*{}; (-30,-6)*{} **\dir{-}?(1)*\dir{>};
(-30,6)*{}; (-25,6)*{} **\dir{-}?(.6)*\dir{>};
(15,6)*{}; (23,6)*{} **\dir{-}?(1)*\dir{>};
(15,-6)*{}; (23,-6)*{} **\dir{-}?(.4)*\dir{<};
(25,6)*{\hat{m}};
(25,-6)*{\hat{n}};
(0,0)*{_{\hat{m}'}};
\endxy \ .
\end{equation}
We must show that the complexes 
$\left \llbracket W' \bullet \hat{\rho} \right \rrbracket_s$, with appropriate shifts, 
are supported in $A_{1/M}$. Let $\hat{n}'$ be the number of strands which leave $\hat{P_3}$ 
and actually cross strands in $\hat{\sigma_2}$. Refining our earlier estimate, if $r$ is the number 
of internal faces in the zero resolution of $\rho$ then
\[
r \leq \max(0,\hat{l}-\hat{n})
\]
where $\hat{l}$ is the number of negative crossings in \eqref{tangleformhat}. The 
$0$-resolution of $W'\bullet \hat{\rho}$ hence has at most 
\[
r_{W'}+\hat{l}-\hat{n}'+2\hat{m}'+\hat{n}'-1 = r_{W'}+\hat{l}+2\hat{m}'-1
\]
internal faces where $r_{W'}$ is the number of internal faces in $W'$ (the additional 
$2\hat{m}'+\hat{n}'-1$ possible faces come from gluing $W'$ to $\hat{\rho}$).

Since we chose $u$ maximal, the smooth resolution of $W' \bullet \hat{\rho}$ has no closed 
components; Lemma \ref{lem2} gives that 
$\left \llbracket W' \bullet \hat{\rho} \right \rrbracket_s$ is 
supported in
\[
A_{1/M}\{-r_{W'}-\hat{l}-2\hat{m}'\}.
\]
The contribution to $\left \llbracket W' \bullet \smalltwistmn{m'}{n'}{} \right \rrbracket_s$ 
is obtained considering the shifts. As before we find it is given by
\[
\left \llbracket W' \bullet \hat{\rho} \right \rrbracket_s
[(\hat{m}')^2 - \hat{m}'+a]\{(\hat{m}')^2 - \hat{m}'+a\}
[\hat{m}'+b]\{2(\hat{m}'+b)\}[\hat{l}]\{\hat{l}\}
\]
with $a\leq \hat{m}^2$ and $b\leq \hat{m}$ so this complex is supported in
\[
A_{1/M}[(\hat{m}')^2+a+b+\hat{l}]\{(\hat{m}')^2-\hat{m}'+a+2b-r_{W'}\}.
\]
Since $(\hat{m}')^2+a+b+\hat{l} \leq M$ we see that this support is contained in 
\[
A_{1/M}\{-r_{W'}-1\}. 
\]

The contribution to $\left \llbracket W' \bullet R \bullet 
\smalltwistmn{m}{n}{} \right \rrbracket_s$ is supported in
\[
A_{1/M}[2u(m+n-u)]\{2u(m+n-u+2)-r_{W'}-1\}
\]
which is contained in $A_{1/M}\{-r_{W'}\}$ since $1 \leq u \leq \min(m,n)$ 
and 
\[
2u(m+n-u) \leq M.
\]
Moreover, applying the second statement from 
Lemma \ref{lem2} throughout the preceding argument, we see that any 
web $V$ in the contribution to 
\[ 
\left \llbracket W' \bullet R \bullet 
\smalltwistmn{m}{n}{} \right \rrbracket_s
\]
is supported in $A_{1/M}\{r_V-r_{(W'\bullet R)}\}$. \newline

\noindent
\underline{\textbf{Case 2:}}
$\left \llbracket
W' \bullet
\xy  
(-3,5); (3,5) **\dir{-}?(1)*\dir{>};
(-3,3); (3,3) **\dir{-}?(0)*\dir{<};
(-3,-5); (3,-5) **\dir{-}?(0)*\dir{<};
(0,-1)*{\smallYri};
\endxy
\bullet \smalltwistmn{m}{n}{}
\right \rrbracket_s$ \newline

We assume that $W' \bullet
\xy  
(-3,5); (3,5) **\dir{-}?(1)*\dir{>};
(-3,3); (3,3) **\dir{-}?(0)*\dir{<};
(-3,-5); (3,-5) **\dir{-}?(0)*\dir{<};
(0,-1)*{\smallYri};
\endxy$ does not take the form $W'\bullet R$ above. We begin by using $R3$ moves to 
express the tangle $\smalltwistmn{m}{n}{}$ as in \eqref{tangle1}. We then use $R3$ moves 
to pull the two strand twist which `lines up' with $\smallYri$ to the left through the tangle. 
Lemma \ref{untwist} gives
\[
\left \llbracket
W' \bullet
\xy  
(-3,5); (3,5) **\dir{-}?(1)*\dir{>};
(-3,3); (3,3) **\dir{-}?(0)*\dir{<};
(-3,-5); (3,-5) **\dir{-}?(0)*\dir{<};
(0,-1)*{\smallYri};
\endxy
\bullet \smalltwistmn{m}{n}{}
\right \rrbracket_s \simeq
\left \llbracket
W' \bullet
\xy  
(-3,5); (3,5) **\dir{-}?(1)*\dir{>};
(-3,3); (3,3) **\dir{-}?(0)*\dir{<};
(-3,-5); (3,-5) **\dir{-}?(0)*\dir{<};
(0,-1)*{\smallYri};
\endxy
\bullet T
\right \rrbracket_s [2]\{4\}
\]
where
\[
T = 
\onewrapmn{m-1}{n} \bullet
\xy
(-10,5)*{}; (3,5)*{} **\dir{-}?(1)*\dir{>};
(0,-2)*{\medonewrapmn{m-2}{n}};
\endxy \bullet \cdots \bullet
\xy
(-7,5)*{};(8,5)*{} **\dir{-}?(1)*\dir{>};
(0,-1)*{\medtwotwistoneone};
(12,5)*{^{_{m-1}}};
(8,-5)*{^{_{n}}};
\endxy \bullet
\xy
(-7,5)*{};(8,5)*{} **\dir{-}?(1)*\dir{>};
(2,0)*{\leftmedtwistds{n}{}};
(11,5)*{^{_{m}}};
\endxy \ .
\]
As in our analysis of the complex assigned to a single twist, we use 
Lemma \ref{m=1=n} starting with 
\[
\xy
(-7,5)*{};(8,5)*{} **\dir{-}?(1)*\dir{>};
(0,-1)*{\medtwotwistoneone};
(12,5)*{^{_{m-1}}};
(8,-5)*{^{_{n}}};
\endxy
\]
and moving left to express 
$\left \llbracket T \right \rrbracket_s$ in terms of complexes
$\left \llbracket \rho \right \rrbracket_s$ or 
$\left \llbracket \rho \bullet
\xy
(-3,5); (3,5) **\dir{-}?(1)*\dir{>};
(3,3); (-3,3) **\dir{-}?(1)*\dir{>};
(3,-5); (-3,-5) **\dir{-}?(1)*\dir{>};
(0,-1)*{\crossleft};
\endxy \
\right \rrbracket_s$ with $\rho$ as in \eqref{tangleform}.
In the latter case the top strand of the crossing in 
\[\xy
(-3,5); (3,5) **\dir{-}?(1)*\dir{>};
(3,3); (-3,3) **\dir{-}?(1)*\dir{>};
(3,-5); (-3,-5) **\dir{-}?(1)*\dir{>};
(0,-1)*{\crossleft};
\endxy\] 
turns back to the right in $\rho$;
this case arises due to the pair of non-twisting strands in $T$ which 
prevent the use of a good $R2$ move which in our previous analysis removed 
the crossing.

We hence analyze the complexes 
$\left \llbracket W' \bullet
\xy  
(-3,5); (3,5) **\dir{-}?(1)*\dir{>};
(-3,3); (3,3) **\dir{-}?(0)*\dir{<};
(-3,-5); (3,-5) **\dir{-}?(0)*\dir{<};
(0,-1)*{\smallYri};
\endxy \bullet \rho \right \rrbracket_s$
and 
$\left \llbracket W' \bullet
\xy  
(-3,5); (3,5) **\dir{-}?(1)*\dir{>};
(-3,3); (3,3) **\dir{-}?(0)*\dir{<};
(-3,-5); (3,-5) **\dir{-}?(0)*\dir{<};
(0,-1)*{\smallYri};
\endxy \bullet \rho \bullet
\xy
(-3,5); (3,5) **\dir{-}?(1)*\dir{>};
(3,3); (-3,3) **\dir{-}?(1)*\dir{>};
(3,-5); (-3,-5) **\dir{-}?(1)*\dir{>};
(0,-1)*{\crossleft};
\endxy \
\right \rrbracket_s$ noting that every web appearing takes the form 
\eqref{goodwebform2}. It suffices to consider the support of such 
complexes and we begin with the former (and slightly easier) case. 
The smooth resolution of 
$W' \bullet
\xy  
(-3,5); (3,5) **\dir{-}?(1)*\dir{>};
(-3,3); (3,3) **\dir{-}?(0)*\dir{<};
(-3,-5); (3,-5) **\dir{-}?(0)*\dir{<};
(0,-1)*{\smallYri};
\endxy \bullet \rho$ takes the form 
\[
W' \bullet
\xy  
(-3,5); (3,5) **\dir{-}?(1)*\dir{>};
(-3,3); (3,3) **\dir{-}?(0)*\dir{<};
(-3,-5); (3,-5) **\dir{-}?(0)*\dir{<};
(0,-1)*{\smallYri};
\endxy \bullet 
\xy
(-7,5); (7,5) **\dir{-}?(1)*\dir{>};
(12,5)*{_{m-m'}};
(-7,3); (-7,-3) **\crv{(0,0)}?(1)*\dir{>};
(0,0)*{_{m'}};
(7,-3); (7,3) **\crv{(0,0)}?(1)*\dir{>};
(7,-5); (-7,-5) **\dir{-}?(1)*\dir{>};
(12,-5)*{_{n-m'}};
\endxy
\]
which has no closed components (since $W' \bullet
\xy  
(-3,5); (3,5) **\dir{-}?(1)*\dir{>};
(-3,3); (3,3) **\dir{-}?(0)*\dir{<};
(-3,-5); (3,-5) **\dir{-}?(0)*\dir{<};
(0,-1)*{\smallYri};
\endxy \neq W' \bullet R$). The $0$-resolution of
$W' \bullet
\xy  
(-3,5); (3,5) **\dir{-}?(1)*\dir{>};
(-3,3); (3,3) **\dir{-}?(0)*\dir{<};
(-3,-5); (3,-5) **\dir{-}?(0)*\dir{<};
(0,-1)*{\smallYri};
\endxy \bullet \rho$
has at most
\[
r_{W'}+l-n'+2m'+n'-1 = r_{W'}+l+2m'-1
\]
internal faces, where $n'$ is the number of strands in $\rho$ 
leaving $P_3$ and crossing under strands in $\sigma_2$.

Lemma \ref{lem2} now gives that any web $V$ appearing in 
$\left \llbracket W' \bullet
\xy  
(-3,5); (3,5) **\dir{-}?(1)*\dir{>};
(-3,3); (3,3) **\dir{-}?(0)*\dir{<};
(-3,-5); (3,-5) **\dir{-}?(0)*\dir{<};
(0,-1)*{\smallYri};
\endxy \bullet \rho \right \rrbracket_s$ is supported in 
$A_{1/M}\{r_V-r_{W'}-l-2m'\}$. The contribution to 
$\left \llbracket W' \bullet
\xy  
(-3,5); (3,5) **\dir{-}?(1)*\dir{>};
(-3,3); (3,3) **\dir{-}?(0)*\dir{<};
(-3,-5); (3,-5) **\dir{-}?(0)*\dir{<};
(0,-1)*{\smallYri};
\endxy \bullet \smalltwistmn{m}{n}{} \right \rrbracket_s$ is computed considering 
the shifts and as before we see it is given by
\[
\left \llbracket W' \bullet
\xy  
(-3,5); (3,5) **\dir{-}?(1)*\dir{>};
(-3,3); (3,3) **\dir{-}?(0)*\dir{<};
(-3,-5); (3,-5) **\dir{-}?(0)*\dir{<};
(0,-1)*{\smallYri};
\endxy \bullet \rho \right \rrbracket_s 
[(m')^2-m'+a]\{(m')^2-m'+a\}[m'+b]\{2(m'+b)\}[l]\{l\}[2]\{4\}
\]
with $a\leq m^2$ and $b \leq m$. Any web $V$ in this complex is supported in 
\[
A_{1/M}[(m')^2+a+b+l+2]\{(m')^2-m'+a+2b+r_V-r_{W'}+4\}
\]
and since $(m')^2+a+b+l+2\leq M$ this is contained in $A_{1/M}\{r_V-r_{W'}\}$.

We now consider 
$\left \llbracket W' \bullet
\xy  
(-3,5); (3,5) **\dir{-}?(1)*\dir{>};
(-3,3); (3,3) **\dir{-}?(0)*\dir{<};
(-3,-5); (3,-5) **\dir{-}?(0)*\dir{<};
(0,-1)*{\smallYri};
\endxy \bullet \rho \bullet
\xy
(-3,5); (3,5) **\dir{-}?(1)*\dir{>};
(3,3); (-3,3) **\dir{-}?(1)*\dir{>};
(3,-5); (-3,-5) **\dir{-}?(1)*\dir{>};
(0,-1)*{\crossleft};
\endxy \
\right \rrbracket_s$. The smooth resolution of 
\[W' \bullet
\xy  
(-3,5); (3,5) **\dir{-}?(1)*\dir{>};
(-3,3); (3,3) **\dir{-}?(0)*\dir{<};
(-3,-5); (3,-5) **\dir{-}?(0)*\dir{<};
(0,-1)*{\smallYri};
\endxy \bullet \rho \bullet
\xy
(-3,5); (3,5) **\dir{-}?(1)*\dir{>};
(3,3); (-3,3) **\dir{-}?(1)*\dir{>};
(3,-5); (-3,-5) **\dir{-}?(1)*\dir{>};
(0,-1)*{\crossleft};
\endxy
\] 
takes the form 
$
W' \bullet
\xy  
(-3,5); (3,5) **\dir{-}?(1)*\dir{>};
(-3,3); (3,3) **\dir{-}?(0)*\dir{<};
(-3,-5); (3,-5) **\dir{-}?(0)*\dir{<};
(0,-1)*{\smallYri};
\endxy \bullet 
\xy
(-7,5); (7,5) **\dir{-}?(1)*\dir{>};
(12,5)*{_{m-m'}};
(-7,3); (-7,-3) **\crv{(0,0)}?(1)*\dir{>};
(0,0)*{_{m'}};
(7,-3); (7,3) **\crv{(0,0)}?(1)*\dir{>};
(7,-5); (-7,-5) **\dir{-}?(1)*\dir{>};
(12,-5)*{_{n-m'}};
\endxy
$
which again has no closed components. Due to the alignment of the negative 
crossing, the $0$-resolution has the same number of internal faces as the case 
where it is not present. It follows that the only difference from the previous 
case is that we make one less good $R2$ move. It follows that any web $V$ in the 
contribution to 
$\left \llbracket W' \bullet
\xy  
(-3,5); (3,5) **\dir{-}?(1)*\dir{>};
(-3,3); (3,3) **\dir{-}?(0)*\dir{<};
(-3,-5); (3,-5) **\dir{-}?(0)*\dir{<};
(0,-1)*{\smallYri};
\endxy \bullet \smalltwistmn{m}{n}{} \right \rrbracket_s$ 
is supported in 
\[
A_{1/M}[(m')^2+a+b+l+1]\{(m')^2-m'+a+2b+r_V-r_{W'}+3\}
\]
which is contained in $A_{1/M}\{r_V-r_{W'}\}$ and has minimal homological 
degree $2$ (since we necessarily have $l\geq 1$).

\noindent
\underline{\textbf{Case 3:}}
$\left \llbracket
W' \bullet
\xy  
(-3,5); (3,5) **\dir{-}?(1)*\dir{>};
(-3,-3); (3,-3) **\dir{-}?(1)*\dir{>};
(-3,-5); (3,-5) **\dir{-}?(0)*\dir{<};
(0,1)*{\smallYro};
\endxy\bullet \smalltwistmn{m}{n}{}
\right \rrbracket_s$ \newline

This case can be handled completely analogously to case $2$ above. 
In some detail, begin by rotating the diagram
$W' \bullet
\xy  
(-3,5); (3,5) **\dir{-}?(1)*\dir{>};
(-3,-3); (3,-3) **\dir{-}?(1)*\dir{>};
(-3,-5); (3,-5) **\dir{-}?(0)*\dir{<};
(0,1)*{\smallYro};
\endxy\bullet \smalltwistmn{m}{n}{}$
$180^{\circ}$ about a horizontal axis and reversing the direction of all strands. 
We are then in case $2$ (with $m$ and $n$ switched). Apply the above analysis 
(noting that $M$ is symmetric in $m$ and $n$) then rotate every web appearing in the complexes 
$180^{\circ}$ about a horizontal axis and reverse the direction of all strands.
\end{proof}

Now, let $\tilde{P}_w$ be the limit of $\mathbf{T}_w$ lying in $\K^{\angle}(\mathcal{F})$ constructed 
using the complexes $C_k^{\cdot}$ from the above proof.

\begin{prop}
Let $w=(+\cdots+-\cdots-)$. The web $\Id_w$ appears only once in $\tilde{P}_w$ and does so in 
quantum and homological degree zero; all other webs in the complex take the form \eqref{goodwebform2}. 
If $\wt(v) < \wt(w)$ then $\tilde{P}_w\bullet W_1 \simeq 0$ for any 
$W_1 \in \Hom_{\bullet}(w,v)$ and $W_2\bullet \tilde{P}_w \simeq 0$ for any 
$W_2 \in \Hom_{\bullet}(v,w)$. Finally, $\tilde{P}_w \bullet \tilde{P}_w \simeq \tilde{P}_w$.
\end{prop}

\begin{proof}
The first statement follows as in the proof of Proposition \ref{plus1} from the description of the 
complexes $C_k^{\cdot}$ in the proof of Proposition \ref{hardproof}. The third statement follows from 
the first and second as in the proof of Proposition \ref{plus3}.

It hence suffices to prove the second statement. Pulling the strands lining up with a $U$-web through 
the twist and using the homotopy equivalence
\begin{equation}\label{capandtwist}
\left \llbracket \
\xy 0;/r.1pc/:
\htwist~{(-10,5)}{(0,5)}{(-10,-5)}{(0,-5)};
\htwist~{(0,5)}{(10,5)}{(0,-5)}{(10,-5)};
(10,5); (10,-5) **\crv{(17,5)&(17,-5)};
{\ar (-10,-5)*{};(-12,-5)*{}};
\endxy \
\right \rrbracket_s \simeq
\left \llbracket \
\xy 0;/r.1pc/:
(0,5); (0,-5) **\crv{(10,5)&(10,-5)};
(.1,-4.985); (0,-5) **\dir{}?(1)*\dir{>};
\endxy \
\right \rrbracket_s [2]\{6\}
\end{equation}
we see that 
$
\abs{
\left \llbracket
\smalltwistmn{m}{n}{k}
\bullet 
\xy 0;/r.15pc/: 
(-3,5); (3,5) **\dir{-}?(1)*\dir{>};
(-3,3); (-3,-3) **\crv{(3,3)&(3,-3)};
(-2.9,-2.975); (-3,-3) **\dir{}?(1)*\dir{>};
(-3,-5); (3,-5) **\dir{-}?(0)*\dir{<};
\endxy
\right \rrbracket_s}_h \geq 2k
$. A similar analysis using 
Lemma \ref{untwist} shows the same for the complexes 
$
\left \llbracket
\smalltwistmn{m}{n}{k}
\bullet 
\xy  
(-3,5); (3,5) **\dir{-}?(1)*\dir{>};
(-3,-3); (3,-3) **\dir{-}?(1)*\dir{>};
(-3,-5); (3,-5) **\dir{-}?(0)*\dir{<};
(0,1)*{\smallYli};
\endxy \ \right \rrbracket_s
$
and 
$
\left \llbracket
\smalltwistmn{m}{n}{k}
\bullet 
\xy  
(-3,5); (3,5) **\dir{-}?(1)*\dir{>};
(-3,3); (3,3) **\dir{-}?(0)*\dir{<};
(-3,-5); (3,-5) **\dir{-}?(0)*\dir{<};
(0,-1)*{\smallYlo};
\endxy \ \right \rrbracket_s
$. Using Proposition \ref{Kuprop} and following the proof of Proposition \ref{plus2}, 
we have that $\tilde{P}_w \bullet W_1 \simeq 0$. The result for $W_2$ follows 
similarly.
\end{proof}
This gives Theorem \ref{1} for the case $w=(+\cdots+-\cdots-)$. 
Since we have already seen that $\tilde{P}_w$ lies in $\K^{\angle}(\mathcal{F})$, Theorem 
\ref{2} follows as in the case $w=(+\cdots+)$.

The methods used to show that $\tilde{P}_w$ is supported in $\K^{\angle}(\mathcal{F})$ 
can be employed to simplify their computation. 
We exhibit this in the following computation of $\tilde{P}_{(+-)}$.

\begin{prop}\label{catex2}
\begin{equation}\label{exeq}
\tilde{P}_{(+-)} =
\xymatrix{
\xy
(0,1)*{\xy
(-2,2)*{}="a";
(2,2)*{}="b";
(-2,-2)*{}="c";
(2,-2)*{}="d";
"a";"b" **\crv{(0,1)};
"d";"c" **\crv{(0,-1)};
(1.9,1.97);(2,2)**\dir{-}?(1)*\dir{>};
(-1.9,-1.97);(-2,-2)**\dir{-}?(1)*\dir{>};
\endxy};
\endxy \ar[r]^-{s}
& q^2 \ \xy
(-2,2)*{}="a";
(2,2)*{}="b";
(-2,-2)*{}="c";
(2,-2)*{}="d";
"a";"c" **\crv{-};
"d";"b" **\crv{-};
(1.9,1.85);(2,2)**\dir{-}?(1)*\dir{>};
(-1.9,-1.87);(-2,-2)**\dir{-}?(1)*\dir{>};
\endxy \ar[r]^-{ct_{-}}
& q^4 \ \xy
(-2,2)*{}="a";
(2,2)*{}="b";
(-2,-2)*{}="c";
(2,-2)*{}="d";
"a";"c" **\crv{-};
"d";"b" **\crv{-};
(1.9,1.85);(2,2)**\dir{-}?(1)*\dir{>};
(-1.9,-1.87);(-2,-2)**\dir{-}?(1)*\dir{>};
\endxy \ar[rr]^-{\frac{1}{3}ct_B - T_{+}}
& & q^8 \ \xy
(-2,2)*{}="a";
(2,2)*{}="b";
(-2,-2)*{}="c";
(2,-2)*{}="d";
"a";"c" **\crv{-};
"d";"b" **\crv{-};
(1.9,1.85);(2,2)**\dir{-}?(1)*\dir{>};
(-1.9,-1.87);(-2,-2)**\dir{-}?(1)*\dir{>};
\endxy \ar[r]^-{ct_{-}}
& \cdots
}
\end{equation}
where $ct_B = ct_L \circ ct_R$ with $ct_{-}$, $ct_L$ and $ct_R$ as in Lemma \ref{m=1=n}. $T_{+} = T_L + T_R$ where 
$T_L$ is the identity foam on the right arc and is the 
foam:
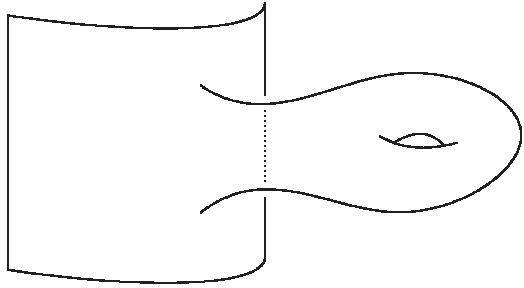
\[
\xy 0;/r.2pc/:
(-20,18); (20,20) **\crv{(0,15)&(20,15)};
(-20,-22); (-20,18) **\dir{-};
(-20,-22); (20,-20) **\crv{(0,-25)&(20,-25)};
(10,7); (10,-13) **\crv{(20,0)&(40,15)&(70,0)&(40,-20)&(20,-5)};
(20,4.2)*{\hole}="h1";
(20,-9.3)*{\hole}="h2";
(20,20); "h1" **\dir{-};
"h1"; "h2" **\dir{.};
"h2"; (20,-20) **\dir{-};
(38,-1); (50,-2) **\crv{(43,-4)}\POS?(.2)="a" \POS?(.85)="b";
"a"; "b" **\crv{(45,1)};
\endxy
\]
on the left arc; $T_R$ is defined similarly.
\end{prop}
\begin{proof}
Lemma \ref{m=1=n} gives that 
\[
\left \llbracket \ \smalltwotwistoneone \ \right\rrbracket_s \simeq 
\cone\left(
\left \llbracket \
\xy
(-2,2)*{}="a";
(2,2)*{}="b";
(-2,-2)*{}="c";
(2,-2)*{}="d";
"a";"b" **\crv{(0,1)};
"d";"c" **\crv{(0,-1)};
(1.9,1.97);(2,2)**\dir{-}?(1)*\dir{>};
(-1.9,-1.97);(-2,-2)**\dir{-}?(1)*\dir{>};
\endxy \ \right \rrbracket \overset{s}{\longrightarrow} 
\left(
\xymatrix{q^2 \xy
(-2,2)*{}="a";
(2,2)*{}="b";
(-2,-2)*{}="c";
(2,-2)*{}="d";
"a";"c" **\crv{-};
"d";"b" **\crv{-};
(1.9,1.85);(2,2)**\dir{-}?(1)*\dir{>};
(-1.9,-1.87);(-2,-2)**\dir{-}?(1)*\dir{>};
\endxy \ar[r]^-{-ct_{-}}
& q^4 \xy
(-2,2)*{}="a";
(2,2)*{}="b";
(-2,-2)*{}="c";
(2,-2)*{}="d";
"a";"c" **\crv{-};
"d";"b" **\crv{-};
(1.9,1.85);(2,2)**\dir{-}?(1)*\dir{>};
(-1.9,-1.87);(-2,-2)**\dir{-}?(1)*\dir{>};
\endxy}
\right)
\right)[1]
\]
from which we compute that
$
\left \llbracket \ \smalltwotwistoneone \ \right\rrbracket_s \bullet 
\left \llbracket \ \smalltwotwistoneone \ \right\rrbracket_s 
$
is homotopy equivalent to
\[
\cone\left(
\left \llbracket \
\xy
(-2,2)*{}="a";
(2,2)*{}="b";
(-2,-2)*{}="c";
(2,-2)*{}="d";
"a";"b" **\crv{(0,1)};
"d";"c" **\crv{(0,-1)};
(1.9,1.97);(2,2)**\dir{-}?(1)*\dir{>};
(-1.9,-1.97);(-2,-2)**\dir{-}?(1)*\dir{>};
\endxy \ \right \rrbracket \overset{s}{\longrightarrow} 
\left(
\xymatrix{q^2 \xy
(-2,2)*{}="a";
(2,2)*{}="b";
(-2,-2)*{}="c";
(2,-2)*{}="d";
"a";"c" **\crv{-};
"d";"b" **\crv{-};
(1.9,1.85);(2,2)**\dir{-}?(1)*\dir{>};
(-1.9,-1.87);(-2,-2)**\dir{-}?(1)*\dir{>};
\endxy \ar[r]^-{-ct_{-}}
& q^4 \xy
(-2,2)*{}="a";
(2,2)*{}="b";
(-2,-2)*{}="c";
(2,-2)*{}="d";
"a";"c" **\crv{-};
"d";"b" **\crv{-};
(1.9,1.85);(2,2)**\dir{-}?(1)*\dir{>};
(-1.9,-1.87);(-2,-2)**\dir{-}?(1)*\dir{>};
\endxy}
\right)
\right)[1] \bullet \left \llbracket \ \smalltwotwistoneone \ \right\rrbracket_s
\]
\[
\simeq 
\cone\left(
\left \llbracket \ \smalltwotwistoneone \ \right\rrbracket_s \overset{f'}{\longrightarrow} 
\left(
\xymatrix{q^2 \xy
(-2,2)*{}="a";
(2,2)*{}="b";
(-2,-2)*{}="c";
(2,-2)*{}="d";
"a";"c" **\crv{-};
"d";"b" **\crv{-};
(1.9,1.85);(2,2)**\dir{-}?(1)*\dir{>};
(-1.9,-1.87);(-2,-2)**\dir{-}?(1)*\dir{>};
\endxy \ar[r]^-{-ct_{-}}
& q^4 \xy
(-2,2)*{}="a";
(2,2)*{}="b";
(-2,-2)*{}="c";
(2,-2)*{}="d";
"a";"c" **\crv{-};
"d";"b" **\crv{-};
(1.9,1.85);(2,2)**\dir{-}?(1)*\dir{>};
(-1.9,-1.87);(-2,-2)**\dir{-}?(1)*\dir{>};
\endxy}
\right)[2]\{6\}
\right)[1]
\]
\[= \xymatrix{
\xy
(0,1)*{\xy
(-2,2)*{}="a";
(2,2)*{}="b";
(-2,-2)*{}="c";
(2,-2)*{}="d";
"a";"b" **\crv{(0,1)};
"d";"c" **\crv{(0,-1)};
(1.9,1.97);(2,2)**\dir{-}?(1)*\dir{>};
(-1.9,-1.97);(-2,-2)**\dir{-}?(1)*\dir{>};
\endxy};
\endxy \ar[r]^-{s}
& q^2 \xy
(-2,2)*{}="a";
(2,2)*{}="b";
(-2,-2)*{}="c";
(2,-2)*{}="d";
"a";"c" **\crv{-};
"d";"b" **\crv{-};
(1.9,1.85);(2,2)**\dir{-}?(1)*\dir{>};
(-1.9,-1.87);(-2,-2)**\dir{-}?(1)*\dir{>};
\endxy \ar[r]^-{ct_{-}}
& q^4 \xy
(-2,2)*{}="a";
(2,2)*{}="b";
(-2,-2)*{}="c";
(2,-2)*{}="d";
"a";"c" **\crv{-};
"d";"b" **\crv{-};
(1.9,1.85);(2,2)**\dir{-}?(1)*\dir{>};
(-1.9,-1.87);(-2,-2)**\dir{-}?(1)*\dir{>};
\endxy \ar[r]^-{f}
& q^8 \xy
(-2,2)*{}="a";
(2,2)*{}="b";
(-2,-2)*{}="c";
(2,-2)*{}="d";
"a";"c" **\crv{-};
"d";"b" **\crv{-};
(1.9,1.85);(2,2)**\dir{-}?(1)*\dir{>};
(-1.9,-1.87);(-2,-2)**\dir{-}?(1)*\dir{>};
\endxy \ar[r]^-{ct_{-}}
& q^{10} \xy
(-2,2)*{}="a";
(2,2)*{}="b";
(-2,-2)*{}="c";
(2,-2)*{}="d";
"a";"c" **\crv{-};
"d";"b" **\crv{-};
(1.9,1.85);(2,2)**\dir{-}?(1)*\dir{>};
(-1.9,-1.87);(-2,-2)**\dir{-}?(1)*\dir{>};
\endxy
}.
\]
A direct (and tedious!) computation shows that $f=\frac{1}{3}ct_B - T_{+}$, although it can be argued based on 
degree that $f$ must be a multiple of this map. Repeating this procedure inductively to compute 
$
\left \llbracket \ 
\xy
(0,0)*{\smalltwotwistoneone};
(0,4)*_{_{k-1}};
\endxy 
\ \right\rrbracket_s \bullet 
\left \llbracket \ \smalltwotwistoneone \ \right\rrbracket_s
$
shows that 
$
\left \llbracket \ 
\xy
(0,0)*{\smalltwotwistoneone};
(0,4)*_{_{k}};
\endxy 
\ \right\rrbracket_s
$
is given by the complex in equation \eqref{exeq} truncated at homological degree $2k$.
\end{proof}

\subsection{$\tilde{P}_w$ and decategorification for non-segregated words}\label{scp4}

In the decategorified case, we construct the projector $P_w$ for a non-segregated word $w$ by considering 
the segregated projector of the same weight and (horizontally) composing with (a composition of) 
$H$-webs on both sides. This procedure works in the categorified setting as well. 

Indeed, suppose that $w'$ is a word for which we have constructed $\tilde{P}_{w'}$ 
satisfying the conditions of Theorem \ref{1} and 
$w$ is a word of the same weight obtained by transposing one pair of adjacent $+$ and $-$ 
signs in $w'$. Let $h$ be the web in $\Hom_{\bullet}(w',w)$ given by the tensor product of 
$ \ \xy
(-2,2)*{}="a";
(0,1)*{}="b";
(2,2)*{}="c";
(0,-1)*{}="d";
(-2,-2)*{}="e";
(2,-2)*{}="f";
"a";"b" **\dir{-};
"b";"c" **\dir{-};
"b";"d" **\dir{-};
"e";"d" **\dir{-};
"d";"f" **\dir{-};
\endxy$ (oriented appropriately) and identity webs. Consider 
$\bar{h} \bullet \tilde{P}_{w'}\bullet h$ where $\bar{h}$ is the web in $\Hom_{\bullet}(w,w')$
obtained from $h$ by reversing the orientation of the strands in $\xy
(-2,2)*{}="a";
(0,1)*{}="b";
(2,2)*{}="c";
(0,-1)*{}="d";
(-2,-2)*{}="e";
(2,-2)*{}="f";
"a";"b" **\dir{-};
"b";"c" **\dir{-};
"b";"d" **\dir{-};
"e";"d" **\dir{-};
"d";"f" **\dir{-};
\endxy$. 
If $V$ is a web in $\Hom_{\bullet}(v,w)$ with $\wt(v) < \wt(w)$ then 
\[
V\bullet \bar{h} \cong \bigoplus_{\alpha}q^{l_{\alpha}}W_{\alpha}
\]
for $W_{\alpha}$ in $\Hom_{\bullet}(v,w')$. Since $\wt(v) < \wt(w')$ we have 
that 
\begin{align*}
V \bullet \bar{h} \bullet \tilde{P}_{w'} \bullet h &\cong 
\left(\bigoplus_{\alpha} q^{l_{\alpha}}W_{\alpha}\right)\bullet \tilde{P}_{w'} \bullet h \\
&\cong \bigoplus_{\alpha} \left(q^{l_{\alpha}}W_{\alpha}\bullet \tilde{P}_{w'}\right) \bullet h \\
& \simeq 0 .
\end{align*}
A similar computation shows that $\bar{h} \bullet \tilde{P}_{w'} \bullet h \bullet V \simeq 0$ for 
$V$ in $\Hom_{\bullet}(w,v)$ with $\wt(v) < \wt(w)$.

If $h$ is as above, $h \bullet \bar{h} \cong \Id_w \oplus W$ where 
$W \bullet \tilde{P}_{w'} \simeq 0$ so we have
\begin{align*}
\left(\bar{h}\bullet \tilde{P}_{w'} \bullet h\right)\bullet \left(\bar{h}\bullet \tilde{P}_{w'} \bullet h\right) 
&\cong \bar{h}\bullet \tilde{P}_{w'} \bullet\left(\Id_w \oplus W\right) 
\bullet \tilde{P}_{w'} \bullet h \\
&\simeq \bar{h}\bullet \tilde{P}_{w'} \bullet \tilde{P}_{w'} \bullet h \\
&\simeq \bar{h}\bullet \tilde{P}_{w'} \bullet h .
\end{align*}
Since all non-identity webs appearing in $\tilde{P}_{w'}$ factor through a word of lower weight, these webs will not 
contribute an identity web to $\bar{h} \bullet \tilde{P}_{w'}\bullet h$. Noting that 
$\bar{h}\bullet h$ is the direct sum of an identity web and a web factoring through a word of lower weight, 
this shows that $\bar{h} \bullet \tilde{P}_{w'}\bullet h$ gives 
the categorified projector $\tilde{P}_w$. 

Since any word can be obtained from the segregated word of the same weight via a sequence of permutations of the symbols $+$ and $-$, 
this proves Theorem \ref{1} for arbitrary $w$. Theorem \ref{2} also follows 
since $\K^{\angle}(\mathcal{F})$ is (essentially) closed under horizontal composition and
\begin{align*}
\chi(\tilde{P}_w) &= \chi(\bar{h} \bullet \tilde{P}_{w'}\bullet h) \\
&= \bar{h} \bullet \chi(\tilde{P}_{w'})\bullet h \\
&= \bar{h} \bullet P_{w'}\bullet h \\
&= P_w.
\end{align*}

Since the above construction of $\tilde{P}_w$ for non-segregated $w$ is somewhat indirect, 
a natural question to ask is whether this projector can also be realized as the stable limit of torus braids. In fact, 
the answer is yes. To illustrate this, let $w'$ be a segregated word and suppose that $w$ is the word 
that results from switching the last $+$ in $w'$ with the first $-$. We then have 
\begin{align*}
\tilde{P}_w 
&= \xy
(-3,5); (3,5) **\dir{-}?(1)*\dir{>};
(3,-5); (-3,-5) **\dir{-}?(1)*\dir{>};
(0,0)*{\xy 0;/r.15pc/:
(0,3)*{};(-5,5)*{} **\dir{-}?(.8)*\dir{>};
(0,3)*{};(5,5)*{} **\dir{-}?(.8)*\dir{>};
(0,-3)*{};(0,3)*{} **\dir{-}?(.3)*\dir{<};
(-5,-5)*{};(0,-3)*{} **\dir{-}?(.7)*\dir{>};
(5,-5)*{};(0,-3)*{} **\dir{-}?(.7)*\dir{>};
\endxy}
\endxy \bullet \left(
\lim_{k\to\infty} \left \llbracket \ \xy (0,1)*{\smalltwistmn{}{}{k}} \endxy \right \rrbracket_s \right) \bullet
\xy
(-3,5); (3,5) **\dir{-}?(1)*\dir{>};
(3,-5); (-3,-5) **\dir{-}?(1)*\dir{>};
(0,0)*{\xy 0;/r.15pc/:
(0,-3)*{};(-5,-5)*{} **\dir{-}?(.8)*\dir{>};
(0,-3)*{};(5,-5)*{} **\dir{-}?(.8)*\dir{>};
(0,3)*{};(0,-3)*{} **\dir{-}?(.3)*\dir{<};
(-5,5)*{};(0,3)*{} **\dir{-}?(.7)*\dir{>};
(5,5)*{};(0,3)*{} **\dir{-}?(.7)*\dir{>};
\endxy}
\endxy \\
& \simeq \lim_{k\to\infty} \left \llbracket \ \xy
(-3,5); (3,5) **\dir{-}?(1)*\dir{>};
(3,-5); (-3,-5) **\dir{-}?(1)*\dir{>};
(0,0)*{\xy 0;/r.15pc/:
(0,3)*{};(-5,5)*{} **\dir{-}?(.8)*\dir{>};
(0,3)*{};(5,5)*{} **\dir{-}?(.8)*\dir{>};
(0,-3)*{};(0,3)*{} **\dir{-}?(.3)*\dir{<};
(-5,-5)*{};(0,-3)*{} **\dir{-}?(.7)*\dir{>};
(5,-5)*{};(0,-3)*{} **\dir{-}?(.7)*\dir{>};
\endxy}
\endxy \bullet
\xy (0,1)*{\smalltwistmn{}{}{k}} \endxy \bullet
\xy
(-3,5); (3,5) **\dir{-}?(1)*\dir{>};
(3,-5); (-3,-5) **\dir{-}?(1)*\dir{>};
(0,0)*{\xy 0;/r.15pc/:
(0,-3)*{};(-5,-5)*{} **\dir{-}?(.8)*\dir{>};
(0,-3)*{};(5,-5)*{} **\dir{-}?(.8)*\dir{>};
(0,3)*{};(0,-3)*{} **\dir{-}?(.3)*\dir{<};
(-5,5)*{};(0,3)*{} **\dir{-}?(.7)*\dir{>};
(5,5)*{};(0,3)*{} **\dir{-}?(.7)*\dir{>};
\endxy}
\endxy \ \right \rrbracket_s \\
& \simeq \lim_{k \to \infty} \left \llbracket \
\xy 
(-4,5); (4,5) **\dir{-}?(1)*\dir{>};
(4,-5); (-4,-5) **\dir{-}?(1)*\dir{>};
(0,0)*{\xy 0;/r.1pc/:
(-5,5)*{}="a";
(5,5)*{}="b";
(-5,-5)*{}="c";
(5,-5)*{}="d";
(-10,10)*{}="e";
(10,10)*{}="f";
(-10,-10)*{}="g";
(10,-10)*{}="h";
"a";"b" **\dir{-}?(.7)*\dir{>};
"d";"b" **\dir{-}?(.7)*\dir{>};
"a";"c" **\dir{-}?(.7)*\dir{>};
"d";"c" **\dir{-}?(.7)*\dir{>};
"a";"e" **\dir{-}?(.7)*\dir{>};
"f";"b" **\dir{-}?(.7)*\dir{>};
"g";"c" **\dir{-}?(.7)*\dir{>};
"d";"h" **\dir{-}?(.7)*\dir{>};
\endxy}
\endxy \bullet 
\xy (0,1)*{\smalltwist{k}}; (6,0)*{_{w}} \endxy \right \rrbracket_s \\
& \simeq \lim_{k \to \infty} \left \llbracket 
\xy (0,1)*{\smalltwist{k}}; (6,0)*{_{w}} \endxy \right \rrbracket_s \oplus
\lim_{k \to \infty} \left \llbracket \
\xy 
(0,0)*{
\xy 0;/r.15pc/:
(-5,5); (5,5) **\dir{-}?(1)*\dir{>};
(-5,3); (-5,-3) **\crv{(0,3)&(0,-3)}?(0)*\dir{<};
(5,-3); (5,3) **\crv{(0,-3)&(0,3)}?(0)*\dir{<};
(5,-5); (-5,-5) **\dir{-}?(0)*\dir{<};
\endxy}
\endxy \bullet 
\xy (0,1)*{\smalltwist{k}}; (6,0)*{_{w}} \endxy \right \rrbracket_s \\
& \simeq \lim_{k \to \infty} \left \llbracket 
\xy (0,1)*{\smalltwist{k}}; (6,0)*{_{w}} \endxy \right \rrbracket_s
\end{align*}
where $\xy (0,1)*{\smalltwist{k}}; (6,0)*{_{w}} \endxy$ denotes $k$ full twists on strands oriented according 
to $w$. The fact that $\displaystyle \lim_{k \to \infty} \left \llbracket \
\xy 0;/r.15pc/:
(-5,5); (5,5) **\dir{-}?(1)*\dir{>};
(-5,3); (-5,-3) **\crv{(0,3)&(0,-3)}?(0)*\dir{<};
(5,-3); (5,3) **\crv{(0,-3)&(0,3)}?(0)*\dir{<};
(5,-5); (-5,-5) **\dir{-}?(0)*\dir{<};
\endxy \bullet 
\xy (0,1)*{\smalltwist{k}}; (6,0)*{_{w}} \endxy \right \rrbracket_s$
is null-homotopic follows from equation \eqref{capandtwist} and Lemma \ref{contractiblelem}.
Repeating this argument to switch all desired $+$'s and $-$'s gives the general result.

\section{The categorified $\mathfrak{sl}_3$ Reshetikhin-Turaev invariant of tangles}\label{sectioninvt}

We now use the categorified projectors $\tilde{P}_w$ to give a categorification of the $\mathfrak{sl}_3$ 
Reshetikhin-Turaev invariant of framed tangles. Recall that this invariant assigns an element of $\C(q)$ for each 
labeling of the components of the tangle by irreducible representations. 
We now describe a combinatorial method, given in \cite{Kuperberg}, for computing this 
invariant. For each component of the tangle, consider any\footnote{The apparent dependence on the choice of word is immaterial; 
the invariant is independent of the choice of words labeling closed components and the invariants obtained by labeling components 
with boundary by different words of the same weight are isomorphic.}
word which specifies the highest weight of the 
representation. Cable the component according to the framing with the number and direction of strands given by the word. 
Insert the corresponding projector anywhere along the 
cabling and then evaluate the resulting webs in the $\mathfrak{sl}_3$ spider. We define the categorified invariant 
analogously.

\begin{defn}\label{invt}
The categorified $\mathfrak{sl}_3$ Reshetikhin-Turaev invariant of a framed tangle $T$ with 
$i^{th}$ component labeled by the irreducible representation corresponding to the word $w_i$, denoted 
$\left \llbracket T \right \rrbracket_{(w_1,\ldots,w_r)}$, is computed by cabling each component according 
to the framing with strands directed according to the corresponding word, 
inserting the categorified projector, and evaluating to obtain a complex in $\K^{\angle}(\mathcal{F})$.
\end{defn}

To prove that this defines an invariant of framed tangles 
it suffices to show that the resulting complex is invariant up to homotopy under $R2$ and $R3$ Reidemeister moves and 
under choice of where the projector is inserted. We first establish some diagrammatic notation. Let 
\[
\left \llbracket \xy (0,0)*{\smallproj}; (6,0)*{_{w}}; \endxy  \right \rrbracket 
:= \tilde{P}_w;
\]
this notation will prove useful when considering the horizontal composition of categorified projectors with complexes 
assigned to tangles. 

The following result is the analog in our setting of Lemma $5.2$ from \cite{CooperKrushkal}.

\begin{lem}
Let $w$ be a word. We have 
\[
\left \llbracket
\xy
(0,5); (0,-5) **\dir{-};
(-5,0)*{\smallproj};
(1,0); (6,0) **\dir{-};
(8,0)*{_{w}};
\endxy \right \rrbracket \simeq
\left \llbracket
\xy
(0,5); (0,-5) **\dir{-};
(5,0)*{\smallproj};
(-5,0)*{\smallproj};
(11,0)*{_{w}};
\endxy \right \rrbracket \simeq 
\left \llbracket
\xy
(0,5); (0,-5) **\dir{-};
(5,0)*{\smallproj};
(-1,0); (-6,0) **\dir{-};
(11,0)*{_{w}};
\endxy 
\right \rrbracket
\]
where the vertical strand can be oriented in either direction. A similar result holds for sliding 
a categorified projector over a strand.
\end{lem}
\begin{proof}
It suffices to show the first homotopy equivalence.
We have (dropping the word specifying the projector)
\[
\left \llbracket
\xy
(0,5); (0,-5) **\dir{-};
(5,0)*{\smallproj};
(-5,0)*{\smallproj};
\endxy \right \rrbracket = 
\left \llbracket
\xy
(0,5); (0,-5) **\dir{-};
(-5,0)*{\smallproj};
(1,0); (6,0) **\dir{-};
\endxy \right \rrbracket \bullet 
\left(
\xymatrix{C^0 \ar[r] & C^1 \ar[r] & \cdots}
\right)
\]
where $C^i$ is a ($q$-linear) direct sum of webs annihilated by $\tilde{P}_w$ for $i>0$ and $C^0$ is 
the direct sum of an identity web with webs annihilated by $\tilde{P}_w$. If $V$ is one such (non-identity) web, equation 
\eqref{catknottedspider} gives
\begin{align*}
\left \llbracket
\xy
(0,5); (0,-5) **\dir{-};
(-5,0)*{\smallproj};
(1,0); (4,0) **\dir{-};
(7.5,0)*{\bullet \ V};
\endxy \right \rrbracket &\simeq 
\left \llbracket
\xy
(0,5); (0,-5) **\dir{-};
(-18,0)*{\smallproj};
(1,0); (4,0) **\dir{-};
(-9,0)*{\bullet \ V \ \bullet};
(-3,0); (-1,0) **\dir{-};
\endxy \right \rrbracket\\
&\simeq 0.
\end{align*}
It then follows, using Propositions \ref{replacement} and \ref{contractiblelem}, that
\begin{align*}
\left \llbracket
\xy
(0,5); (0,-5) **\dir{-};
(5,0)*{\smallproj};
(-5,0)*{\smallproj};
\endxy \right \rrbracket &\simeq 
\left \llbracket
\xy
(0,5); (0,-5) **\dir{-};
(-5,0)*{\smallproj};
(1,0); (6,0) **\dir{-};
\endxy  \bullet C^0 \right \rrbracket \\
&\simeq \left \llbracket
\xy
(0,5); (0,-5) **\dir{-};
(-5,0)*{\smallproj};
(1,0); (6,0) **\dir{-};
\endxy \right \rrbracket.
\end{align*}
\end{proof}

\begin{prop}\label{proplast}
The complex assigned to a labeled, framed tangle according to Definition \ref{invt} is invariant up to homotopy under 
$R2$ and $R3$ Reidemeister moves and under choice of where along a component the categorified projector is inserted.
\end{prop}
\begin{proof}
The preceding lemma shows that the categorified projector can be slid along a component to any desired location without 
changing the complex up to homotopy. Invariance under $R2$ and $R3$ Reidemeister moves follows since we can assume 
that the projector is not located in the region of the knot diagram where the moves take place.
\end{proof}
Theorem \ref{3} follows from Proposition \ref{proplast}, 
Theorem \ref{2}, and the similarities in the definitions of the categorified and decategorified invariants.

We conclude with some explicit computations of this invariant.

\begin{ex}\label{ex++}
Using Example \ref{catex1} we find that
\[
\left\llbracket \ \
\xy
(8,8)*{_{0}};
(-8,0)*{}="a";
(8,0)*{}="b";
"a";"b" **\crv{(-8,-10)&(8,-10)};
"a";"b" **\crv{(-8,10)&(8,10)}?(1)*\dir{>};
\endxy \ \
\right\rrbracket_{(++)} = 
\left \llbracket \ \
\xy
(-10,-1)*{}="a";
(-10,1)*{}="b";
(-2,-1)*{}="c";
(-2,1)*{}="d";
"a";"b" **\dir{-};
"b";"d" **\dir{-};
"d";"c" **\dir{-};
"c";"a" **\dir{-};
(-8,-1)*{}="e";
(-8,1)*{}="f";
(-4,-1)*{}="g";
(-4,1)*{}="h";
(8,0)*{}="p";
(4,0)*{}="q";
"e";"p" **\crv{(-8,-10)&(8,-10)};
"f";"p" **\crv{(-8,10)&(8,10)}?(1)*\dir{>};
"g";"q" **\crv{(-4,-6)&(4,-6)};
"h";"q" **\crv{(-4,6)&(4,6)}?(1)*\dir{>};
\endxy \ \
\right \rrbracket \\
\]
\[
\ \ \ = \xymatrix{
\xy
(0,0)*\xycircle(2,2){-};
(0,0)*\xycircle(4,4){-};
(-2,.5); (-2,0.51) **\dir{-}?(1)*\dir{>};
(-4,.5); (-4,0.51) **\dir{-}?(1)*\dir{>};
\endxy \ar[r]^-{z}
&
q \ \xy 0;/r.2pc/:
(-3,2); (-3,-2) **\dir{-}?(.7)*\dir{>};
(-3,2); (-3,-2) **\crv{(0,5)&(7,0)&(0,-5)};
(-3,2); (-3,-2) **\crv{(-7,5)&(3,8)&(9,0)&(3,-8)&(-7,-5)};
(3.4,-1); (3.4,-1.1) **\dir{-}?(1)*\dir{>};
(7,-1); (7,-1.1) **\dir{-}?(1)*\dir{>};
\endxy \ar[r]^-{0}
&
q^3 \ \xy 0;/r.2pc/:
(-3,2); (-3,-2) **\dir{-}?(.7)*\dir{>};
(-3,2); (-3,-2) **\crv{(0,5)&(7,0)&(0,-5)};
(-3,2); (-3,-2) **\crv{(-7,5)&(3,8)&(9,0)&(3,-8)&(-7,-5)};
(3.4,-1); (3.4,-1.1) **\dir{-}?(1)*\dir{>};
(7,-1); (7,-1.1) **\dir{-}?(1)*\dir{>};
\endxy \ar[r]^-{p}
&
q^5 \ \xy 0;/r.2pc/:
(-3,2); (-3,-2) **\dir{-}?(.7)*\dir{>};
(-3,2); (-3,-2) **\crv{(0,5)&(7,0)&(0,-5)};
(-3,2); (-3,-2) **\crv{(-7,5)&(3,8)&(9,0)&(3,-8)&(-7,-5)};
(3.4,-1); (3.4,-1.1) **\dir{-}?(1)*\dir{>};
(7,-1); (7,-1.1) **\dir{-}?(1)*\dir{>};
\endxy \ar[r]^-{0}
& \cdots
}
\]
where $p$ is the foam which zips and then unzips along the two downward arcs. Using Gaussian elimination
(and a somewhat involved foam calculation), we find that 
this complex is homotopy equivalent to the complex 
\[
\xymatrix{
q^{-2} \xy
(0,0)*\xycircle(2,2){-};
\endxy \ar[r]
& 0 \ar[r]
& q^2
\xy
(0,0)*\xycircle(2,2){-};
\endxy \ar[r]^-{T}
& q^6
\xy
(0,0)*\xycircle(2,2){-};
\endxy \ar[r]^-{0}
& q^6
\xy
(0,0)*\xycircle(2,2){-};
\endxy \ar[r]^-{T}
& q^{10}
\xy
(0,0)*\xycircle(2,2){-};
\endxy \ar[r]^-{0}
&\cdots}
\]
where 
\[
T \ = \ 
\xy 0;/r.15pc/:
(0,20)*\xycircle(20,5){-};
(-20,-20); (-20,20) **\dir{-};
(0,-10)*\ellipse(20,5){.}; 
(0,-10)*\ellipse(20,5)__,=:a(-180){-};
(10,7); (10,-13) **\crv{(20,0)&(40,15)&(70,0)&(40,-20)&(20,-5)};
(20,4.2)*{\hole}="h1";
(20,-9.3)*{\hole}="h2";
(20,20); "h1" **\dir{-};
"h1"; "h2" **\dir{.};
"h2"; (20,-20) **\dir{-};
(38,-1); (50,-2) **\crv{(43,-4)}\POS?(.2)="a" \POS?(.85)="b";
"a"; "b" **\crv{(45,1)};
\endxy
\]
and we have omitted the orientation of the circles.
Applying the functor $\widehat{\Hom}(\emptyset,-)$ which assigns the graded vector space of (not necessarily degree-zero) 
foams from $\emptyset$ to closed webs, we obtain a complex of graded vector spaces. We have that 
\[
\widehat{\Hom}(\emptyset, \xy (0,0)*\xycircle(2,2){-} \endxy \ ) = 
\C \ \diskv \oplus \ \C \ \ctv \oplus \ \C \ \tv 
\]
in gradings $2$, $0$, and $-2$ respectively (see \cite{MorrisonNieh}) and the map 
$\widehat{\Hom}(\emptyset,T)$ has rank $1$, giving the cohomology as 
\[
\mathcal{H}^{i,j}\left(
\left\llbracket \ \
\xy 0;/r.15pc/:
(8,8)*{_{0}};
(-8,0)*{}="a";
(8,0)*{}="b";
"a";"b" **\crv{(-8,-10)&(8,-10)};
"a";"b" **\crv{(-8,10)&(8,10)}?(1)*\dir{>};
\endxy \ \
\right\rrbracket_{(++)}
\right) = 
\left\{
\begin{array}{ll}
       \C & i=0 \text{ and } j= -4,-2,0 \\
       \C & i=2k  \text{ and } j= 4k-4, 4k-2 \text{ for } k>0  \\
       \C & i=2k+1  \text{ and } j= 4k+2, 4k+4 \text{ for } k>0  \\
        0 & \text{else }.
     \end{array} 
\right.
\]
In the above formula, $i$ denotes homological degree while $j$ denotes the
vector space grading.
\end{ex}

\begin{ex} Using Proposition \ref{catex2} we compute
\[
\left\llbracket \ \
\xy
(8,8)*{_{0}};
(-8,0)*{}="a";
(8,0)*{}="b";
"a";"b" **\crv{(-8,-10)&(8,-10)};
"a";"b" **\crv{(-8,10)&(8,10)}?(1)*\dir{>};
\endxy \ \
\right\rrbracket_{(+-)} = 
\left \llbracket \ \
\xy
(-10,-1)*{}="a";
(-10,1)*{}="b";
(-2,-1)*{}="c";
(-2,1)*{}="d";
"a";"b" **\dir{-};
"b";"d" **\dir{-};
"d";"c" **\dir{-};
"c";"a" **\dir{-};
(-8,-1)*{}="e";
(-8,1)*{}="f";
(-4,-1)*{}="g";
(-4,1)*{}="h";
(8,0)*{}="p";
(4,0)*{}="q";
"e";"p" **\crv{(-8,-10)&(8,-10)};
"f";"p" **\crv{(-8,10)&(8,10)}?(1)*\dir{>};
"g";"q" **\crv{(-4,-6)&(4,-6)};
"h";"q" **\crv{(-4,6)&(4,6)}?(1)*\dir{<};
\endxy \ \
\right \rrbracket \\
\]
\[
\ \ \ = \xymatrix{
\xy
(0,0)*\xycircle(2,2){-};
(0,0)*\xycircle(4,4){-};
\endxy \ar[r]^-{s}
& q^2 \ \xy
(0,0)*\xycircle(3,3){-};
\endxy \ar[r]^-{0}
& q^4 \ \xy
(0,0)*\xycircle(3,3){-};
\endxy \ar[r]^-{-3T}
& q^8 \ \xy
(0,0)*\xycircle(3,3){-};
\endxy \ar[r]^-{0}
& q^{10} \ \xy
(0,0)*\xycircle(3,3){-};
\endxy \ar[r]^-{-3T}
& \cdots
}
\]
\[
\simeq \xymatrix{
q^{-2} \xy
(0,0)*\xycircle(2,2){-};
\endxy \oplus 
\xy
(0,0)*\xycircle(2,2){-};
\endxy \ar[r]
& 0 \ar[r]
& q^4
\xy
(0,0)*\xycircle(2,2){-};
\endxy \ar[r]^-{T}
& q^8
\xy
(0,0)*\xycircle(2,2){-};
\endxy \ar[r]^-{0}
& q^{10}
\xy
(0,0)*\xycircle(2,2){-};
\endxy \ar[r]^-{T}
& q^{14}
\xy
(0,0)*\xycircle(2,2){-};
\endxy \ar[r]^-{0}
&\cdots}
\]
where again we have omitted the orientation of the circles. Applying 
$\widehat{\Hom}(\emptyset,-)$ and taking cohomology gives
\[
\mathcal{H}^{i,j}\left(
\left\llbracket \ \
\xy 0;/r.15pc/:
(8,8)*{_{0}};
(-8,0)*{}="a";
(8,0)*{}="b";
"a";"b" **\crv{(-8,-10)&(8,-10)};
"a";"b" **\crv{(-8,10)&(8,10)}?(1)*\dir{>};
\endxy \ \
\right\rrbracket_{(+-)}
\right) = 
\left\{
\begin{array}{ll}
       \C^2 & i=0 \text{ and } j=-2,0\\
       \C & i=0 \text{ and } j= -4,2 \\
       \C & i=2k  \text{ and } j= 6k-4, 6k-2 \text{ for } k>0  \\
       \C & i=2k+1  \text{ and } j= 6k+2, 6k+4 \text{ for } k>0  \\
        0 & \text{else }.
     \end{array} 
\right.
\]

\end{ex}

\bibliographystyle{QT}

\end{document}